\documentclass{scrartcl}

\usepackage{amsfonts}
\usepackage{amsthm}
\usepackage{cite}
\usepackage{fouridx}
\usepackage[T1]{fontenc}
\usepackage{hyperref}
\usepackage{lmodern}
\usepackage{mathscinet}
\usepackage{mathtools}
\usepackage{microtype}
\usepackage[automark]{scrpage2}%

\usepackage{165arxiv}

\mathtoolsset{mathic=true}
\numberwithin{equation}{section} 
 \newtheorem{thm}{Theorem}[section]
 
 \newtheorem{lem}[thm]{Lemma}
 \newtheorem{prop}[thm]{Proposition}
\theoremstyle{definition}
 
 \newtheorem{nota}[thm]{Notation}
\theoremstyle{remark}
 \newtheorem{rem}[thm]{Remark}
 \newtheorem{ex}{Example}
 \numberwithin{equation}{section}

\begin{document}
 \title{On the structure of Hausdorff moment sequences of complex matrices}
 \author{Bernd Fritzsche \and Bernd Kirstein \and Conrad M\"adler}
 \dedication{Dedicated to Daniel Alpay on the occasion of his 60th birthday}

 \maketitle

\begin{abstract}
 The paper treats several aspects of the truncated matricial \tabH{} type moment problems.
 It is shown that each \tabH{} moment sequence has a particular intrinsic structure.
 More precisely, each element of this sequence varies within a closed bounded matricial interval.
 The case that the corresponding moment coincides with one of the endpoints of the interval plays a particular important role.
 This leads to distinguished molecular solutions of the truncated matricial \tabH{} moment problem, which satisfy some extremality properties.
 The proofs are mainly of algebraic character.
 The use of the parallel sum of matrices is an essential tool in the proofs.
\end{abstract}

\begin{description}
 \item[Mathematics Subject Classification (2010):] 44A60, 47A57.
 \item[Keywords:] truncated matricial Hausdorff moment problem, canonical molecular solutions, matricial intervals associated with matricial Hausdorff moment sequences, parallel sum of matrices.
\end{description}

\section{Introduction}
 The starting point of this paper was a question connected to matricial versions of the truncated power moment problem on a compact interval \(\ab\) of the real axis.
 In joint work with A.~E.~Choque~Rivero and Yu.~M.~Dyukarev (see~\zitas{MR2222521,MR2342899}), the first and second authors could extend the characterizations of solvability of this moment problem, which were given in the scalar case by M.~G.~Krein~\zita{MR0044591} (see also Krein/Nudelman~\zitaa{MR0458081}{\cchap{III}}) to the matrix case.
 In the case that \(q\in\N\), \(n\in\N\), and \(\seqs{2n-1}\) is a sequence of complex \tqqa{matrices} for which the moment problem in question is solvable, in their joint paper~\zita{MR2735313} with Yu.~M.~Dyukarev the authors constructed a concrete molecular solution (,\ie{}, a discrete \tnnH{} \tqqa{measure} concentrated on finitely many points of the interval \(\ab\)) of the moment problem.
 The motivation for this paper was to find an explicit molecular solution for the case of a given sequence \(\seqs{2n}\) of prescribed moments.
 A closer look at our method used in~\zita{MR2735313} shows that the realization of our aim can be reached by a thorough study of the structure of finite \tabHnnd{} sequences of complex \tqqa{matrices} (see~\rdefn{D1159}).
 The key information comes from \rthm{T0718}, which says that if \(m\in\N\) and if \(\seqs{m}\) is an \tabHnnd{} sequence, then we can always find a complex \tqqa{matrix} \(\su{m+1}\) such that the sequence \(\seqs{m+1}\) is \tabHnnd{}.
 We are even able to describe all complex \tqqa{matrices} \(\su{m+1}\) which can be chosen to realize this aim.
 More precisely, the set of all these matrices \(\su{m+1}\) turns out to be a closed matricial interval of complex \tqqa{matrices}.
 As well the left and right endpoint as the midpoint of this interval play (as clearly expected) a distinguished role (see \rsec{S0747}).
 The main part of this paper is concerned with the investigation of the structure of \tabHnnd{} sequences of complex \tqqa{matrices} and the study of the above mentioned extension problem for such sequences.
 These results lead us to interesting new insights concerning a whole family of particular molecular solutions of the matrix version of the truncated \tabH{} moment problem.
 In particular, we guess that the choice of the endpoints of the interval exactly leads to those extremal molecular solutions which were studied by M.~G.~Krein~\zita{MR0044591} (see also Krein/Nudelman~\zitaa{MR0458081}{\cchap{III}, \S5}).
 M.~G.~Krein found them via the lower and upper principal representation of the given moment sequence (see~\rsec{S0733}).
 
 A more careful view shows that the situation is in some sense similar as in the case of \tnn{} definite sequences from \(\Cqq\) (see~\zita{MR885621}) or \tpqa{Schur} sequences (see~\zita{MR918682}).
 If \(n\in\N\) and if \(\seq{C_j}{j}{0}{n}\) is a sequence of one of the just mentioned types, then, for each \(m\in\set{1,2,\dotsc,n}\), the matrix \(C_m\) belongs to a matrix ball the parameters of which depend on the sequence \(\seq{C_j}{j}{0}{m-1}\).
 Having in mind the geometry of a matrix ball, we see that there are two types of distinguished points, namely the center of the matrix ball at the one hand and its boundary points on the other hand.
 The \tqqa{\tnn{}} definite sequences or \tpqa{Schur} sequences which are starting from some index consist only of the centers of the matrix balls in question occupy an extremal position in the set of all sequences of the considered types.
 Similar things can be said about those sequences which contain an element of the boundary of the relevant matrix ball.
 A similar situation will be met for \tabHnnd{} sequences.
 This will be discussed in detail in \rsec{S0747}.
 
 The study of moment spaces was initiated by C.~Carath\'eodory~\zitas{MR1511425,zbMATH02629876} in the context of the trigonometric moment problem.
 The approach of Carath\'eodory was based on the theory of convexity.
 O.~Toeplitz~\zita{zbMATH02629875} observed that the results of Carath\'eodory can be reformulated in terms of \tnnH{} Toeplitz matrices.
 This view was then generally accepted and is also the basis for the approach in the matrix case in~\zitas{MR885621,MR918682}.
 
 In the study of the moment space connected with the \tabH{} moment problem, the theory of convexity played an important role from the very beginning.
 These developments were strongly influenced by M.~G.~Krein's landmark paper~\zita{MR0044591}, which essentially determined the further direction of research reflected in the monographs Karlin/Shapley~\zita{MR0059329}, Karlin/Studden~\zita{MR0204922}, and Krein/Nudelman~\zita{MR0458081}.
 It should be mentioned that Skibinsky~\zitas{MR0228040,MR0246351} considered probability measures on \([0,1]\) and observed that the \((n+1)\)\nobreakdash-th moment of such measures can vary within a closed bounded interval of the real axis.
 The work of Skibinsky was also an important source of inspiration for the theory of canonical moments, which was worked out by Dette/Studden~\zita{MR1468473}

\section{On the matrix version of the truncated \habH{} moment problem}\label{S1215}
 In order to formulate the moment problem, we are going to study, we first give some notation.
 Let \(\C\)\index{c@\(\C\)}, \(\R\)\index{r@\(\R\)}, \(\NO\)\index{n@\(\NO\)}, and \(\N\)\index{n@\(\N\)} be the set of all complex numbers, the set of all real numbers, the set of all \tnn{} integers, and the set of all positive integers, resp.
 For every choice of \(\kappa,\tau\in\R\cup\set{-\infty,\infi}\), let \symba{\mn{\kappa}{\tau}}{z} be the set of all integers \(\ell\) for which \(\kappa\leq\ell\leq\tau\) holds.
 Throughout this paper, let \(p\)\index{p@\(p\)} and \(q\)\index{q@\(q\)} be positive integers.
 If \(\cX\) is a \tne{} set, then \(\cX^\pxq\)\index{\(\cX^\pxq\)} stands for the set of all \tpqa{matrices} each entry of which belongs to \(\cX\), and \(\cX^p\)\index{\(\cX^p\)} is short for \(\cX^\xx{p}{1}\).
 We will write \symba{\CHq}{c}, \symba{\Cggq}{c}, and \symba{\Cgq}{c} for the set of all \tH{} complex \tqqa{matrices}, the set of all \tnnH{} complex \tqqa{matrices}, and the set of all \tpH{} complex \tqqa{matrices}, resp.
 If \(A\) and \(B\) are complex \tqqa{matrices}, then \symb{A\lgeq B} or \symb{B\lleq A} (resp.\ \symb{A>B} or \symb{B<A}) means that \(A\) and \(B\) are \tH{} and \(A-B\) is \tnnH{} (resp.\ \tpH{}).
 
 Let \((\Omega,\mathfrak{A})\) be a measurable space.
 Each countably additive mapping whose domain is \(\gA\) and whose values belong to \(\Cggq\) is called a \tnnH{} \tqqa{measure} on \((\Omega, \gA)\).
 For the integration theory with respect to \tnnH{} measures, we refer to Kats~\zita{MR0080280} and Rosenberg~\zita{MR0163346}.
 If \(\mu = \matauuo{\mu_{jk}}{j,k}{1}{q}\) is a \tnnH{} measure on \((\Omega,\mathfrak{A})\), then each entry function \(\mu_{jk}\) is a complex measure on \((\Omega,\mathfrak{A})\).
 In particular, \(\mu_{11},\mu_{22},\dotsc,\mu_{qq}\) are finite \tnn{} real-valued measures.
 For each \(H\in\Cggq\), the inequality \(H\leq(\tr H)\Iq\) holds true.
 Hence, each entry function \(\mu_{jk}\) is absolutely continuous with respect to the so-called trace measure \(\tau\defeq\sum_{j=1}^q\mu_{jj}\)\index{t@\(\tau\)} of \(\mu\), \ie{}, for each \(M\in\mathfrak{A}\) which satisfies \(\tau(M)=0\), it follows \(\mu(M)=\Oqq\).
 The Radon-Nikodym derivatives \symba{\dif\mu_{jk}/\dif\tau}{d} are thus well defined up to sets of zero \(\tau\)\nobreakdash-measure.
 Obviously, the matrix-valued function \(\mu_\tau'\defeq\matauuo{\dif\mu_{jk}/\dif\tau}{j,k}{1}{q}\)\index{\(\mu_\tau'\)} is \(\mathfrak{A}\)\nobreakdash-measurable and integrable with respect to \(\tau\).
 The matrix-valued function \(\mu_\tau'\) is said to be the trace derivative of \(\mu\).
 If \(\nu\) is a \tnn{} real-valued measure on \(\mathfrak{A}\), then let the class of all \(\mathfrak{A}\)\nobreakdash-measurable \tpqa{matrix-valued} functions \(\Phi=\matau{\phi_{jk}}{\substack{j=1,\dotsc,p\\ k=1,\dotsc,q}}\) on \(\Omega\) such that each \(\phi_{jk}\) is integrable with respect to \(\nu\) be denoted by \symba{\ek{\LaaaC{\Omega}{\mathfrak{A}}{\nu}}^\xx{p}{q}}{l}.
 An ordered pair \((\Phi,\Psi)\) consisting of an \(\mathfrak{A}\)\nobreakdash-measurable \tpqa{matrix-valued} function \(\Phi\) on \(\Omega\) and an \(\mathfrak{A}\)\nobreakdash-measurable \taaa{r}{q}{matrix-valued} function \(\Psi\) on \(\Omega\) is said to be integrable with respect to a \tnnH{} measure \(\mu\) on \((\Omega,\mathfrak{A})\) if \(\Phi\mu_\tau'\Psi^\ad\) belongs to \(\ek{\LaaaC{\Omega}{\mathfrak{A}}{\tau}}^\xx{p}{r}\), where \(\tau\) is the trace measure of \(\mu\).
 In this case, the integral of \((\Phi,\Psi)\) with respect to \(\mu\) is defined by\index{\(\int_\Omega\Phi\dif\mu\Psi^\ad\)}
 \[
  \int_\Omega\Phi\dif\mu\Psi^\ad
  \defeq\int_\Omega\Phi\mu_\tau'\Psi^\ad\dif\tau
 \]
 and for any \(M\in\mathfrak{A}\), the pair \((1_M\Phi,1_M\Psi)\) is integrable with respect to \(\mu\), where \symba{1_M}{1} is the indicator function of the set \(M\).
 Then the integral of \((\Phi,\Psi)\) with respect to \(\mu\) over \(M\) is defined by \(\int_M\Phi\dif\mu\Psi^\ad\defeq\int_\Omega\rk{1_M\Phi}\mu_\tau'\rk{1_M\Psi}^\ad\dif\tau\)\index{\(\int_M\Phi\dif\mu\Psi^\ad\)}.
 An \(\mathfrak{A}\)\nobreakdash-measurable complex-valued function \(f\) on \(\Omega\) is said to be integrable with respect to a \tnnH{} measure \(\mu\) on \((\Omega,\mathfrak{A})\) if the pair \((f\Iq,\Iq)\) is integrable with respect to \(\mu\).
 In this case the integral of \(f\) with respect to \(\mu\) is defined by\index{\(\int_\Omega f\dif\mu\)}
 \[
  \int_\Omega f\dif\mu
  \defeq\int_\Omega(f\Iq)\dif\mu\Iq^\ad
 \]
 and for any \(M\in\mathfrak{A}\), the function \(1_M f\) is integrable with respect to \(\mu\).
 Then the integral of \(f\) with respect to \(\mu\) over \(M\) is defined by \(\int_M f\dif\mu\defeq\int_M(f\Iq)\dif\mu\Iq^\ad\)\index{\(\int_M f\dif\mu\)}.
 We denote by \symba{\Loaaaa{1}{\Omega}{\mathfrak{A}}{\mu}{\C}}{l} the set of all \(\mathfrak{A}\)\nobreakdash-measurable complex-valued functions \(f\) on \(\Omega\) which are integrable with respect to a \tnnH{} measure \(\mu\) on \((\Omega,\mathfrak{A})\).

 Let \(\BorR\)\index{b@\(\BorR\)} (resp.\ \(\gB_\C \)) be the \(\sigma\)\nobreakdash-algebra of all Borel subsets of \(\R\) (resp.\ \(\C \)).
 For all \(\Omega\in\BorR\setminus\set{\emptyset}\), let \(\Bori{\Omega}\)\index{b@\(\Bori{\Omega}\)} be the \(\sigma\)\nobreakdash-algebra of all Borel subsets of \(\Omega\) and let \(\Mggqa{\Omega}\)\index{m@\(\Mggqa{\Omega}\)} be the set of all \tnnH{} \tqqa{measures} on \((\Omega,\Bori{\Omega})\).
 For all \(\kappa\in\NOinf \), let \(\Mgguqa{\kappa}{\Omega}\)\index{m@\(\Mgguqa{\kappa}{\Omega}\)} be the set of all \(\sigma\in\Mggqa{\Omega}\) such that for each \(j\in\mn{0}{\kappa}\) the function \(f_j\colon\Omega\to\C\) defined by \(f_j (x) \defeq  x^j\) belongs to \(\cL^1 (\Omega,\Bori{\Omega},\sigma;\C)\).
 If \(\sigma\in\MggquO{\kappa}\), then let
 \begin{equation}\label{mom}
 \suo{j}{\sigma}
 \defg\int_\Omega x^j\sigma(\dif x)
 \end{equation}
 \index{s@\(\suo{j}{\sigma}\)}for all \(j\in\mn{0}{\kappa}\).
 Obviously, \(\Mggoa{q}{\Omega}\subseteq \MggquO{k}\subseteq \MggquO{k+1}\subseteq\MggquO{\infi}\) holds true for every choice of \(\Omega\in\gB_\R\setminus \{\emptyset\}\) and \(k\in\NO \).
 Furthermore, if \(\Omega\) is a \tne{} bounded Borel subset of \(\R\), then one can easily see that \(\Mggoa{q}{\Omega} =\MggquO{\infi}\).
 In particular, for all \(\ug\in\R\) and \(\og\in(\ug,\infp)\) we have
\beql{ab8}
 \MggqF
 =\MggquF{\infi}.
\eeq
 
 Let \symba{\MggqmolO}{m} be the set of all \(\sigma\in\MggqO\) for which there exists a finite subset \(B\) of \(\Omega\) satisfying \(\sigma\rk{\Omega\setminus B}=\Oqq\).
 The elements of \(\MggqmolO\) are called \notii{molecular}.
 Obviously, \(\MggqmolO\) is the set of all \(\sigma\in\MggqO\) for which there exist an \(m\in\N\) and sequences \(\seq{\xi_\ell}{\ell}{0}{m}\) and \(\seq{M_\ell}{\ell}{0}{m}\) from \(\Omega\) and \(\Cggq\) such that
 \(
  \sigma
  =\sum_{\ell=1}^m\Kronu{\xi_\ell}M_\ell
 \),
 where \symba{\Kronu{\xi_\ell}}{d} is the Dirac measure on \(\BorO\) with unit mass at \(\xi_\ell\), \(\ell\in\mn{1}{m}\).
 In particular, we have \(\MggqmolO\subseteq\MggquO{\infi}\).
 
 The following matricial moment problem lies in the background of our considerations:

 \begin{description}
  \item[\mproblem{\Omega}{\kappa}{=}]\index{m@\mproblem{\Omega}{\kappa}{=}} Let \(\Omega\in\BorR\setminus\set{\emptyset}\) and let  \(\seqs{\kappa}\) be a sequence from \(\Cqq\).
  Parametrize the set \(\Mggqaag{\Omega}{\seqs{\kappa}}\)\index{m@\(\Mggqaag{\Omega}{\seqs{\kappa}}\)} of all \(\sigma\in\Mgguqa{\kappa}{\Omega}\) such that \(s_j^{[\sigma]}=\su{j}\) for all \(j\in\mn{0}{\kappa}\).
 \end{description}
 
 In this paper, we often use the procedure of reflecting measures on the real axis.
 For this reason, we introduce some terminology.
 Let \(R\colon\R\to\R\) be defined by \(x\mapsto-x\).
 Then \(R\) is a continuous involution of \(\R\) and consequently \(\BorR\)\nobreakdash-\(\BorR\)\nobreakdash-measurable.
 In view of \(\Omega\in\BorR\setminus\set{\emptyset}\), the set \symb{\refl{\Omega}\defeq R^\inv\rk{\Omega}} belongs to \(\BorR\setminus\set{\emptyset}\) and the mapping \symba{R_\Omega\defeq\Rstr_\Omega R}{r} is \(\BorO\)\nobreakdash-\(\Bori{\refl{\Omega}}\)\nobreakdash-measurable, whereas the mapping \symba{R_{\refl{\Omega}}\defeq\Rstr_{\refl{\Omega}}R}{r} is \(\Bori{\refl{\Omega}}\)\nobreakdash-\(\BorO\)\nobreakdash-measurable.
 Moreover \(R_\Omega\circ R_{\refl{\Omega}}=\Id_\Omega\) and \(R_{\refl{\Omega}}\circ R_\Omega=\Id_{\refl{\Omega}}\).
 For each \(\sigma\in\MggqO\), we denote by \symb{\refl{\sigma}} the image measure of \(\sigma\) with respect to \(R_\Omega\), \ie{}, for \(B\in\Bori{\refl{\Omega}}\), we have
\beql{reflm}
 \refl{\sigma}\rk{B}
 \defeq\sigma\rk*{R_\Omega^\inv\rk{B}}.
\eeq
 By construction then \(\refl{\sigma}\in\Mggqa{\refl{\Omega}}\).
 If \(\kappa\in\NOinf\) and if \(\sigma\in\MggquO{\kappa}\) then it is easily checked that \(\refl{\sigma}\in\Mgguoa{\kappa}{q}{\refl{\Omega}}\) and that \(\suo{j}{\refl{\sigma}}=(-1)^j\suo{j}{\sigma}\) for all \(j\in\mn{0}{\kappa}\).

 Using the preceding considerations we get the following result:
\breml{R0817}
 Let \(\Omega\in\BorR\setminus\set{\emptyset}\), let \(\kappa\in\NOinf\), and let \(\seqska\) be a sequence from \(\Cqq\).
 Let \(\refl{\Omega}\defeq\setaa{x\in\R}{-x\in\Omega}\) and let \(r_j\defeq(-1)^j\su{j}\) for all \(j\in\mn{0}{\kappa}\).
 Then  \(\refl{\Omega}\in\BorR\setminus\set{\emptyset}\) and, using the notation introduced in \eqref{reflm}, we get \(\Mggqaag{\refl{\Omega}}{\seqr{\kappa}}=\setaa{\refl{\sigma}}{\sigma\in\MggqOsg{\kappa}}\).
\erem

 The discussions of this paper are mostly concentrated on the case that the set \(\Omega\) is a bounded and closed interval of the real axis.
 Such moment problems are called to be of Hausdorff-type.
 The following special case of \rrem{R0817} is of particular interest for us.
\breml{R0827}
 Let \(\ug\in\R\), let \(\og\in(\ug,\infp)\), let \(\kappa\in\NOinf\), and let \(\seqska\) be a sequence from \(\Cqq\).
 Let \(r_j\defeq(-1)^j\su{j}\) for all \(j\in\mn{0}{\kappa}\).
 Then \(\Mggqaag{[-\og,-\ug]}{\seqr{\kappa}}=\setaa{\refl{\sigma}}{\sigma\in\MggqFs{\kappa}}\).
\erem
 
 The solvability of the truncated matricial \tabH{}-type moment problem can be characterized as follows:
\bthmnl{\zitaa{MR2222521}{\cthm{1.2}}}{T12}
 Let \(\ug\in\R\), let \(\og\in(\ug,\infp)\), let \(n\in\N\), and let \(\seqs{2n-1}\) be a sequence from \(\Cqq\).
 Then \(\MggqFs{2n-1} \ne \emptyset\)  if and only if the block Hankel matrices
\begin{align}\label{T12.A}
 &\matauuo{-\alpha\su{j+k} +\su{j+k+1}}{j,k}{0}{n-1}&
 &\text{and}&
 &\matauuo{\beta\su{j+k}  -\su{j+k+1}}{j,k}{0}{n-1}
\end{align}
 are both \tnnH{}.
\end{thm}

\bthmnl{\zitaa{MR2342899}{\cthm{1.3}}}{T1.2*}
 Let \(\ug\in\R\), let \(\og\in(\ug,\infp)\), let \(n\in\N\), and let \(\seqs{2n}\) be a sequence from \(\Cqq\).
 Then \(\MggqFs{2n} \ne \emptyset\)  if and only if the block Hankel matrices
\begin{align}\label{CR2}
 &\matauuo{\su{j+k}}{j,k}{0}{n}&
 &\text{and}&
 &\matauuo{-\alpha\beta\su{j+k} + (\alpha +\beta)\su{j+k+1} -\su{j+k+2}}{j,k}{0}{n-1}
\end{align}
 are both \tnnH{}.
\end{thm}

 In the scalar case \(q=1\), \rthmss{T12}{T1.2*} are due to  M.~G.~Krein~\cite{MR0044591} (see also~\zitaa{MR0458081}{\cchap{III}}).
 The approach of M.~G.~Krein is essentially based on the use of the theory of convexity along the lines applied by C.~Carath\'eodory~\zitas{MR1511425,zbMATH02629876} in the context of the trigonometric moment problem.
 The proofs of \rthmss{T12}{T1.2*} given in~\zitas{MR2222521,MR2342899}, resp., are rather complicated and not constructive.
 They are based on Potapov's method of fundamental matrix inequalities and an extensive explicit solving procedure of the corresponding system of these inequalities in the non-degenerate case.
 Indeed, if the block Hankel matrices in \eqref{T12.A} and \eqref{CR2} are both \tpH{}, then~\zitaa{MR2222521}{\ccor{6.13}} and~\zitaa{MR2342899}{\ccor{6.15}} show that Problem~\mproblem{\ab}{2n-1}{=} and Problem~\mproblem{\ab}{2n}{=} have solutions.
 In the proofs of~\zitaa{MR2222521}{\cthm{1.2}} and~\zitaa{MR2342899}{\cthm{1.3}} a perturbation argument for the construction of a sequence of corresponding approximating non-degenerate problems is used and then a matricial version of the Helly-Prokhorov Theorem (see~\cite[Satz~9]{MR975253} or~\cite[Lemma~2.2.1]{MR1152328}) is applied.

\section{On \hHnnd{} sequences}\label{S0832}
 In this section, we summarize essential properties on two classes of sequences of complex \tqqa{matrices} which are a main tool for the investigations of this paper.
 This material is mainly taken from~\zitas{MR2570113,MR2805417}.
 
\bdefnl{D1238}
 Let \(n\in\NO\) and let \(\seqs{2n}\) be a sequence from \(\Cqq\).
 Let %
\symba{\Hu{n}\defeq\matauuo{\su{j+k}}{j,k}{0}{n}}{h}.
 Then \(\seqs{2n}\) is called \noti{\tHnnd}{n@non-negative definite!Hankel} (resp.\ \noti{\tHpd}{p@positive definite!Hankel}) if \(\Hu{n}\in\Cggo{(n+1)q}\) (resp.\ \(\Cgo{(n+1)q}\)).
 We denote by \symba{\Hggqu{2n}}{h} (resp.\ \symba{\Hgqu{2n}}{h}) the set of all \tHnnd{} (resp.\ \tHpd{}) sequences \(\seqs{2n}\) from \(\Cqq\).
\edefn

\bdefnl{D1255}
 \benui
  \il{D1255.a} Let \(n\in\NO\) and let \(\seqs{2n}\) be a sequence from \(\Cqq\).
  Then \(\seqs{2n}\) is called \notii{\tHnnde{}} (resp.\ \notii{\tHpde{}}) if there exist matrices \(\su{2n+1}\) and \(\su{2n+2}\) from \(\Cqq\) such that \(\seqs{2n+2}\in\Hggqu{2n+2}\) (resp.\ \(\seqs{2n+2}\in\Hgqu{2n+2}\)).
  \il{D1255.b} Let \(n\in\NO\) and let \(\seqs{2n+1}\) be a sequence from \(\Cqq\).
  Then \(\seqs{2n+1}\) is called \notii{\tHnnde{}} (resp.\ \notii{\tHpde{}}) if there exists a matrix \(\su{2n+2}\) from \(\Cqq\) such that \(\seqs{2n+2}\in\Hggqu{2n+2}\) (resp.\ \(\seqs{2n+2}\in\Hgqu{2n+2}\)).
 \eenui
 If \(m\in\NO\), then the notation \symba{\Hggequ{m}}{h} (resp.\ \symba{\Hgequ{m}}{h}) stands for the set of all sequences \(\seqs{m}\) from \(\Cqq\) which are \tHnnde{} (resp.\ \tHpde{}).
\edefn

 The importance of the class \(\Hggequ{m}\) in the context of moment problems is caused by the following observation:
\begin{thm}[see~\zitaa{MR2570113}{\ctheo{4.17}}]\label{T1.2}
 Let \(m\in\NO\) and let \(\seqs{m}\) be a sequence from \(\Cqq\).
 Then \(\MggqRsg{m}\neq\emptyset\) if and only if \(\seqs{m}\in\Hggequ{m}\).
\end{thm}

 Now we are going to indicate some essential features of the structure of \tHnnd{} sequences.
 First we introduce some matrices which occupy a key role in the sequel.
 For each matrix \(A\in\Cpq\), we denote by \symb{A^\mpi} its Moore-Penrose inverse.
 This means that \(A^\mpi\) is the unique matrix \(X\) from \(\Cqp\) which satisfies the four equations \(AXA=A\), \(XAX=X\), \((AX)^\ad=AX\), and \((XA)^\ad=XA\).
 For every choice of \(n\in\N\) and \(A_1,A_2,\dotsc,A_n\in\Cpq\), let \symba{\col\seq{A_j}{j}{1}{n}\defeq\mat{A_1^\ad,A_2^\ad,\dotsc,A_n^\ad}^\ad}{c} and let \symba{\row\seq{A_j}{j}{1}{n}\defeq\mat{A_1,A_2,\dotsc,A_n}}{r}.
 The null matrix which belongs to \(\Cpq\) is denoted by \symba{\Opq}{0}, whereas \symba{\Ip}{i} is the identity matrix belonging to \(\Cpp\).
\bnotal{N41}
 Let \(\kappa\in\NOinf\) and let \(\seqska \) be a sequence from \(\Cpq\).
\benui
 \il{N41.a} Let \symba{\Hu{n}\defeq \matauuo{\su{j+k}}{j,k}{0}{n}}{h} for all \(n\in\NO \) with  \(2n\leq\kappa\), let \symba{\Ku{n}\defeq\matauuo{\su{j+k+1}}{j,k}{0}{n}}{k} for all  \(n\in\NO \) with  \(2n+1\leq\kappa\), and let \symba{\Gu{n}\defeq\matauuo{\su{j+k+2}}{j,k}{0}{n}}{g} for all  \(n\in\NO \) with \(2n+2\leq\kappa\).
 \il{N41.b} Let \symba{\yuu{\ell}{m} \defeq  \col (s_j)_{j=\ell}^m}{y} and \symba{\zuu{\ell}{m} \defeq\row(s_j)_{j=\ell}^m}{z} for all \(\ell,m\in\NO\) with \(\ell\leq m\leq\kappa\).
 \il{N41.c} Let \symba{\Trip{0} \defeq  \Opq}{t} and, for all \(n\in\N\) with \(2n-1\leq\kappa\), let \symba{\Trip{n}\defeq  \zuu{n}{2n-1}  \Hu{n-1}^\mpi    \yuu{n}{2n-1}}{t}.
 Furthermore, for all \(n\in\NO \) with  \(2n\leq\kappa\), let \symba{\Lu{n}  \defeq\su{2n}-\Trip{n}}{l}.
 \il{N41.d} Let \symba{\Mu{0} \defeq  \Opq}{m},
 and, for all \(n\in\N\) with \(2n\leq\kappa\), let \symba{\Mu{n} \defeq  \zuu{n}{2n-1}\Hu{n-1}^\mpi \yuu{n+1}{2n}}{m}.
\eenui
\end{nota}

 If we build the matrices introduced in \rnota{N41} from an other given sequence, \eg{} \(\seq{t_j}{j}{0}{\kappa}\), then this is indicated by a superscript \(\langle t\rangle\), \ie{} \symba{\Hu{n}^{\langle t\rangle}\defeq \matauuo{t_{j+k}}{j,k}{0}{n}}{h}, etc.

\breml{R3240}
 Let \(n\in\N\) and let \(\seqs{2n}\) be a sequence from \(\Cqq\).
 Then the block Hankel matrix \(\Hu{n}\) admits the block partition
 \beql{R3240.1}
  \Hu{n}
  =
  \bMat
   \Hu{n-1}&\yuu{n}{2n-1}\\
   \zuu{n}{2n-1}&\su{2n}
  \eMat.
 \eeq
\erem

 Let \symba{\ran{A}\defeq\setaa{Ax}{x\in\Cq}}{r} and \symba{\nul{A}\defeq\setaa{x\in\Cq}{Ax=\Ouu{p}{1}}}{n} be the column space and the null space of a matrix \(A\in\Cpq\).
\breml{R1240}
 Let \(n\in\N\) and let \(\seqs{2n}\) be a sequence from \(\Cqq\).
 Then:
 \benui
  \il{R1240.a} \(\seqs{2n}\in\Hggqu{2n}\) if and only if \(\seqs{2n-2}\in\Hggqu{2n-2}\), \(\Lu{n}\in\Cggq\), \(\ran{\yuu{n}{2n-1}}\subseteq\Hu{n-1}\), and \(\su{2n+1}\in\CHq\).
  \il{R1240.b} \(\seqs{2n}\in\Hgqu{2n}\) if and only if \(\seqs{2n-2}\in\Hgqu{2n-2}\), \(\Lu{n}\in\Cgq\), and \(\su{2n+1}\in\CHq\).
 \eenui
\erem
\bproof
 Taking into account \eqref{R3240.1}, this is a consequence of \rlem{AEP}.
\eproof

 A sequence \(\seqsinf\) from \(\Cqq\) is called \noti{\tHnnd{}}{\hHnnd{}} (resp.\ \noti{\tHpd{}}{\hHpd{}}) if \(\seqs{2n}\in\Hggqu{2n}\) (resp.\ \(\seqs{2n}\in\Hgqu{2n}\)) for all \(n\in\NO\).
 We denote by \symba{\Hggqinf}{h} (resp.\ \symba{\Hgqinf}{h}) the set of all \tHnnd{} (resp.\ \tHpd{}) sequences \(\seqsinf\) from \(\Cqq\).

\breml{L1251}
 Let \(\kappa\in\NOinf\).
 Then:
 \benui
  \il{L1251.a} If \(\seqs{2\kappa}\in\Hggqu{2\kappa}\), then \(\su{j}\in\CHq\) for all \(j\in\mn{0}{2\kappa}\) and \(\su{2k}\in\Cggq\) for all \(k\in\mn{0}{\kappa}\).
  \il{L1251.b} If \(\seqs{2\kappa}\in\Hgqu{2\kappa}\), then \(\su{j}\in\CHq\) for all \(j\in\mn{0}{2\kappa}\) and \(\su{2k}\in\Cgq\) for all \(k\in\mn{0}{\kappa}\).
 \eenui
\erem

\bleml{L1548}
 Let \(n\in\N\) and let \(\seqs{2n-1}\) be a sequence from \(\CHq\).
 Then:
 \benui
  \il{L1548.a} If \(\seqs{2n-2}\in\Hggqu{2n-2}\), then \(\Trip{n}\in\Cggq\).
  \il{L1548.b} If \(\seqs{2n-2}\in\Hgqu{2n-2}\) and \(n\geq2\), then \(\Trip{n}\in\Cgq\).
 \eenui
\elem
\bproof
 \eqref{L1548.a} This follows immediately from \rremss{R1546}{R0705}.

 \eqref{L1548.b} Now suppose that \(\seqs{2n-2}\in\Hgqu{2n-2}\) and \(n\geq2\). Then \(\Hu{n-1}\) is \tpH{}. In particular, \(\Hu{n-1}^\inv\) is \tpH{}. From \rremp{L1251}{L1251.b}  we obtain \(\su{2n-2}\in\Cgq\). In particular, \(\det\su{2n-2}\neq0\). Because of \(n\geq2\), the matrix \(\su{2n-2}\) is a block in \(\yuu{n}{2n-1}\). Thus, \(\rank\yuu{n}{2n-1}=q\). Hence, \(\Trip{n}\in\Cgq\) follows from \rrem{R0705}.
\eproof

\breml{R1451}
 Let \(n\in\NO\).
 Then \(\Hggequ{2n}\subseteq\Hggqu{2n}\) and, in the case \(n\in\N\), furthermore \(\Hggequ{2n}\neq\Hggqu{2n}\).
\erem

 For all \(n\in\NO\) let\index{j@\(\Jq{n}\defeq\diag\seq{(-1)^j\Iq}{j}{0}{n}\)}
\[
 \Jq{n}
 \defeq\diag\Seq{(-1)^j\Iq}{j}{0}{n}.
\]

\breml{R0950}
 Let \(n\in\NO\).
 Then \(\Jq{n}^\ad=\Jq{n}\) and \(\Jq{n}^2=\Iu{(n+1)q}\).
 In particular, \(\Jq{n}\) is unitary.
\erem

\bleml{L0937}
 Let \(n\in\NO\) and let \(\seqs{2n}\) be a sequence from \(\Cqq\).
 Let the sequence \(\seqr{2n}\) be given by \(r_j\defeq(-1)^j\su{j}\).
 Then:
 \benui
  \il{L0937.a} Let \symba{\Huo{n}{s}\defeq\matauuo{\su{j+k}}{j,k}{0}{n}}{h} and \symba{\Huo{n}{r}\defeq\matauuo{r_{j+k}}{j,k}{0}{n}}{h}.
 Then \(\Huo{n}{r}=\Jq{n}^\ad\Huo{n}{s}\Jq{n}\).
  \il{L0937.b} \(\seqs{2n}\in\Hggqu{2n}\) if and only if \(\seqr{2n}\in\Hggqu{2n}\).
  \il{L0937.c} \(\seqs{2n}\in\Hgqu{2n}\) if and only if \(\seqr{2n}\in\Hgqu{2n}\).
 \eenui
\elem
\bproof
 \rPart{L0937.a} follows by direct computation. 
 \rPartss{L0937.b}{L0937.c} follow from~\eqref{L0937.a}.
\eproof

\bleml{L0946}
 Let \(m\in\NO\) and let \(\seqs{m}\) be a sequence from \(\Cqq\).
 Let the sequence \(\seqr{m}\) be given by \(r_j\defeq(-1)^j\su{j}\).
 Then \(\seqs{m}\in\Hggequ{m}\) if and only if \(\seqr{m}\in\Hggequ{m}\).
\elem
\bproof
 Use \rlemp{L0937}{L0937.b}.
\eproof

 Now we turn our attention to a subclass of \tHnnd{} sequences, which plays a central role in the sequel.
 If \(n\in\NO\) and if \(\seqs{2n}\in\Hggqu{2n}\), then \(\seqs{2n}\) is called \notii{\tHd{}} if \(\Lu{n}=\Oqq\), where \(\Lu{n}\) is given in \rnotap{N41}{N41.c}.
 We denote by \symba{\Hggdqu{2n}}{h} the set of all sequences \(\seqs{2n}\in\Hggqu{2n}\) which are \tHd{}.
 If \(n\in\NO\), then a sequence \(\seqsinf\in\Hggqinf\) is said to be \notii{\tHdo{n}} if \(\seqs{2n}\in\Hggdqu{2n}\).
 A sequence \(\seqsinf\in\Hggqinf\) is called \notii{\tHd{}} if there exists an \(n\in\NO\) such that \(\seqsinf\) is \tHdo{n}.

\section{On \habHnnd{} sequences from $\Cqq$}\label{S1237}
 Against to the background of \rthmss{T12}{T1.2*} now we are going to introduce one of the central notions of this paper.
 Before doing this, we introduce some notation.

\begin{nota}\label{N5.1}
 Let \(\alpha,\beta\in\C\), let \(\kappa\in\Ninf\), and let  \(\seqska \) be a sequence from \(\Cqq\).
 Then let the sequences \((\sau{j})_{j=0}^{\kappa - 1}\) and \((\sbu{j})_{j=0}^{\kappa - 1}\) be given by 
\begin{align*}%
 \sau{j}&\defeq  -\alpha s_j +\su{j+1}&
 &\text{and}&
 \sbu{j}&\defeq \beta s_j -\su{j+1}&
 \text{for each }j&\in\mn{0}{\kappa -1}.
\end{align*}
 Furthermore, if \(\kappa \ge 2\), then let \((\scu{j})_{j=0}^{\kappa - 2}\) be defined by 
\begin{align*}%
 \scu{j}&\defeq  -\alpha\beta s_j + (\alpha +\beta)\su{j+1} -\su{j+2}&
 \text{for all }j&\in\mn{0}{\kappa -2}.
\end{align*}
\end{nota}
 When using the sequences from \rnota{N5.1}, we write \symba{\Hau{n}}{h}, \symba{\Hbu{n}}{h}, \symba{\Hcu{n}}{h}, etc.
 
\bdefnl{D1159}
 Let \(\ug\in\R\), let \(\og\in(\ug,\infp)\), and let \(n\in\NO\).
 \benui
  \il{D1159.a} Let \(\seqs{2n+1}\) be a sequence from \(\Cqq\).
  Then \(\seqs{2n+1}\) is called \noti{\tabHnnd{}}{h@\habHnnd} (resp.\ \noti{\tabHpd{}}{h@\habHpd}) if both sequences \(\seqsa{2n}\) and \(\seqsb{2n}\) are \tHnnd{} (resp.\ \tHpd{}).
  \il{D1159.b} Let \(\seqs{2n}\) be a sequence from \(\Cqq\).
  Then \(\seqs{2n}\) is called \noti{\tabHnnd{}}{h@\habHnnd} (resp.\ \noti{\tabHpd{}}{h@\habHpd}) if \(\seqs{2n}\) and, in the case \(n\geq1\), moreover \(\seqsc{2(n-1)}\) is \tHnnd{} (resp.\ \tHpd{}).
 \eenui
 If \(m\in\NO\), then the symbol \symba{\Fggqu{m}}{f} (resp.\ \symba{\Fgqu{m}}{f}) stands for the set of all \tabHnnd{} (resp.\ \tabHpd{}) sequences \(\seqs{m}\) from \(\Cqq\).
\edefn

 In view of \rdefn{D1159}, now \rthmss{T12}{T1.2*} can be summarized and reformulated as follows:

\begin{thm}\label{LHP}
 Let \(\ug\in\R\), let \(\og\in(\ug,\infp)\), let \(m\in\NO \), and let \(\seqs{m}\) be a sequence from \(\Cqq\).
 Then \(\MggqFs{m}\ne \emptyset\) if and only if \(\seqs{m} \in \Fggqu{m}\).
\end{thm}

\bdefnl{D1609}
 Let \(\ug\in\R\), let \(\og\in(\ug,\infp)\), let \(m\in\NO\), and let \(\seqs{m}\) be a sequence from \(\Cqq\).
 Then \(\seqs{m}\) is called \noti{\tabHnnde{}}{h@\habHnnde} (resp.\ \noti{\tabHpde{}}{h@\habHpde}) if there exists a matrix \(\su{m+1}\in\Cqq\) such that \(\seqs{m+1}\) is \tabHnnd{} (resp.\ \tabHpd{}).
 We denote by \symba{\Fggequ{m}}{f} (resp.\ \symba{\Fgequ{m}}{f}) the set of all \tabHnnde{} (resp.\ \tabHpde{}) sequences \(\seqs{m}\) from \(\Cqq\).
\edefn

 Using \rthm{LHP}, we derive now several algebraic results on the matrix sequences introduced in \rdefnss{D1159}{D1609}, resp.
 
\bpropl{P1750}
 Let \(\ug\in\R\), let \(\og\in(\ug,\infp)\), let \(m\in\NO\), and let  \(\seqs{m}\in\Fggqu{m}\).
 Then  \(\seqs{\ell} \in  \Fggqu{\ell}\) for all \(\ell\in\mn{0}{m}\).
\eprop
\bproof
 In view of \rthm{LHP} we have \(\MggqFs{m}\neq\emptyset\).
 Let \(\ell\in\mn{0}{m}\).
 Then \(\MggqFs{m}\subseteq\MggqFs{\ell}\).
 Hence \(\MggqFs{\ell}\neq\emptyset\).
 Thus, again applying \rthm{LHP}, we get \(\seqs{\ell} \in  \Fggqu{\ell}\).
\eproof

\bcorl{C1754}
 Let \(\ug\in\R\), \(\og\in(\ug,\infp)\), and \(m\in\NO\).
 Then \(\Fggequ{m}\subseteq\Fggqu{m}\).
\ecor
\bproof
 Use \rprop{P1750}.
\eproof

 \rprop{P1750} leads us to the following notions:
\bdefnl{D0846}
 Let \(\ug\in\R\), let \(\og\in(\ug,\infp)\), and let \(\seqsinf\) be a sequence from \(\Cqq\).
 Then \(\seqsinf\) is called \noti{\tabHnnd{}}{\habHnnd{}} (resp.\ \noti{\tabHpd{}}{\habHpd{}}) if for all \(m\in\NO\) the sequence \(\seqs{m}\) is \tabHnnd{} (resp.\ \tabHpd{}).
 The notation \symba{\Fggqinf}{f} (resp.\ \symba{\Fgqinf}{f}) stands for set of all \tabHnnd{} (resp.\ \tabHpd{}) sequences \(\seqsinf\) from \(\Cqq\).
\edefn

\bleml{L0703}
 Let \(\ug\in\R\), let \(\og\in(\ug,\infp)\), and let \(\sigma\in\MggqF\).
 Then \(\sigma\in\MggquF{\infi}\) and the sequence \(\seq{\suo{j}{\sigma}}{j}{0}{\infi}\) given via \eqref{mom} belongs to \(\Fggqinf\).
\elem
\bproof
 By the choice of \(\sigma\) we have \(\sigma\in\bigcap_{m=0}^\infi\MggquF{m}\).
 Thus, \rthm{LHP} implies \(\seq{\suo{j}{\sigma}}{j}{0}{m}\in\Fggqu{m}\) for each \(m\in\NO\).
 Hence \(\seq{\suo{j}{\sigma}}{j}{0}{\infi}\in\Fggqinf\).
\eproof

\bpropl{P0710}
 Let \(\ug\in\R\), let \(\og\in(\ug,\infp)\), let \(m\in\NO\), and let \(\seqs{m}\in\Fggqu{m}\).
 Then there is a sequence \(\seq{\su{j}}{j}{m+1}{\infi}\) from \(\Cqq\) such that \(\seqsinf\in\Fggqinf\).
\eprop
\bproof
 In view of \rthm{LHP} we have \(\MggqFs{m}\neq\emptyset\).
 Let \(\sigma\in\MggqFs{m}\).
 In view of \rlem{L0703}, we have then \(\sigma\in\MggquF{\infi}\), \(\seq{\suo{j}{\sigma}}{j}{0}{\infi}\in\Fggqinf\), and \(\seq{\suo{j}{\sigma}}{j}{0}{m}=\seqs{m}\).
\eproof

\bthml{T0718}
 Let \(\ug\in\R\), let \(\og\in(\ug,\infp)\), and let \(m\in\NO\).
 Then \(\Fggequ{m}=\Fggqu{m}\).
\ethm
\bproof
 Combine \rcor{C1754} and \rpropss{P0710}{P1750}.
\eproof
 Comparing \rthm{T0718} with \rrem{R1451}, we see that both statements are completely different.
 Against to the background of \rthmss{LHP}{T1.2} we can now immediately see the reason for this phenomenon, namely, in view of \eqref{ab8}, we have \(\MggqF=\MggquF{\infi}\), whereas on the other hand it can be easily checked that, for each \(k\in\N\), the proper inclusion \(\MggquR{k+1}\subset\MggquR{k}\) is satisfied.

 In view of \rlem{L0703}, we will consider the following problem:
\begin{description}
 \item[\pproblem{\ab }{m}{=}] Let \(\ug\in\R\), let \(\og\in(\ug,\infp)\), let \(m\in\NO \), and let \(\seqs{m}\) be a sequence from \(\Cqq\).
 Describe the set\index{p@\(\nextmom{m}\defeq\setaa{\int_{\ab} x^{m+1}\sigma(\dif x)}{\sigma\in\MggqFs{m}}\)}
\beql{SGG}
 \Nextmom{m}
 \defeq\setaa*{\int_{\ab} x^{m+1}\sigma(\dif x)}{\sigma\in\MggqFS{m}}.
\eeq
\end{description}

\bpropl{P0756}
 Let \(\ug\in\R\), let \(\og\in(\ug,\infp)\), let \(m\in\NO \), and let \(\seqs{m}\) be a sequence from \(\Cqq\). Then \(\nextmom{m}\neq\emptyset\) if and only if \(\seqs{m}\in\Fggqu{m}\). In this case, \(\nextmom{m}=\setaa{\su{m+1}\in\Cqq}{\seqs{m+1}\in\Fggqu{m+1}}\).
\eprop
\bproof
 Let \(\mathcal{P}\defeq\nextmom{m}\).
 First Assume \(\mathcal{P}\ne\emptyset\).
 Let \(s\in \mathcal{P}\).
 Then there is a \(\sigma\in\MggqFs{m}\) such that \(\int_{\ab}x^{m+1}\sigma(\dif x) = s\).
 In particular, \(\MggqFs{m}\neq\emptyset\).
 Thus, \(\seqs{m}\in\Fggqu{m}\) according to \rthm{LHP}.
 Furthermore, setting \(\su{m+1} \defeq  s\), we have  \(\sigma\in\MggqFs{m+1}\).
 Again, \rthm{LHP} yields \(\seqs{m+1}\in\Fggqu{m+1}\).

 Now suppose that \(\seqs{m}\in\Fggqu{m}\).
 Then \(\MggqFs{m}\ne \emptyset\) by virtue of \rthm{LHP}.
 Let \(\sigma\in\MggqFs{m}\).
 In view of \rlem{L0703}, then \(\int_{\ab} x^{m+1}\sigma(\dif x)\in \mathcal{P}\) follows.
 In particular, \(\mathcal{P}\neq\emptyset\).
 Finally, we consider an arbitrary \(\su{m+1}\in\Cqq\) with \(\seqs{m+1}\in\Fggqu{m+1}\).
 Using \rthm{LHP}, we get  \(\MggqFs{m+1}\ne \emptyset\).
 Let \(\sigma\in\MggqFs{m+1}\).
 Then \(\sigma\in\MggqFs{m}\) and \(\su{m+1}=\int_{\ab} x^{m+1}\sigma(\dif x)\).
 Consequently, \(\su{m+1}\in\mathcal{P}\).
\eproof

 \rprop{P0756} leads us to the problem of describing the set \(\setaa{\su{m+1}\in\Cqq}{\seqs{m+1}\in\Fggqu{m+1}}\).

\begin{rem}\label{R5.2.}
 Let  \(\alpha,\beta\in \C\), let \(\kappa\in\Ninf\), and let \(\seqska \) be a sequence from \(\Cpq\).
 Then
 \begin{align*}
  \ba  s_j&= \sau{j} + \sbu{j},&
 &\text{and}&
  \ba \su{j+1}&= \beta \sau{j} + \alpha \sbu{j}, 
 \end{align*}
 for all \(j\in\mn{0}{\kappa -1}\).
 Furthermore, if \(\kappa \ge 2\), then
 \begin{align*}
  \scu{j}&= -\alpha \sbu{j} + \sbu{j+1}=\rk*{\fourIdx{}{\ug}{}{}{\rk{\fourIdx{}{}{}{\og}{s}}}}_j,&
  \scu{j}&= \beta \sau{j} - \sau{j+1}=\rk*{\fourIdx{}{}{}{\og}{\rk{\fourIdx{}{\ug}{}{}{s}}}}_j,
 \end{align*}
 and \(\ba \su{j+2} = \beta^2 \sau{j} +\alpha^2 \sbu{j} - \ba  \scu{j}\) for all \(j\in\mn{0}{\kappa-2}\).
\end{rem}

\begin{rem}\label{R1542}
 Let  \(\alpha,\beta\in \C\), let \(\kappa\in\Ninf\), and let \(\seqska \) be a sequence from \(\Cpq\).
 Then \(\Hau{n}=-\ug\Hu{n}+\Ku{n}\) and \(\Hbu{n}= \og\Hu{n} -\Ku{n}\) for all \(n\in\NO\) with \(2n+1\leq\kappa\).
 Furthermore, if \(\kappa \ge 2\), then \(\Hcu{n} = -\ug\og\Hu{n}+\rk{\ug+\og}\Ku{n}-\Gu{n}\) for all \(n\in\NO\) with \(2n+2\leq\kappa\).
\end{rem}
 
 Using \rrem{R1542}, we obtain:
\breml{R1539}
 Let  \(\alpha,\beta\in \C\), let \(\kappa\in\Ninf\), and let \(\seqska \) be a sequence from \(\Cpq\).
 Then
 \(
  \ba\Hu{n}=\Hau{n}+\Hbu{n}
 \) and \(
  \ba\Ku{n}= \og\Hau{n} + \ug\Hbu{n}
 \)
 for all \(n\in\NO\) with \(2n+1\leq\kappa\).
 Furthermore, if \(\kappa \ge 2\), then
 \begin{align*}
  \Hcu{n}&= -\ug\Hbu{n} +\Kbu{n}=\rk*{\fourIdx{}{\ug}{}{}{\rk{\fourIdx{}{}{}{\og}{H}}}}_n,\\
  \Hcu{n}&= \og\Hau{n}-\Kau{n}=\rk*{\fourIdx{}{}{}{\og}{\rk{\fourIdx{}{\ug}{}{}{H}}}}_n,
 \end{align*}
 and \(\ba\Gu{n}= \og^2\Hau{n}+\ug^2\Hbu{n}- \ba\Hcu{n}\) for all \(n\in\NO\) with \(2n+2\leq\kappa\).
\erem

\bleml{L1307}
 Let \(\ug\in\R\), let \(\og\in(\ug,\infp)\), let \(n\in\NO\), and let \(\seqs{2n+1}\) be a sequence from \(\Cqq\). If \(\seqs{2n+1}\in\Fggqu{2n+1}\) (resp.\ \(\seqs{2n+1}\in\Fgqu{2n+1}\)), then \(\seqs{2n}\in\Hggqu{2n}\) (resp.\ \(\seqs{2n}\in\Hgqu{2n}\)).
\elem
\begin{proof}
 By virtue of \rrem{R1539}, we have \(\Hu{n}=\ba^\inv\ek{\Hau{n}+\Hbu{n}}\). In view of \rdefnp{D1159}{D1159.a} and \rdefn{D1238}, the assertions follow.
\end{proof}

\bleml{L0923}
 Let \(\ug\in\R\) and \(\og\in(\ug,\infp)\).
 Then \(\Fggqinf\subseteq\Hggqinf\) and \(\Fgqinf\subseteq\Hgqinf\).
\elem
\begin{proof}
 This follows in view of \rdefn{D0846} from \rlem{L1307}.
\end{proof}

\bleml{L1319}
 Let \(\ug\in\R\), let \(\og\in(\ug,\infp)\), let \(m\in\NO\), and let \(\seqs{m}\in\Fggqu{m}\).
 Then \(\su{j}\in\CHq\) for all \(j\in\mn{0}{m}\) and \(\su{2k}\in\Cggq\) for all \(k\in\NO\) with \(2k\leq m\).
\elem
\bproof
 In the case \(m=2n\) with some \(n\in\NO\), the assertion follows from \rdefn{D1159} and \rremp{L1251}{L1251.a}.
 
 Now suppose that \(m=2n+1\) for some \(n\in\NO\).
 By virtue of \rlem{L1307}, then \(\seqs{2n}\in\Hggqu{2n}\).
 According to \rremp{L1251}{L1251.a}, we obtain $\su{j}\in\CHq\) for all \(j\in\mn{0}{2n}\) and \(\su{2k}\in\Cggq\) for all \(k\in\mn{0}{n}\).
 In view of \rdefnp{D1159}{D1159.a}, furthermore \(\seqsa{2n}\in\Hggqu{2n}\).
 Because of \rremp{L1251}{L1251.a}, in particular \(\sau{2n}\in\CHq\).
 Since the matrix \(\su{2n+1}\) admits the representation \(\su{2n+1}=\sau{2n}+\ug\su{2n}\), we see that \(\su{2n+1}\in\CHq\) holds true as well.
\eproof

\bleml{R0936}
 Let \(\kappa\in\NOinf\), let \(\seqska\) be a sequence from \(\Cpq\), and let the sequence \(\seqr{\kappa}\) be given by \(r_j\defeq(-1)^j\su{j}\).
 Then \(\Huo{n}{r}=\Jp{n}\Hu{n}\Jq{n}\) for all \(n\in\NO\) with \(2n\leq\kappa\) and \(\Kuo{n}{r}=-\Jp{n}\Ku{n}\Jq{n}\) for all \(n\in\NO\) with \(2n+1\leq\kappa\) and \(\Guo{n}{r}=\Jp{n}\Gu{n}\Jq{n}\) for all \(n\in\NO\) with \(2n+2\leq\kappa\).
\elem
\bproof
 Apply \rlemp{L0937}{L0937.a}.
\eproof

\bleml{L1338}
 Let \(\kappa\in\NOinf\), let \(\seqska\) be a sequence from \(\Cpq\), and let the sequence \(\seqr{\kappa}\) be given by \(r_j\defeq(-1)^j\su{j}\).
 Then:
 \benui
  \il{L1338.a} \(\Tripuo{n}{r}=\Trip{n}\) for all \(n\in\NO\) with \(2n-1\leq\kappa\).
  \il{L1338.b} \(\Muo{n}{r}=-\Mu{n}\) and \(\Luo{n}{r}=\Lu{n}\) for all \(n\in\NO\) with \(2n\leq\kappa\).
 \eenui
\elem
\bproof
 In view of \rnota{N41}, we have \(\Tripuo{0}{r}=\Opq=\Trip{0}\) and \(\Muo{0}{r}=\Opq=-\Mu{0}\).
 
 Suppose that \(n\in\N\) with \(2n-1\leq\kappa\).
 We have \(\zuuo{n}{2n-1}{r}=(-1)^n\zuu{n}{2n-1}\Jq{n-1}\) and \(\yuuo{n}{2n-1}{r}=(-1)^n\Jp{n-1}\yuu{n}{2n-1}\).
 \rlem{R0936} and \rremss{R0950}{MA.7.3.} yield \(\rk{\Huo{n-1}{r}}^\mpi=\rk{\Jp{n-1}\Hu{n-1}\Jq{n-1}}^\mpi=\Jq{n-1}\rk{\Hu{n-1}}^\mpi\Jp{n-1}\).
 In view of \rrem{R0950}, we obtain \(\Tripuo{n}{r}=(-1)^{2n}\zuu{n}{2n-1}\Hu{n-1}^\mpi\yuu{n}{2n-1}=\Trip{n}\).
 
 Now suppose that \(n\in\N\) with \(2n\leq\kappa\).
 We have then furthermore \(\yuuo{n+1}{2n}{r}=(-1)^{n+1}\Jp{n-1}\yuu{n+1}{2n}\).
 Using \rrem{R0950}, we obtain
 \(\Muo{n}{r}=(-1)^{2n+1}\zuu{n}{2n-1}\Hu{n-1}^\mpi\yuu{n+1}{2n}=-\Mu{n}\).
 The rest follows from~\eqref{L1338.a}.
\eproof

\bleml{L1229}
 Let \(\ug,\og\in\C\), let \(m\in\N\), and let \(\seqs{m}\) be a sequence from \(\Cpq\).
 Let the sequence \(\seqr{m}\) be given by \(r_j\defeq(-1)^j\su{j}\).
 \benui
  \il{L1229.a} Let \(m=2n+1\) with some \(n\in\NO\).
 Then
  \begin{align*}
   -\rk{-\og}\Huo{n}{r}+\Kuo{n}{r}
   &=\Jp{n}^\ad\Hbu{n}\Jq{n}&
   &\text{and}&
   \rk{-\ug}\Huo{n}{r}-\Kuo{n}{r}
   &=\Jp{n}^\ad\Hau{n}\Jq{n}.
  \end{align*}
  \il{L1229.b} Let \(m=2n\) with some \(n\in\N\).
 Then
  \[
   -\rk{-\og}\rk{-\ug}\Huo{n-1}{r}+\ek*{\rk{-\og}+\rk{-\ug}}\Kuo{n-1}{r}-\Guo{n-1}{r}\\
   =\Jp{n-1}^\ad\Hcu{n-1}\Jq{n-1}.
  \]
 \eenui
\elem
\bproof
 \eqref{L1229.a} According to \rlem{R0936} and \rrem{R0950}, we have \(\Huo{n}{r}=\Jp{n}^\ad\Hu{n}\Jq{n}\) and \(\Kuo{n}{r}=-\Jp{n}^\ad\Ku{n}\Jq{n}\).
 Hence,
 \(
  -\rk{-\og}\Huo{n}{r}+\Kuo{n}{r}
  =\og \Jp{n}^\ad\Hu{n}\Jq{n}-\Jp{n}^\ad\Ku{n}\Jq{n}
  =\Jp{n}^\ad\Hbu{n}\Jq{n}
 \)
 and, analogously, 
 \(
  \rk{-\ug}\Huo{n}{r}-\Kuo{n}{r}
  =\Jp{n}^\ad\Hau{n}\Jq{n}
 \).
 
 \eqref{L1229.b}  According to \rlem{R0936} and \rrem{R0950}, we have \(\Huo{n-1}{r}=\Jp{n-1}^\ad\Hu{n-1}\Jq{n-1}\), \(\Kuo{n-1}{r}=-\Jp{n-1}^\ad\Ku{n-1}\Jq{n-1}\), and \(\Guo{n-1}{r}=\Jp{n-1}^\ad\Gu{n-1}\Jq{n-1}\).
 Hence,
 \begin{multline*}
  -\rk{-\og}\rk{-\ug}\Huo{n-1}{r}+\ek*{\rk{-\og}+\rk{-\ug}}\Kuo{n-1}{r}-\Guo{n-1}{r}\\
  =-\og\ug \Jp{n-1}^\ad\Hu{n-1}\Jq{n-1}+\rk{\og+\ug}\Jp{n-1}^\ad\Ku{n-1}\Jq{n-1}-\Jp{n-1}^\ad\Gu{n-1}\Jq{n-1}\\
  =\Jp{n-1}^\ad\Hcu{n-1}\Jq{n-1}.\qedhere
 \end{multline*}
\eproof

 The following technical result is useful for our further purposes.
\bleml{L1039}
 Let \(\ug\in\R\), let \(\og\in(\ug,\infp)\), let \(m\in\NO\), and let \(\seqs{m}\) be a sequence from \(\Cqq\).
 Let the sequence \(\seqr{m}\) be given by \(r_j\defeq(-1)^j\su{j}\).
 Then:
 \benui
  \il{L1039.a} \(\seqs{m}\in\Fggqu{m}\) if and only if \(\seqr{m}\in\Fgguuuu{q}{m}{-\og}{-\ug}\).
  \il{L1039.b} \(\seqs{m}\in\Fgqu{m}\) if and only if \(\seqr{m}\in\Fguuuu{q}{m}{-\og}{-\ug}\).
  \il{L1039.c} \(\seqs{m}\in\Fggequ{m}\) if and only if \(\seqr{m}\in\Fggeuuuu{q}{m}{-\og}{-\ug}\).
  \il{L1039.d} \(\seqs{m}\in\Fgequ{m}\) if and only if \(\seqr{m}\in\Fgeuuuu{q}{m}{-\og}{-\ug}\).
 \eenui
\elem
\bproof
 \eqref{L1039.a} The case \(m=0\) is trivial.
 Assume that \(m=2n+1\) with some \(n\in\NO\).
 If \(\seqs{2n+1}\in\Fggqu{2n+1}\), then \(\set{\Hau{n},\Hbu{n}}\subseteq\Cggo{(n+1)q}\).
 Hence, \(\set{-\rk{-\og}\Huo{n}{r}+\Kuo{n}{r},\rk{-\ug}\Huo{n}{r}-\Kuo{n}{r}}\subseteq\Cggo{(n+1)q}\) follows from \rlemp{L1229}{L1229.a}.
 Thus, \(\seqr{2n+1}\in\Fgguuuu{q}{2n+1}{-\og}{-\ug}\).
 If \(\seqr{2n+1}\in\Fgguuuu{q}{2n+1}{-\og}{-\ug}\), then, in view of \(\seqs{2n+1}=\seq{(-1)^jr_j}{j}{0}{2n+1}\), we get from the already proved implication that \(\seqs{2n+1}\in\Fggqu{2n+1}\). 
 Assume that \(m=2n\) with some \(n\in\N\).
 If \(\seqs{2n}\in\Fggqu{2n}\), then the matrices \(\Hu{n}\) and \(\Hcu{n-1}\) are \tnnH{}.
 From \(\Huo{n}{r}=\Jq{n}^\ad\Hu{n}\Jq{n}\) and \rlemp{L1229}{L1229.b} we see that the matrices \(\Huo{n}{r}\) and \(-\rk{-\og}\rk{-\ug}\Huo{n-1}{r}+\ek*{\rk{-\og}+\rk{-\ug}}\Kuo{n-1}{r}-\Guo{n-1}{r}\) are \tnnH{}.
 Hence, \(\seqr{2n}\in\Fgguuuu{q}{2n}{-\og}{-\ug}\).
 If \(\seqr{2n}\in\Fgguuuu{q}{2n}{-\og}{-\ug}\), then, by \(\seqs{2n}=\seq{(-1)^jr_j}{j}{0}{2n}\), we get from the already proved implication that \(\seqs{2n}\in\Fggqu{2n}\).
 
 \eqref{L1039.b} This can be proved analogously to~\eqref{L1039.a} using \(\det\Jq{n}\neq0\).
 
 \eqref{L1039.c}--\eqref{L1039.d} The assertions of~\eqref{L1039.c} and~\eqref{L1039.d} follow immediately from~\eqref{L1039.a} and~\eqref{L1039.b}, resp.
\eproof

\section{On the \hcCDm{} associated with a sequence belonging to $\Fggqu{2n-1}$}\label{S1626}
 In our joint paper~\zita{MR2735313} with Yu.~M.~Dyukarev, we constructed explicitly a molecular solution of Problem~\mproblem{\ab}{2n-1}{=}.
 Now we are looking for an appropriate modification of our method used in~\zita{MR2735313}, which allows us to handle Problem~\mproblem{\ab}{2n}{=} in a similar way as in~\zita{MR2735313}.
 First we sketch some essential features of our construction of a distinguished molecular solution of~\mproblem{\ab}{2n-1}{=} in~\zita{MR2735313}.
 
\bthml{T1002}
 Let \(n\in\N\) and let \(\seqs{2n-1}\in\Hggequ{2n-1}\).
 Then:
 \benui
  \il{T1002.a} There is a unique sequence \(\seq{\su{k}}{k}{2n}{\infi}\) from \(\Cqq\) such that \(\seqsinf\) is a \tHd{} \tHnnd{} sequence of order \(n\).
  \il{T1002.b} The set \(\MggqRsg{\infi}\) contains exactly one element \(\cdm{n}\).
  \il{T1002.c} The measure \(\cdm{n}\) is molecular.
 \eenui
\ethm
\bproof
 From~\zitaa{MR2570113}{\cprop{2.38}} we get~\eqref{T1002.a}. 
 In view of~\eqref{T1002.a} the assertions of~\eqref{T1002.b} and~\eqref{T1002.c} follow from parts~(a) and~(b) of \cprop{4.9} in~\zita{MR2570113}, resp.
\eproof

 If \(n\in\N\) and if \(\seqs{2n-1}\in\Hggequ{2n-1}\), then the measure \symba{\cdm{n}}{s} given via \rthmp{T1002}{T1002.b} is called the \notii{\tHd{} \tnnH{} measure} (short \notii{\tCDm{}}) associated with \(\seqs{2n-1}\).
 (For additional information, we refer to~\zitaa{MR2570113}{\cchapss{4}{5}}.)
 In~\zita{MR2735313}, we found a molecular solution for Problem~\mproblem{\ab}{2n-1}{=} in the case of a sequence \(\seqs{2n-1}\in\Fggqu{2n-1}\). 
 This will be explained now in more detail.
\bthml{T1019}
 Let \(\ug\in\R\), let \(\og\in(\ug,\infp)\), let \(n\in\N\), and let \(\seqs{2n-1}\in\Fggqu{2n-1}\).
 Then \(\seqs{2n-1}\in\Hggequ{2n-1}\). Moreover, the \tCDm{} \(\cdm{n}\) associated with \(\seqs{2n-1}\) fulfills \(\cdma{n}{\R\setminus\ab}=\Oqq\) and the measure \(\cdablm{2n-1}\defeq\Rstr_\BorF\cdm{n}\) is molecular and satisfies \(\cdablm{2n-1}\in\MggqFs{2n-1}\).
\ethm
\bproof
 Because of~\zitaa{MR2735313}{\clem{2.21}}, the sequence \(\seqs{2n-1}\) belongs to \(\Hggequ{2n-1}\).
 In view of \(\seqs{2n-1}\in\Fggqu{2n-1}\), we see from formulas~(1.2) and~(1.3) in~\zita{MR2735313} that in~\zitaa{MR2735313}{\cthm{1.1}} we can particularly choose \(\Omega=\ab\).
 Then from~\zitaa{MR2735313}{Proof of \cthm{1.1}, \cpage{921}} we see that \(\cdma{n}{\R\setminus\ab}=\Oqq\).
 From \rthmp{T1002}{T1002.c} we see that \(\cdm{n}\) is molecular.
 In view of \(\cdma{n}{\R\setminus\ab}=\Oqq\) and \rthmp{T1002}{T1002.b}, we obtain the rest.
\eproof

\bpropl{P1050}
 Let \(\ug\in\R\), let \(\og\in(\ug,\infp)\), let \(n\in\NO\), and let \(\seqs{2n}\in\Fggequ{2n}\).
 Further, let \(\su{2n+1}\in\Cqq\) be such that \(\seqs{2n+1}\in\Fggqu{2n+1}\).
 Then \(\cdablm{2n+1}\) given in \rthm{T1019} belongs to \(\MggqFs{2n}\cap\MggqmolF\).
\eprop
\bproof
 Use \rthm{T1019}.
\eproof

\bcorl{C1301}
 Let \(\ug\in\R\), let \(\og\in(\ug,\infp)\), let \(n\in\NO\), and let \(\seqs{2n}\in\Fggequ{2n}\).
 Then \(\MggqFs{2n}\cap\MggqmolF\neq\emptyset\).
\ecor
\bproof
 In view of \rdefn{D1609}, we have \(\setaa{\su{2n+1}\in\Cqq}{\seqs{2n+1}\in\Fggqu{2n+1}}\neq\emptyset\).
 Now the assertion follows from \rprop{P1050}.
\eproof

 In the situation of \rprop{P1050}, we see that each matrix \(\su{2n+1}\in\Cqq\) such that \(\seqs{2n+1}\in\Fggqu{2n+1}\) generates a molecular solution of Problem~\mproblem{\ab}{2n}{=}.
 In order to get more information about this family of molecular measures, we are led to the problem of describing the set
 \(
  \setaa{\su{2n+1}\in\Cqq}{\seqs{2n+1}\in\Fggqu{2n+1}}
 \).
 This problem is of purely algebraic nature as well as our method to handle it. %
 We will carefully study the intrinsic structure of finite \tabHnnd{} sequences.
 More precisely, if
 \(m\in\N\) and if \(\seqs{m}\in\Fggqu{m}\), then we will show in \rsec{S0747} that, for each \(k\in\mn{1}{m}\), the matrix \(\su{k}\) belongs to some matrix interval \([\umg{k-1},\omg{k-1}]\), where \(\umg{k-1}\) and \(\omg{k-1}\) will be explicitly computed in terms of the sequence \(\seqs{k-1}\).
 Particular attention is drawn to the case that \(\su{k}\) coincides with one of the interval ends \(\umg{k-1}\) or \(\omg{k-1}\).

\section{On \hHnnde{} sequences}\label{S1105}
 The class of \tabHnnd{} sequences stands in the center of this paper.
 If we look back at \rdefn{D1159}, where we have introduced it, then we see immediately that such a sequence is determined by the interplay of two \tHnnd{} sequences.
 Hence, our further considerations will be mainly based on the theory of \tHnnd{} sequences as well as the theory of \tHnnde{} sequences.
 For this reason, in this section we sketch some features of this theory, which turn out to be important for our further considerations.

\blemnl{\zitaa{MR3014199}{\clem{3.1(a)--(d)}}}{L1355}
 Let \(\kappa\in\NOinf\) and let \(\seqs{\kappa}\in\Hggequ{\kappa}\).
 Then:
 \benui
  \il{L1355.a} \(\su{j}\in\CHq\) for all \(j\in\mn{0}{\kappa}\) and \(\su{2k}\in\Cggq\) for all \(k\in\NO\) with \(2k\leq\kappa\).
  \il{L1355.b} \(\bigcup_{j=2k}^\kappa\ran{\su{j}}\subseteq\ran{\su{2k}}\) and \(\nul{\su{2k}}\subseteq\bigcap_{j=2k}^\kappa\nul{\su{j}}\) for all \(k\in\NO\) with \(2k\leq\kappa\).
 \eenui
\elem

\blemnl{\cf{}~\zitaa{MR2570113}{\clem{2.7}}}{L0920}
 Let \(n\in\N\).
 Then
 \[
  \Hggequ{2n}
  =\setaa*{\seqs{2n}\in\Hggqu{2n}}{\ran{\yuu{n+1}{2n}}\subseteq\ran{\Hu{n-1}}}.
 \]
\elem

\breml{L1211}
 For each \(n\in\NO\) we see from~\zitaa{MR2570113}{\cprop{2.24}} that \(\Hgequ{2n}=\Hgqu{2n}\).
\erem

\bpropnl{\cf{}~\zitaa{MR2570113}{\cprop{2.13}}}{P1333}
 Let \(n\in\N\).
 Then
 \[
  \Hggequ{2n}
  =\setaa*{\seqs{2n}\in\Hggqu{2n}}{\nul{\Lu{n-1}}\subseteq\nul{\Lu{n}}}.
 \]
\eprop

 \rprop{P1333} indicates an essential difference between the set \(\Hggequ{2n}\) and the set of finite \tnn{} definite sequences from \(\Cqq\).
 \rprop{P1333} tells us that there is an inclusion of the null spaces of the consecutive associated Schur complements introduced in \rnotap{N41}{N41.c}, whereas~\zitaa{MR885621}{\clem{6}} shows that the corresponding consecutive Schur complements associated with a \mbox{(Toeplitz-)non-negative} definite sequence from \(\Cqq\) are even decreasing with respect to the L\"owener semi-ordering in the set \(\CHq\).
 Clearly, this implies the inclusion for the corresponding null spaces.
 
\blemnl{\cf{}~\zitaa{MR2570113}{\clem{2.16}}}{L0913}
 Let \(n\in\NO\) and let \(\seqs{2n+1}\) be a sequence from \(\Cqq\).
 Then \(\seqs{2n+1}\in\Hggequ{2n+1}\) if and only if \(\seqs{2n}\in\Hggqu{2n}\), \(\su{2n+1}^\ad=\su{2n+1}\), and \(\ran{\yuu{n+1}{2n+1}}\subseteq\ran{\Hu{n}}\).
\elem

\breml{L1034}
 Let \(n\in\NO\) and let \(\seqs{2n+1}\) be a sequence from \(\Cqq\).
 Then from~\zitaa{MR2570113}{\cprop{2.24}} we see that \(\seqs{2n+1}\in\Hgequ{2n+1}\) if and only if \(\seqs{2n}\in\Hgqu{2n}\) and \(\su{2n+1}^\ad=\su{2n+1}\).
\erem

\bpropnl{\cf{}~\zitaa{MR2570113}{\cprop{2.22(a)}}}{P0906}
 Let \(n\in\NO\) and let \(\seqs{2n+1}\in\Hggequ{2n+1}\).
 Then
 \[
  \setaa*{\su{2n+2}\in\Cqq}{\seqs{2n+2}\in\Hggqu{2n+2}}
  =\setaa*{X\in\CHq}{X\lgeq\Trip{n+1}}.
 \]
\eprop

\bpropl{P1033}
 Let \(n\in\NO\) and let \(\seqs{2n+1}\in\Hgequ{2n+1}\).
 Then
 \[
  \setaa*{\su{2n+2}\in\Cqq}{\seqs{2n+2}\in\Hgqu{2n+2}}
  =\setaa*{X\in\CHq}{X>\Trip{n+1}}.
 \]
\eprop
\bproof
 This follows from \rrem{L1034} and~\zitaa{MR2570113}{\cprop{2.24}}.
\eproof

 We denote by \symba{\Dpqu{\kappa}}{d} the set of all sequences \(\seqska \) from \(\Cpq\) which fulfill \(\bigcup_{j=0}^\kappa \cR (s_j)\subseteq \cR (\su{0})\) and \(\cN (\su{0})\subseteq \bigcap_{j=0}^\kappa \cN (s_j)\).

\bpropnl{\zitaa{MR3014199}{\cprop{4.24}}}{R0957}
 Let \(m\in\NO\).
 Then \(\Hggequ{m}\subseteq\Dqqu{m}\).
\eprop

\section{Algebraic approach to \rprop{P1750} on the sections of \habHnnd{} sequences}\label{S1306}
 A main theme of our following considerations is to gain more insights into the intrinsic structure of \tabHnnd{} sequences.
 Such sequences are characterized by a particular interplay between some \tHnnd{} sequences.
 More precisely, we have to organize the interplay between several \tnnH{} block Hankel matrices.
 The difficulty of the situation results from the fact that the block Hankel matrix of step \(n\) is not a principle submatrix of the block Hankel matrix of stage \(n+1\).
 It will turn out (see \rlem{M8.1.18.}) that the parallel sum of matrices (see~\rapp{A1542}) will be the essential tool to overcome this difficulty.
 
 The construction of a purely algebraic approach to \rprop{P1750} stands in the center of this section.
 Our proof of \rprop{P1750} above is based on \rthm{LHP}, which gives a characterization of the solvability of the truncated \tabH{} moment problem.
 Since \rprop{P1750} is a statement of purely algebraic character, for our further purposes it is advantageous to have an independent algebraic proof as well.
 
 The key instrument of our purely algebraic approach to \rthm{T0718} is to find convenient expressions for the relevant block Hankel matrices in terms of parallel sums of matrices and applying \rlem{MA.8.7*.}.
 For every choice of complex \tpqa{matrices} \(A\) and \(B\), let \( A\ps  B\) be the \notii{parallel sum of \(A\) and \(B\)}, \ie{}, let \index{\(A\ps  B \defeq  A (A+B)^\mpi B\)}
\[
 A\ps B
 \defeq A(A+B)^\mpi B.
\]
 
\breml{R1528}
 Let \(\kappa\in\NOinf\) and let \(\seqska \) be a sequence from \(\Cpq\).
 For each \(n\in\NO \) with \(2n\le\kappa\), then the block Hankel matrix \(\Hu{n}\) admits the block representations \(\Hu{n} =\row (\yuu{k}{k+n})_{k=0}^n\) and \(\Hu{n} = \col (\zuu{j}{j+n})_{j=0}^n\).
\erem
 For all \(n\in\N\), let \(\Delta_{q,n}\defeq\tmatp{\Iu{nq}}{\Ouu{q}{nq}}\)\index{d@\(\Delta_{q,n}\defeq\tmatp{\Iu{nq}}{\Ouu{q}{nq}}\)} and \(\nabla_{q,n}\defeq\tmatp{\Ouu{q}{nq}}{\Iu{nq}}\)\index{n@\(\nabla_{q,n}\defeq\tmatp{\Ouu{q}{nq}}{\Iu{nq}}\)}.
 
 In~\cite[proof of Lemma~2.5]{Thi06}, the following result (formulated for the square case \(p=q\)) was already proved.
 Now we give an alternative proof.

\begin{lem}\label{M8.1.18.}
 Let \(n\in\N\) and let \(\seqs{2n+1}\)  be a sequence from \(\Cpq\) such that
\begin{align}\label{CMM}
 \cR (\yuu{n+1}{2n+1})&\subseteq \cR (\Hu{n})&
 &\text{and}&
 \cN (\Hu{n})&\subseteq \cN (\zuu{n}{2n+1}).
\end{align}
 For every choice of complex numbers  \(\alpha\) and \(\beta\) with \(\alpha \ne\beta\), then 
\[%
 \Hcu{n-1}
 = \ba  \Delta_{p,n}^\ad  \ek*{ \Hau{n} \ps  \Hbu{n}} \Delta_{q,n}.
\]
\end{lem}
\begin{proof}
 Using \rrem{R1539}, we see that
\begin{align}\label{M8.1.9N}
\Hau{n}  &\ps  \Hbu{n} = \Hau{n} \ek*{ \Hau{n} + \Hbu{n}}^\mpi  \Hbu{n} = \frac{1}{\beta -\alpha} \Hau{n} \Hu{n}^\mpi  \Hbu{n}.
\end{align}
 Setting \(Y \defeq  \row (\yuu{k}{k+n})_{k=1}^n\) and \(Z \defeq  \col (\zuu{j}{j+n})_{j=1}^n\), from \rrem{R1528} we see that \(\Hu{n}\) and \(\Ku{n}\) can be represented via
\begin{align*}
 \Hu{n}&=\mat{\yuu{0}{n}, Y},&
 \Hu{n}&= \matp{\zuu{0}{n}}{Z},&
 \Ku{n}&=\mat{Y, \yuu{n+1}{2n+1}},&
 &\text{and}&
 \Ku{n}&=\matp{Z}{\zuu{n+1}{2n+1}}.
\end{align*}
 Thus, \eqref{CMM} implies \(\cR (\Ku{n}) \subseteq \cR (\Hu{n})\) and \(\cN (\Hu{n}) \subseteq \cN (\Ku{n})\).
 Consequently, \rremss{MA.7.8.}{A.7.8-1.}  provide us \(\Hu{n} \Hu{n}^\mpi  \Ku{n} = \Ku{n}\) and \(\Ku{n} \Hu{n}^\mpi  \Hu{n} = \Ku{n}\).
 Hence, in view of \rrem{R1542}, we get then
\begin{align}\label{ST1022N}
\Hau{n}  & \Hu{n}^\mpi  \Hbu{n} = (-\alpha \Hu{n} + \Ku{n})  \Hu{n}^\mpi  (\beta \Hu{n} - \Ku{n})\cr
& = -\alpha\beta \Hu{n} \Hu{n}^\mpi  \Hu{n} + \alpha \Hu{n} \Hu{n}^\mpi  \Ku{n} + \beta \Ku{n} \Hu{n}^\mpi  \Hu{n} - \Ku{n}\Hu{n}^\mpi  \Ku{n}\cr
&= -\alpha\beta \Hu{n} + (\alpha + \beta) \Ku{n} - \Ku{n} \Hu{n}^\mpi  \Ku{n}.
\end{align}
 Obviously, \(\Delta_{p,n}^\ad  \Ku{n} = Z = \nabla_{p,n}^\ad  \Hu{n}\) and \(\Ku{n} \Delta_{q,n} = Y = \Hu{n}\nabla_{q,n}\) and, consequently,
\begin{equation}\label{ZI}
\Delta_{p,n}^\ad  \Ku{n} \Hu{n}^\mpi  \Ku{n} \Delta_{q,n}
= \nabla_{p,n}^\ad  \Hu{n} \Hu{n}^\mpi  \Hu{n}\nabla_{q,n} 
= \nabla_{p,n}^\ad  \Hu{n} \nabla_{q,n} 
= \Gu{n-1}.
\end{equation}
 Furthermore, \rrem{R3240} yields \(\Delta_{p,n}^\ad  \Hu{n} \Delta_{q,n} =  \Hu{n-1}\) and \(\Delta_{p,n}^\ad  \Ku{n} \Delta_{q,n} = \Ku{n-1}\).
 Using the last two equations, \eqref{ZI}, \eqref{ST1022N}, and \eqref{M8.1.9N}, we get then
\[\begin{split}
 \Hcu{n-1}
 &=-\alpha\beta \Hu{n-1} + (\alpha +\beta) \Ku{n-1} - \Gu{n-1}\\
 &=\Delta_{p,n}^\ad  \ek*{ -\alpha\beta \Hu{n} + (\alpha +\beta)\Ku{n} - \Ku{n}\Hu{n}^\mpi  \Ku{n} }  \Delta_{q,n}\\
 &=\Delta_{p,n}^\ad  \Hau{n} \Hu{n}^\mpi  \Hbu{n}  \Delta_{q,n}
 =\ba \Delta_{p,n}^\ad \ek*{  \Hau{n}  \ps   \Hbu{n} }\Delta_{q,n}.\qedhere
\end{split}\]
\end{proof}

\breml{R1601}
 Let \(\kappa\in\NOinf\) and let \(\seqska \) be a sequence from \(\Cpq\).
 For each \(n\in\N\) with \(2n\le\kappa\), then the block Hankel matrix \(\Hu{n}\) admits the block representations
\begin{align*}
\Hu{n}&= \begin{bmatrix}
\Hu{n-1} & \yuu{n}{2n-1}\\
\zuu{n}{2n-1} &\su{2n}
\end{bmatrix},&
&&
\Hu{n}&= \begin{bmatrix}
\yuu{0}{n} & \Ku{n-1}\\
s_n & \zuu{n+1}{2n}
\end{bmatrix},\\
\Hu{n}& = \begin{bmatrix}
z_{0,n-1} & s_n\\
\Ku{n-1} & \yuu{n+1}{2n} 
\end{bmatrix},&
&\text{and}&
\Hu{n}&= \begin{bmatrix}
\su{0} & \zuu{1}{n}\\
\yuu{1}{n} & \Gu{n-1}
\end{bmatrix}.
\end{align*}
\erem

\begin{lem}\label{M8.1.19.}
 Let \(n\in\NO \) and let  \(\seqs{2n+2}\)  be a sequence from \(\Cpq\).
 For every choice of \(\alpha\) and \(\beta\) in \(\C\), then 
\begin{align}
 \ba  \Hau{n}& = \rk{\nabla_{p,n+1} - \ko{\alpha} \Delta_{p,n+1}}^\ad     \Hu{n+1} \rk{\nabla_{q,n+1} -\alpha  \Delta_{q,n+1}} +\Hcu{n} \label{D1109}
\intertext{and}
\ba  \Hbu{n}&= \rk{ \ko{\beta} \Delta_{p,n+1} -  \nabla_{p,n+1}}^\ad     \Hu{n+1} \rk{ \beta \Delta_{q,n+1} - \nabla_{q,n+1}} + \Hcu{n} .\label{D1110}
\end{align}
\end{lem}
\begin{proof}
 Using \rrem{R1601}, we conclude \(\Delta_{p,n+1}^\ad  \Hu{n+1}  \Delta_{q,n+1} =    \Hu{n}\), \(\Delta_{p,n+1}^\ad  \Hu{n+1} \nabla_{q,n+1} = \Ku{n}\), \(\nabla_{p,n+1}^\ad  \Hu{n+1}\Delta_{q,n+1} =    \Ku{n}\), and  \(\nabla_{p,n+1}^\ad  \Hu{n+1} \nabla_{q,n+1} = G_n\).
 For each \(\alpha\in\C\) and each \(\beta\in\C\), this implies \eqref{D1109} and \eqref{D1110}.
\end{proof}
 
\begin{lem}\label{FR5.3.}
 Let \(\ug\in\R\), \(\og\in(\ug,\infp)\), and \(n\in\NO \).
 Then \(\Fggqu{2n+1} \subseteq \Hggequ{2n+1} \).
\end{lem}
\begin{proof}
 This can be seen from~\zitaa{MR2735313}{\clem{2.21}}.
\end{proof}

\breml{R1606}
 Let \(\ug\in\R\), let \(\og\in(\ug,\infp)\), and let \(n\in\NO\).
 Then, we see from \rlem{FR5.3.} and \rprop{R0957} that \(\Fggqu{2n+1}\subseteq\Dqqu{2n+1}\).
\end{rem}

\begin{prop}\label{M8.1.21.}
 Let \(\ug\in\R\), let \(\og\in(\ug,\infp)\), let \(m\in\NO\), and let  \(\seqs{m}\) be a sequence from \(\Cqq\).
 Then:
 \benui
  \il{M8.1.21.a} If \(\seqs{m}\in \Fggqu{m}\), then  \(\seqs{\ell} \in  \Fggqu{\ell}\) for each \(\ell\in\mn{0}{m}\).
  \il{M8.1.21.b} If \(\seqs{m}\in \Fgqu{m}\), then  \(\seqs{\ell} \in  \Fgqu{\ell}\) for each \(\ell\in\mn{0}{m}\).
 \eenui
\end{prop}
\begin{proof}
 The case \(m=0\) is trivial.
 Assume that \(m\geq1\) and \(\seqs{m}\in \Fggqu{m}\) (resp.\ \(\seqs{m}\in \Fgqu{m}\)).
 It is sufficient to prove that \(\seqs{ m  -1}\) belongs to  \(\Fggqu{ m -1}\) (resp.\ \(\Fgqu{ m -1}\)).
 
 First we suppose that there is an integer \(n\in\NO\) such that \( m  = 2n+1\).
 Then \rlem{L1307} yields that the matrix \(\Hu{n}\) is \tnnH{} (resp.\ \tpH{}).
 If \(n=0\), then \(\seqs{2n}\in\Fggqu{2n}\) (resp.\ \(\seqs{2n}\in\Fgqu{2n}\)) immediately follows from \rdefn{D1159}.
 Now we assume that \(n\ge 1\).
 In view of \rlem{L1319}, we have \(s_j^\ad  = s_j\) for all \( j\in \mn{0}{2n+1}\).
 The combination of \rlemss{FR5.3.}{L0913} provides us \(\cR (\yuu{n+1}{2n+1})\subseteq \cR (\Hu{n})\) and \(\cN (\Hu{n}) =\cN (\Hu{n}^\ad ) = \cR (\Hu{n})^\oc \subseteq \cR (\yuu{n+1}{2n+1})^\oc = \cN (\yuu{n+1}{2n+1}^\ad ) = \cN (\zuu{n+1}{2n+1})\).
 From \rlem{M8.1.18.} we get then 
\begin{equation}\label{D2003}
 \Hcu{n-1}
 = \ba  \Delta_{q,n}^\ad  \ek*{ \Hau{n} \ps\Hbu{n}} \Delta_{q,n}.
\end{equation}
 Because of  \(\seqs{2n+1}\in \Fggqu{2n+1}\) (resp.\ \(\seqs{2n+1}\in \Fgqu{2n+1}\)), \rdefnp{D1159}{D1159.a} shows that the matrices \(\Hau{n}\) and \(\Hbu{n}\) are both \tnnH{} (resp.\ \tpH{}).
 Thus, \rlem{MA.8.7*.} (resp.\ \rrem{L0751}) implies that the matrix \(\Hau{n}\ps  \Hbu{n}\) is \tnnH{} (resp.\ \tpH{}) as well.
 Because of \(\beta -\alpha > 0\) and \(\rank\Delta_{q,n}=nq\), then \eqref{D2003} und \rrem{R0705} yield that the matrix \(\Hcu{n-1}\) is \tnnH{} (resp.\ \tpH{}).
 Consequently, in view of \rdefnp{D1159}{D1159.b}, we see that \(\seqs{2n}\) belongs to \(\Fggqu{2n}\) (resp.\ \(\Fggqu{2n}\)).
 
 It remains to consider the case that \( m  =2n\) holds true with some \(n\in\N\).
 By virtue of  \(\seqs{2n} \in\Fggqu{2n}\) (resp.\ \(\seqs{2n} \in\Fgqu{2n}\)) and \rdefnp{D1159}{D1159.b}, the matrix \(\Hcu{n-1}\) is \tnnH{} (resp.\ \tpH{}).
 \rlem{M8.1.19.} and \(\beta>\alpha\) yield \(\Hau{n-1}\lgeq\ba^\inv\Hcu{n-1}\) and \(\Hbu{n-1}\lgeq\ba^\inv\Hcu{n-1}\). Hence, the matrices \(\Hau{n-1}\) and \(\Hbu{n-1}\) are both \tnnH{} (resp.\ \tpH{}).
 Thus, \rdefnp{D1159}{D1159.a} implies \(\seqs{2n-1} \in\Fggqu{2n-1}\) (resp.\ \(\seqs{2n-1} \in\Fgqu{2n-1}\)).
\end{proof}

\bcorl{B8.1.22.}
 Let \(\ug\in\R\), let \(\og\in(\ug,\infp)\), let \(\kappa\in\NOinf\), and let \(\seqs{\kappa}\) be a sequence from \(\Cqq\).
 Then:
 \benui
  \il{B8.1.22.a} If \(\seqs{\kappa}\in\Fggqu{\kappa}\), then \(\Lu{n}\in\Cggq \) for all \(n\in\NO \) with \(2n \le \kappa\), furthermore \(\set{\Lau{n}, \Lbu{n}}\subseteq \Cggq \) for all \(n\in\NO \) with \(2n+1 \le\kappa\), and \(\Lcu{n}\in\Cggq \) for all \(n\in\NO \) with \(2n+2\le \kappa\).
  \il{B8.1.22.b} If \(\seqs{\kappa}\in\Fgqu{\kappa}\), then \(\Lu{n}\in\Cgq \) for all \(n\in\NO \) with \(2n \le \kappa\), furthermore \(\set{\Lau{n}, \Lbu{n}}\subseteq \Cgq \) for all \(n\in\NO \) with \(2n+1 \le\kappa\), and \(\Lcu{n}\in\Cgq \) for all \(n\in\NO \) with \(2n+2\le \kappa\).
 \eenui
\ecor
\bproof
 Use \rprop{M8.1.21.}, \rdefnss{D1159}{D0846}, and \rrem{R1240}.
\eproof

\bcorl{C1342}
 Let \(\ug\in\R\), let \(\og\in(\ug,\infp)\), and let \(m\in\NO\). Then \(\Fggequ{m}\subseteq\Fggqu{m}\) and \(\Fgequ{m}\subseteq\Fgqu{m}\).
\ecor
\bproof
 Use \rprop{M8.1.21.}.
\eproof

\begin{prop}\label{S8.1.23.}
 Let \(\ug\in\R\), let \(\og\in(\ug,\infp)\), and let \(m\in\NO\).
 Then \(\Fggqu{m}\subseteq \Hggequ{m}\) and \(\Fgqu{m}\subseteq \Hgequ{m}\).
\end{prop}
\begin{proof}
 We first show \(\Fggqu{m}\subseteq \Hggequ{m}\).
 Because of \(\Hggequ{0} =\Hggqu{0}\), the case \(m=0\) is trivial. If \(m=2n+1\) for some \(n\in\NO \), then \rlem{FR5.3.} yields the assertion.
 
 Now let \(m=2n\) with some \(n\in\N\).
 We consider an arbitrary sequence  \(\seqs{2n}\in \Fggqu{2n}\).
 Then \rdefnp{D1159}{D1159.b} and \rrem{R1601} provide us \(\seqs{2n}\in\Hggqu{2n}\) and \(\set{\Hcu{n-1}, \Gu{n-1}}\subseteq \Cggo{nq}\).
 Thus, \rrem{A.20.} yields \(\ran{\Gu{n-1} +  \Hcu{n-1}} = \ran{\Gu{n-1}}+\ran{\Hcu{n-1}}\).
 According to \rpropp{M8.1.21.}{M8.1.21.a}, the sequence \(\seqs{2n-1}\) belongs to \(\Fggqu{2n-1}\).
 Consequently, \rdefnp{D1159}{D1159.a} yields \(\set{ \Hau{n-1}, \Hbu{n-1}} \subseteq \Cggo{nq}\).
 Hence, using \rrem{A.20.} again, we conclude \(\ran{\Hau{n-1} + \Hbu{n-1}}=\ran{\Hau{n-1}}+\ran{\Hbu{n-1}}\).
 Since \rrem{R1539} shows that \(\Hu{n-1}=\frac{1}{\og-\ug}\ek{\Hau{n-1}+\Hbu{n-1}}\) and \(\Gu{n-1} + \Hcu{n-1} = \frac{1}{\beta -\alpha}\ek{\beta^2 \Hau{n-1} + \alpha^2 \Hbu{n-1}}\), we get then
\[\begin{split}
 \ran{\Gu{n-1}}
 &\subseteq\ran{\Gu{n-1}}+\Ran{\Hcu{n-1}}\\
 &=\Ran{\beta^2 \Hau{n-1} + \alpha^2 \Hbu{n-1}}
 \subseteq\Ran{\beta^2 \Hau{n-1}} +\Ran{ \alpha^2 \Hbu{n-1}}\\
 &\subseteq\Ran{\Hau{n-1}}+\Ran{\Hbu{n-1}}
 =\ran{\Hu{n-1}}.
\end{split}\]
 Taking into account that \(\yuu{n+1}{2n}\) is the last \taaa{nq}{q}{block} column of \(\Gu{n-1}\), we infer \(\ran{\yuu{n+1}{2n}}\subseteq \ran{\Hu{n-1}}\).
 Thus, \rlem{L0920} yields \(\seqs{2n}\in\Hggequ{2n}\).
 
 Now we show \(\Fgqu{m}\subseteq \Hgequ{m}\).
 If \(m=2n\) with some \(n\in\NO\), this follows from \rrem{L1211} and \rdefnp{D1159}{D1159.b}.
 If \(m=2n+1\) with some \(n\in\NO\), this follows from \rrem{L1034} in view \rlemss{L1307}{L1319}.
\end{proof}

\section{On \hSnnd{} sequences}\label{S1410}
 In this section, we discuss two dual classes of finite sequences of complex \tqqa{matrices} which turn out to be intimately connected with the class \(\Fggqu{m}\).
 The first and second types of this classes will be associated with the left endpoint \(\ug\) and the right endpoint \(\og\) of the interval \(\ab\), resp.
 These classes were already studied by the authors in earlier work (see~\zitas{MR2570113,MR3014201,MR3133464}).
 First we consider the class of sequences connected with the left endpoint \(\ug\) of \(\ab\).
 
\bdefnl{D1413}
 Let \(\ug\in\R\) and let \(n\in\NO\).
 \benui
  \il{D1413.a} Let \(\seqs{2n+1}\) be a sequence from \(\Cqq\).
  Then \(\seqs{2n+1}\) is called \noti{\taSnnd{}}{s@\haSnnd} (resp.\ \noti{\taSpd{}}{s@\haSpd}) if both sequences \(\seqs{2n}\) and \(\seqsa{2n}\) are \tHnnd{} (resp.\ \tHpd{}).
  \il{D1413.b} Let \(\seqs{2n}\) be a sequence from \(\Cqq\).
  Then \(\seqs{2n}\) is called \noti{\taSnnd{}}{s@\haSnnd} (resp.\ \noti{\taSpd{}}{s@\haSpd}) if \(\seqs{2n}\) and, in the case \(n\geq1\), also \(\seqsa{2(n-1)}\) is \tHnnd{} (resp.\ \tHpd{}).
 \eenui
 If \(m\in\NO\), then, \symba{\Kggqu{m}}{k} (resp.\ \symba{\Kgqu{m}}{k}) stands for the set of all \taSnnd{} (resp.\ \taSpd{}) sequences \(\seqs{m}\) from \(\Cqq\).
\edefn

 Let \(\ug\in\R\), let \(m\in\NO\), and let \(\seqs{m}\) be a sequence from \(\Cqq\).
 Then \(\seqs{m}\) is called \noti{\taSnnde{}}{s@\haSnnde{}} (resp.\ \noti{\taSpde{}}{s@\haSpde{}}) if there exists a matrix \(\su{m+1}\in\Cqq\) such that \(\seqs{m+1}\in\Kggqu{m+1}\) (resp.\ \(\seqs{m+1}\in\Kgqu{m+1}\)).
 We denote by \symba{\Kggequ{m}}{k} (resp.\ \symba{\Kgequ{m}}{k}) the set of all \taSnnde{} (resp.\ \taSpde{}) sequences \(\seqs{m}\) from \(\Cqq\).
 The importance of the class \(\Kggequ{m}\) in the context of moment problems is caused by the following observation:
\bthmnl{\zitaa{MR2735313}{\cthm{1.3}}}{T1306}
 Let \(\ug\in\R\), \(m\in\NO\), and let \(\seqs{m}\) be a sequence from \(\Cqq\).
 Then \(\MggqKsg{m}\neq\emptyset\) if and only if \(\seqs{m}\in\Kggequ{m}\).
\ethm

\blemnl{\zitaa{MR3014201}{\clem{2.9(a)--(b)}}}{L1435}
 Let \(\ug\in\R\), let \(m\in\NO\), and let \(\seqs{m}\in\Kggqu{m}\).
 Then \(\su{j}\in\CHq\) for all \(j\in\mn{0}{m}\) and \(\su{2k}\in\Cggq\) for all \(k\in\NO\) with \(2k\leq m\).
 Furthermore, \(\sau{j}\in\CHq\) for all \(j\in\mn{0}{m-1}\) and \(\sau{2k}\in\Cggq\) for all \(k\in\NO\) with \(2k+1\leq m\).
\elem

\breml{R1447}
 Let \(\ug\in\R\) and let \(m\in\NO\).
 Then \(\Kggequ{m}\subseteq\Kggqu{m}\) and, in the case \(m\in\N\), furthermore \(\Kggequ{m}\neq\Kggqu{m}\).
\erem
 
\breml{R1058}
 Let \(\ug\in\R\), let \(m\in\N\), and let \(\seqs{m}\in\Kggqu{m}\).
 Then \(\seqs{\ell}\in\Kggequ{\ell}\) for all \(\ell\in\mn{0}{m-1}\).
\erem

\breml{R1059}
 Let \(\ug\in\R\), let \(m\in\NO\), and let \(\seqs{m}\in\Kggequ{m}\) (resp.\ \(\Kgqu{m}\)).
 Then \(\seqs{\ell}\in\Kggequ{\ell}\) (resp.\ \(\Kgqu{\ell}\)) for all \(\ell\in\mn{0}{\kappa}\).
\erem

 Let \(\ug\in\R\). A sequence \(\seqsinf\) from \(\Cqq\) is called \noti{\taSnnd{}}{\haSnnd{}} (resp.\ \noti{\taSpd{}}{\haSpd{}}) if \(\seqs{m}\in\Kggqu{m}\) (resp.\ \(\seqs{m}\in\Kgqu{m}\)) for all \(m\in\NO\).
 We denote by \symba{\Kggqinfa}{k} (resp.\ \symba{\Kgqinfa}{k}) the set of all \taSnnd{} (resp.\ \taSpd{}) sequences \(\seqsinf\) from \(\Cqq\).
 Furthermore, let \symba{\Kggeqinfa\defeq\Kggqinfa}{k} and let \symba{\Kgequ{\infi}\defeq\Kgqinfa}{k}.

\bleml{R1039}
 Let \(\ug\in\R\), let \(\kappa\in\NOinf\), and let \(\seqs{\kappa} \in\Kggqu{\kappa}\).
 Then \(\Lu{n}\in\Cggq \) for all \(n\in\NO \) with \(2n \le \kappa\), and \(\Lau{n}\in\Cggq \) for all \(n\in\NO \) with \(2n+1\le \kappa\).
\elem
\bproof
 Combine \rdefn{D1413}, \rremss{R1058}{R1447}, and \rremp{R1240}{R1240.a}.
\eproof

\bpropnl{\zitaa{MR3014201}{\cprop{2.19}}}{P1439}
 Let \(\ug\in\R\).
 Then:
 \benui
  \il{P1439.a} Let \(n\in\NO\) and let \(\seqs{2n+1}\in\Kggequ{2n+1}\), then
  \[
   \ran{\Lu{0}}\supseteq\Ran{\Lau{0}}\supseteq\ran{\Lu{1}}\supseteq\Ran{\Lau{1}}\supseteq\dotsb\supseteq\ran{\Lu{n}}\supseteq\Ran{\Lau{n}}
  \]
  and
  \[
   \nul{\Lu{0}}\subseteq\Nul{\Lau{0}}\subseteq\nul{\Lu{1}}\subseteq\Nul{\Lau{1}}\subseteq\dotsb\subseteq\nul{\Lu{n}}\subseteq\Nul{\Lau{n}}.
  \]
  \il{P1439.b} Let \(n\in\N\) and let \(\seqs{2n}\in\Kggequ{2n}\), then
  \[
   \ran{\Lu{0}}\supseteq\Ran{\Lau{0}}\supseteq\ran{\Lu{1}}\supseteq\Ran{\Lau{1}}\supseteq\dotsb\supseteq\Ran{\Lau{n-1}}\supseteq\ran{\Lu{n}}
  \]
  and
  \[
   \nul{\Lu{0}}\subseteq\Nul{\Lau{0}}\subseteq\nul{\Lu{1}}\subseteq\Nul{\Lau{1}}\subseteq\dotsb\subseteq\Nul{\Lau{n-1}}\subseteq\nul{\Lu{n}}.
  \]
 \eenui
\eprop

 Now we consider the constructions connected with the right endpoint \(\og\) of the interval \(\ab\).

\bdefnl{D1420}
 Let \(\og\in\R\) and let \(n\in\NO\).
 \benui
  \il{D1420.a} Let \(\seqs{2n+1}\) be a sequence from \(\Cqq\).
  Then \(\seqs{2n+1}\) is called \noti{\tbSnnd{}}{s@\hbSnnd} (resp.\ \noti{\tbSpd{}}{s@\hbSpd}) if both sequences \(\seqs{2n}\) and \(\seqsb{2n}\) are \tHnnd{} (resp.\ \tHpd{}).
  \il{D1420.b} Let \(\seqs{2n}\) be a sequence from \(\Cqq\).
  Then \(\seqs{2n}\) is called \noti{\tbSnnd{}}{s@\hbSnnd} (resp.\ \noti{\tbSpd{}}{s@\hbSpd}) if \(\seqs{2n}\) and, in the case \(n\geq1\), also \(\seqsb{2(n-1)}\) is \tHnnd{} (resp.\ \tHpd{}).
 \eenui
 If \(m\in\NO\), then \symba{\Lggqu{m}}{l} (resp.\ \symba{\Lgqu{m}}{l}) stands for the set of all \tbSnnd{} (resp.\ \tbSpd{}) sequences \(\seqs{m}\) from \(\Cqq\).
\edefn

 Let \(\og\in\R\), let \(m\in\NO\), and let \(\seqs{m}\) be a sequence from \(\Cqq\).
 Then \(\seqs{m}\) is called \noti{\tbSnnde{}}{s@\hbSnnde} (resp.\ \noti{\tbSpde{}}{s@\hbSpde{}}) if there exists a matrix \(\su{m+1}\in\Cqq\) such that \(\seqs{m+1}\in\Lggqu{m+1}\) (resp.\ \(\seqs{m+1}\in\Lgqu{m+1}\)).
 We denote by \symba{\Lggequ{m}}{l} (resp.\ \symba{\Lgequ{m}}{l}) the set of all \tbSnnde{} (resp.\ \tbSpde{}) sequences \(\seqs{m}\) from \(\Cqq\).
 
 A sequence \(\seqsinf\) from \(\Cqq\) is called \noti{\tbSnnd{}}{\hbSnnd{}} (resp.\ \noti{\tbSpd{}}{\hbSpd{}}) if \(\seqs{m}\in\Lggqu{m}\) (resp.\ \(\seqs{m}\in\Lgqu{m}\)) for all \(m\in\NO\).
 We denote by \symba{\Lggqinf}{l} (resp.\ \symba{\Lgqinf}{l}) the set of all \tbSnnd{} (resp.\ \tbSpd{}) sequences \(\seqsinf\) from \(\Cqq\).
 Furthermore, let \symba{\Lggequ{\infi}\defeq\Lggqinf}{l} and let \symba{\Lgequ{\infi}\defeq\Lgqinf}{l}.

\bleml{FRL4.26.}
 Let \(\og\in\R\), let \( m \in\NO\), and let \(\seqs{m} \) be a sequence from \(\Cqq\).
 Let the sequence \(\seqr{ m }\) be given by \(r_j \defeq  (-1)^j s_j\).
 Then:
 \benui
  \il{FRL4.26.a} \(\seqs{ m } \in \Lggqu{ m } \) if and only if \((r_j)_{j=0}^ m  \in \Kgguuu{q}{ m }{-\og}\).
  \il{FRL4.26.b} \(\seqs{ m } \in \Lggequ{ m } \) if and only if \((r_j)_{j=0}^ m  \in \Kggeuuu{q}{ m }{-\og}\).
  \il{FRL4.26.c} \(\seqs{ m } \in \Lgqu{ m } \) if and only if \((r_j)_{j=0}^ m  \in \Kguuu{q}{ m }{-\og}\).
  \il{FRL4.26.d} \(\seqs{ m } \in \Lgequ{ m } \) if and only if \((r_j)_{j=0}^ m  \in \Kgeuuu{q}{ m }{-\og}\).
 \eenui
\elem
\bproof
 This can be verified analogously to the proof of \rlem{L1039}.
\eproof

 The importance of the class \(\Lggequ{m}\) in the context of moment problems is caused by the following observation:
\bthml{T1328}
 Let \(\og\in\R\), let \(m\in\NO\), and let \(\seqs{m}\) be a sequence from \(\Cqq\).
 Then \(\MggqLsg{m}\neq\emptyset\) if and only if \(\seqs{m}\in\Lggequ{m}\).
\ethm
\bproof
 Combine \rlemp{FRL4.26.}{FRL4.26.b}, \rthm{T1306}, and \rrem{R0817}.
\eproof

\breml{R1449}
 Let \(\og\in\R\) and let \(m\in\NO\).
 Then \(\Lggequ{m}\subseteq\Lggqu{m}\) and, in the case \(m\in\N\), furthermore \(\Lggequ{m}\neq\Lggqu{m}\).
\erem

\breml{R1103}
 Let \(\og\in\R\), let \(m\in\N\), and let \(\seqs{m}\in\Lggqu{m}\).
 Then \(\seqs{\ell}\in\Lggequ{\ell}\) for all \(\ell\in\mn{0}{m-1}\).
\erem

\breml{R1104}
 Let \(\og\in\R\), let \(m\in\NO\), and let \(\seqs{m}\in\Lggequ{m}\) (resp.\ \(\Lgqu{m}\)).
 Then \(\seqs{\ell}\in\Lggequ{\ell}\) (resp.\ \(\Lgqu{\ell}\)) for all \(\ell\in\mn{0}{m}\).
\erem

\begin{rem}\label{R1043}
 Let \(\og\in\R\), let \(\kappa\in\NOinf\), and let \(\seqs{\kappa} \in\Lggqu{\kappa}\).
 In view of \rdefn{D1420}, \rremss{R1103}{R1449}, and \rremp{R1240}{R1240.a}, then \(\Lu{n}\in\Cggq \) for all \(n\in\NO \) with \(2n \le \kappa\), and \(\Lbu{n}\in\Cggq \) for all \(n\in\NO \) with \(2n+1\le \kappa\).
\end{rem}

\begin{rem}\label{FRL425}
 For each \(n\in\N\), it is readily checked that
\begin{align*}
\Kggequ{2n-1}& = \setaa*{\seqs{2n-1} \in \Hggequ{2n-1}}{(\sau{j})_{j=0}^{2(n-1)} \in \Hggqu{2(n-1)}},\\
\Lggequ{2n-1} &= \setaa*{\seqs{2n-1} \in \Hggequ{2n-1}}{(\sbu{j})_{j=0}^{2(n-1)} \in \Hggqu{2(n-1)}},\\
\Kggequ{2n} &= \setaa*{\seqs{2n} \in \Hggqu{2n}}{(\sau{j})_{j=0}^{2n-1} \in \Hggequ{2n-1}},\\
 \Lggequ{2n} &= \setaa*{\seqs{2n} \in \Hggqu{2n}}{(\sbu{j})_{j=0}^{2n-1} \in \Hggequ{2n-1}} .
\end{align*}
\end{rem}

\bpropl{L5.1.30.}
 Let \(\ug,\og\in\R\) and let \(n\in\N\).
 Then:
\benui
 \il{L5.1.30.a} \(\Kggequ{2n-1} = \setaa{\seqs{2n-1} \in \Kggqu{2n-1}}{\cN (\Lu{n-1}) \subseteq \nul{\Lau{n-1}}}\).
 \il{L5.1.30.b} \(\Lggequ{2n-1} = \setaa{\seqs{2n-1} \in \Lggqu{2n-1} }{ \cN (\Lu{n-1}) \subseteq \nul{\Lbu{n-1}}}\).
 \il{L5.1.30.c} \(\Kggequ{2n} = \setaa{\seqs{2n} \in \Kggqu{2n} }{ \nul{\Lau{n-1}} \subseteq  \cN (\Lu{n}) }\).
 \il{L5.1.30.d} \(\Lggequ{2n} = \setaa{\seqs{2n} \in \Lggqu{2n} }{ \nul{\Lbu{n-1}} \subseteq \cN (\Lu{n})}\).
\eenui
\eprop
\begin{proof}
 \rPartss{L5.1.30.a}{L5.1.30.c} are proved in~\cite[\clemss{4.15}{4.16}]{MR2735313}.

 \eqref{L5.1.30.b} Let \(r_j \defeq  (-1)^j s_j\) for all \(j\in\mn{0}{2n-1}\).
 Then \(s_j \defeq  (-1)^j r_j\) for each \(j\in\mn{0}{2n-1}\).
 First we assume \(\seqs{2n-1}\in\Lggequ{2n-1}\).
 \rlemp{FRL4.26.}{FRL4.26.b} provides us then \((r_j)_{j=0}^{2n-1}\in\Kggeuuu{q}{2n-1}{-\og}\).
 Thus, \rpart{L5.1.30.a} shows \((r_j)_{j=0}^{2n-1}\in\Kgguuu{q}{2n-1}{-\og}\) and 
\begin{equation}\label{N489N}
 \cN (\Lu{n-1}^{\langle r\rangle})
 \subseteq\Nul{(_{-\beta}L^{\langle r\rangle})_{n-1}}.
\end{equation}
 In particular, \rlemp{FRL4.26.}{FRL4.26.a} implies
\begin{equation}\label{MAI}
 \seqs{2n-1}
 \in\Lggqu{2n-1}.
\end{equation}
 Taking into account that 
\begin{equation}\label{DIR}
 \rk{\fourIdx{}{-\og}{}{}{r}}_j
 = -(-\beta) (-1)^j s_j + (-1)^{j+1}\su{j+1}
 = (-1)^j (\beta s_j -\su{j+1})
 = (-1)^j \sbu{j}
\end{equation}
 holds true for each \(j\in\mn{0}{2n-1}\), \rlemp{L1338}{L1338.b} implies \((_{-\beta} L^{\langle r\rangle})_{n-1} =\Lbu{n-1}\).
 Since \rlemp{L1338}{L1338.b} also yields \(\Lu{n-1}^{\langle r\rangle} = \Lu{n-1}\), from \eqref{N489N} we get then
\begin{equation}\label{MAG}
 \cN (\Lu{n-1})
 \subseteq\Nul{\Lbu{n-1}}.
\end{equation}
 Conversely, now suppose that \eqref{MAI} and \eqref{MAG} are valid.
 Because of \eqref{MAI} and \rlemp{FRL4.26.}{FRL4.26.a}, then \((r_j)_{j=0}^{2n-1}\) belongs to \(\Kgguuu{q}{2n-1}{-\og}\).
 From \rlemp{L1338}{L1338.b} we get again \(\Lu{n-1}^{\langle r\rangle} = \Lu{n-1}\) and, in view of \eqref{DIR}, the equation \(\rk{\fourIdx{}{-\og}{}{}{L^{\langle r\rangle}}}_{n-1} = \Lbu{n-1}\) as well.
 Thus, \eqref{MAG} implies \eqref{N489N}.
 Consequently, \rpart{L5.1.30.a} yields  \((r_j)_{j=0}^{2n-1}\in\Kggeuuu{q}{2n-1}{-\og}\).
 Hence, \rlemp{FRL4.26.}{FRL4.26.b}  shows  that \(\seqs{2n-1} \in\Lggequ{2n-1}\).
 
 \eqref{L5.1.30.d} Analogously, as \rpart{L5.1.30.b} can be proved by virtue of \rpart{L5.1.30.a}, \rpart{L5.1.30.d} can be checked using \rpart{L5.1.30.c}.
\end{proof}

 If \(\cM\) is a \tne{} subset of \(\Cq\), then let \symb{\cM^\oc} be the set of all \(x\in\Cq\) which fulfill \(\innerE{x}{y}=0\) for each \(y\in\cM\), where \symba{\innerE{\cdot}{\cdot}}{e} is the Euclidean inner product in \(\Cq\).

\bpropl{KL5.1.30.}
 Let \(\ug,\og\in\R\) and let \(n\in\N\).
 Then:
\benui
 \il{KL5.1.30.a} \(\Kggequ{2n-1} = \setaa{\seqs{2n-1} \in \Kggqu{2n-1}}{\ran{\Lau{n-1}}\subseteq \cR (\Lu{n-1})}\).
 \il{KL5.1.30.b} \(\Lggequ{2n-1} = \setaa{\seqs{2n-1} \in \Lggqu{2n-1}}{\ran{\Lbu{n-1}}\subseteq \cR (\Lu{n-1})}\).
 \il{KL5.1.30.c} \(\Kggequ{2n} = \setaa{\seqs{2n} \in \Kggqu{2n}}{\cR (\Lu{n}) \subseteq \ran{\Lau{n-1}}}\).
 \il{KL5.1.30.d} \(\Lggequ{2n} = \setaa{\seqs{2n} \in \Lggqu{2n}}{\cR (\Lu{n}) \subseteq \ran{\Lbu{n-1}}}\).
\eenui
\eprop
\begin{proof}
 In view of the Definition of the set \(\Kggqu{2n-1}\) and \rlem{R1039}, we see that, for each   \(\seqs{2n-1} \in \Kggqu{2n-1} \), the matrices \(\Lu{n-1}\) and \(\Lau{n-1}\) are both \tnnH{}.
 Thus, \(\cN (\Lu{n-1}) = \cN (\Lu{n-1}^\ad ) = \cR (\Lu{n-1})^\oc\) and \( \nul{\Lau{n-1}}  = \nul{\Lau{n-1}^\ad} =\ran{\Lau{n-1}}^\oc\).
 Consequently, the application of \rprop{L5.1.30.} completes the proof of \eqref{KL5.1.30.a}.
 \rPartsss{KL5.1.30.b}{KL5.1.30.c}{KL5.1.30.d} can be checked analogously.
\end{proof}

 Now we are going to study several aspects of the interplay between those four \tHnnd{} sequences which determine the sections of an \tabHnnd{} sequence.
 In particular, we derive formulas connecting the matrices which were introduced in \rpartss{N41.c}{N41.d} of \rnota{N41} for each of the sequences \(\seqs{\kappa}\), \(\seqsa{\kappa-1}\), \(\seqsb{\kappa-1}\), and \(\seqsc{\kappa-2}\).
 These formulas play an important role in the proof of \rthm{M8.1.43.}, which is of central importance for our further considerations.

\begin{lem}\label{FRL432.}
 Let \(\ug,\og\in\R\) and let \(n\in\N\).
 Then:
\benui
 \il{FRL432.a} If \(\seqs{2n} \in\Kggqu{2n}\), then 
\[
\Lu{n}
=    \sau{2n-1} - \Mau{n-1} - \Lau{n-1} \Lu{n-1}^\mpi  (\su{2n-1} - \Mu{n-1}).
\]
 \il{FRL432.b} If \(\seqs{2n} \in\Lggqu{2n}\), then 
\[
 \Lu{n}
 = - \ek*{  \sbu{2n-1} - \Mbu{n-1}} + \Lbu{n-1} \Lu{n-1}^\mpi  (\su{2n-1} - \Mu{n-1}).
\]
\eenui
\end{lem}
\begin{proof}
 A proof of \rpart{FRL432.a} is given in~\cite[Proposition~6.4]{MR3014201}.
 To prove \rpart{FRL432.b}, we consider an arbitrary sequence \(\seqs{2n}\in\Lggqu{2n}\).
 Setting \(r_j \defeq  (-1)^js_j\) for all \(j\in\mn{0}{2n}\), we see from \rlemp{FRL4.26.}{FRL4.26.a} that \((r_j)_{j=0}^{2n}\) belongs to \(\Kgguuu{q}{2n}{-\og}\).
 Because of \rpart{FRL432.a}, we have then 
\begin{equation}\label{MX}
 \Luo{n}{r}
 =   -\rk{-\og}r_{2n-1} +r_{2n} - \rk{\fourIdx{}{-\og}{}{}{M^{\langle r\rangle}}}_{n-1} - \rk{\fourIdx{}{-\og}{}{}{L^{\langle r\rangle}}}_{n-1}(\Luo{n-1}{r})^\mpi\rk{r_{2n-1} - \Muo{n-1}{r}} .
\end{equation}
 Obviously, \(r_{2n-1} = -\su{2n-1}\) and \(r_{2n} =\su{2n}\).
 \rlem{L1338} yields \(\Luo{\ell}{r} =  \Lu{\ell}\) for \(\ell\in\set{n-1,n}\) and \(\Muo{n-1}{r} = -\Mu{n-1}\).
 Taking into account that \eqref{DIR} holds true for all \(j\in\mn{0}{2n-1}\), from \rlem{L1338} we also get \(\rk{\fourIdx{}{-\og}{}{}{M^{\langle r\rangle}}}_{n-1} = - \Mbu{n-1}\) and \(\rk{\fourIdx{}{-\og}{}{}{L^{\langle r\rangle}}}_{n-1} = \Lbu{n-1}\).
 Thus, in view of \eqref{MX}, the proof is complete.
\end{proof}

\begin{lem}\label{FR433.}
Let \(\ug,\og\in\R\) and let \(n\in\N\).
 Then:
\benui
 \il{FR433.a} If \(\seqs{2n-1} \in\Kggequ{2n-1}\), then 
\begin{equation}\label{TP}
\Trip{n}
= \alpha\su{2n-1} +  \Mau{n-1} + \Lau{n-1}  \Lu{n-1}^\mpi  (\su{2n-1} - \Mu{n-1}) .
\end{equation}
 \il{FR433.b} If \(\seqs{2n-1} \in\Lggequ{2n-1}\), then 
\[
\Trip{n} =   \beta\su{2n-1} - \ek*{\Mbu{n-1} + \Lbu{n-1} \Lu{n-1}^\mpi  (\su{2n-1} - \Mu{n-1})}.
\]
\eenui
\end{lem}
\begin{proof}
 First we suppose that \(\seqs{2n-1}\in\Kggequ{2n-1}\).
 Then there is an \(\su{2n}\in\Cqq \) such that \(\seqs{2n}\in\Kggqu{2n}\).
 Using \rlemp{FRL432.}{FRL432.a}, the equation \(\Trip{n} =\su{2n} -\Lu{n}\), and the equation \(\sau{2n-1} = -\alpha\su{2n-1} +\su{2n}\), we get \eqref{TP}.
 \rPart{FR433.a} is proved.
 In view of \rlemp{FRL432.}{FRL432.b}, \rpart{FR433.b} can be checked analogously.
\end{proof}

\begin{lem}\label{FR434.}
 Let \(\ug,\og\in\R\) and let \(n\in\N\).
 Then:
\benui
 \il{FR434.a} If \(\seqs{2n+1} \in\Kggqu{2n+1}\), then 
\begin{equation*}
 \Lau{n}
 =\su{2n+1} - \Mu{n} - \Lu{n}\rk*{\alpha \Iq  + \ek*{\Lau{n-1}}^\mpi\ek*{\sau{2n-1} - \Mau{n-1}}} .
\end{equation*}
 \il{FR434.b} If \(\seqs{2n+1} \in\Lggqu{2n+1}\), then 
\begin{equation}\label{BLB}
 \Lbu{n}
 = - (\su{2n+1} - \Mu{n}) + \Lu{n}\rk*{\beta \Iq  + \ek*{\Lbu{n-1}}^\mpi\ek*{\sbu{2n-1} - \Mbu{n-1}}} .
\end{equation}
\eenui
\end{lem}
\begin{proof}
 \rPart{FR434.a} is proved in ~\cite[Proposition~6.5]{MR3014201}.
 In order to check \rpart{FR434.b}, we consider an arbitrary sequence \(\seqs{2n+1}\in\Lggqu{2n+1}\).
 Let \(r_j \defeq  (-1)^j s_j\) for each \(j\in\mn{0}{2n+1}\).
 According to \rlemp{FRL4.26.}{FRL4.26.a}, the sequence  \((r_j)_{j=0}^{2n+1}\) belongs to \(\Kgguuu{q}{2n+1}{-\og}\).
 Thus, \rpart{FR434.a} provides us 
\begin{equation}\label{CFHGR}
 \rk{\fourIdx{}{-\og}{}{}{L^{\langle r\rangle}}}_n
 =  r_{2n+1} - \Mu{n}^{\langle r\rangle} - \Luo{n}{r} \rk*{ (-\beta) \Iq  + \rk{\fourIdx{}{-\og}{}{}{L^{\langle r\rangle}}}_{n-1}\ek*{\rk{\fourIdx{}{-\og}{}{}{r}}_{2n-1} - \rk{\fourIdx{}{-\og}{}{}{M^{\langle r\rangle}}}_{n-1}}} .
\end{equation}
 Obviously, \(r_{2n+1} = -\su{2n+1}\) and, by \eqref{DIR}, furthermore \(\rk{\fourIdx{}{-\og}{}{}{r}}_{2n-1} = -\sbu{2n-1}\).
 \rlem{L1338} yields \(\Muo{n}{r} = -\Mu{n}\) and \(\Luo{n}{r} = \Lu{n}\).
 Since \eqref{DIR} holds true for all \(j\in\mn{0}{2n}\), \rlem{L1338} also shows that \(\rk{\fourIdx{}{-\og}{}{}{L^{\langle r\rangle}}}_\ell=  \Lbu{\ell}\) for \(\ell\in\set{n-1,n}\) and that \(\rk{\fourIdx{}{-\og}{}{}{M^{\langle r\rangle}}}_{n-1}=-\Mbu{n-1}\).
 Thus, \eqref{BLB} follows from \eqref{CFHGR}.
\end{proof}

\begin{lem}\label{FR435.}
 Let \(\ug,\og\in\R\) and let \(n\in\N\).
 Then:
\benui
 \il{FR435.a} If \(\seqs{2n} \in\Kggequ{2n}\), then 
\begin{equation}\label{FT1}
\Tripa{n}
= -\alpha\su{2n} + \Mu{n} + \Lu{n}\rk*{\alpha \Iq  + \ek*{\Lau{n-1}}^\mpi\ek*{\sau{2n-1} - \Mau{n-1}}} .
\end{equation}
 \il{FR435.b} If \(\seqs{2n} \in\Lggequ{2n}\), then 
\begin{equation}\label{FT2}
 \Tripb{n}
 =  \beta\su{2n} -\ek*{\Mu{n} + \Lu{n}\rk*{\beta \Iq  + \ek*{\Lbu{n-1}}^\mpi\ek*{\sbu{2n-1} - \Mbu{n-1}}}}.
\end{equation}
\eenui
\end{lem}
\begin{proof}
 Choose a matrix \(\su{2n+1}\in\Cqq \) such that \(\seqs{2n+1}\) belongs to \(\Kggqu{2n+1}\) (resp.\ \(\Lggqu{2n+1}\)) and apply \rlem{FR434.}.
\end{proof}

\begin{lem}\label{FR428.}
 Let \(\ug,\og\in\R\) and let \(\kappa\in\Ninf\).
 Then:
\benui
 \il{FR428.a} If \(\seqska  \in\Kggqu{\kappa}\), then \((\sau{j})_{j=0}^m\in \Kggqu{m}\) for all \(m\in\mn{0}{\kappa -1}\).
 \il{FR428.b}  If \(\seqska  \in\Lggqu{\kappa}\), then \((\sbu{j})_{j=0}^m\in \Lggqu{m}\) for all \(m\in\mn{0}{\kappa -1}\).
\eenui
\end{lem}
\begin{proof}
 \eqref{FR428.a} In view of \rremss{R1058}{R1447}, \rpart{FR428.a} is proved in~\cite[Proposition~5.1]{MR3133464}.

 \eqref{FR428.b} Use \rlem{FRL4.26.} and \rpart{FR428.a}.
\end{proof}

 The following result gives us a first impression of the interplay between \(\Fggqu{m}\) and the classes studied in this section.
\bpropl{L8.1.35.}
 Let \(\ug\in\R\), let \(\og\in(\ug,\infp)\), and let \(n\in\NO\).
 Then \(\Fggqu{2n+1} = \Kggqu{2n+1} \cap \Lggqu{2n+1}\).
\eprop
\begin{proof}
 Let \(\seqs{2 n +1} \in \Fggqu{2 n+1}\).
 Then \rdefnp{D1159}{D1159.a} shows that \(\set{(\sau{j})_{j=0}^{2 n},(\sbu{j})_{j=0}^{2 n}} \subseteq \Hggqu{2 n}\).
 Moreover, from \rlem{L1307} we know that \(\seqs{2 n}\in \Hggqu{2 n}\).
 Hence, \(\seqs{2 n +1}\) belongs to \(\Kggqu{2 n+1} \cap  \Lggqu{2 n+1}\).
 
 Conversely, now assume that \(\seqs{2 n +1}\) belongs to \(\Kggqu{2 n+1}\cap  \Lggqu{2 n+1}\).
 Then the definitions of the sets \(\Kggqu{2 n+1}\) and \(\Lggqu{2 n+1}\) in combination with \rdefnp{D1159}{D1159.a} yield immediately \(\seqs{2 n + 1}\in \Fggqu{2 n+1}\).
\end{proof}

\section{On the inclusion $\Fggqu{m}\subseteq\Kggequ{m}\cap\Lggequ{m}$}\label{S0746}
 In this section we investigate interrelations between the set \(\Fggqu{m}\) on the one hand and the intersection \(\Kggequ{m}\cap\Lggequ{m}\) on the other hand.

\bpropl{L8.1.24.}
 Let \(\ug\in\R\), let \(\og\in(\ug,\infp)\), let \(\kappa\in\Ninf\), and let \(\seqs{\kappa} \in\Fggqu{\kappa}\).
 Then:
 \benui
  \il{L8.1.24.a} \(\set{(\sau{j})_{j=0}^{\kappa - 1}, (\sbu{j})_{j-0}^{\kappa -1}} \subseteq \Fggqu{\kappa-1}\).
  \il{L8.1.24.b} If  \(\kappa\ge 2\), then \((\scu{j})_{j=0}^{\kappa - 2} \in \Fggqu{\kappa-2}\).
 \eenui
\eprop
\begin{proof}
 \eqref{L8.1.24.a} First we consider the case that \(\kappa = 2n+1\) with some \(n\in\NO \).
 Then \rdefnp{D1159}{D1159.a} yields \(\set{ \Hau{n}, \Hbu{n}} \subseteq \Cggo{(n+1)q}\).
 If \(n=0\), then~\eqref{L8.1.24.a} is proved.
 Assume \(n\ge 1\).
 Then \rpropp{M8.1.21.}{M8.1.21.a} provides us \(\seqs{2n}\in\Fggqu{2n}\) and, consequently, \(\Hcu{n-1} \in\Cggo{nq}\).
 From \rrem{R1539} we have  \( \Hcu{n-1} =  (_\alpha (H_\beta))_{n-1} = ((_\alpha H)_\beta)_{n-1}\).
 Thus,  \((\sbu{j})_{j=0}^{2n} \in \Kggqu{2n}\) and  \((\sau{j})_{j=0}^{2n}\in\Lggqu{2n}\).
 Applying \rlemp{FR428.}{FR428.a} to  \((\sbu{j})_{j=0}^{2n}\), from \rrem{R5.2.} we get  \((\scu{j})_{j=0}^{2n-1} \in \Kggqu{2n-1}\).
 In  particular, \(-\alpha \Hcu{n-1} + \Kcu{n-1}\in\Cggo{nq}\).
 Taking into account that \rrem{R5.2.} shows that \(-\alpha \scu{j} + \scu{j+1} = -\alpha\beta \sau{j} + (\alpha +\beta) \sau{j+1} - \sau{j+2}\) holds true for all \(j\in\mn{0}{2n-2}\), then \(-\alpha\beta \Hau{n-1} + (\alpha+\beta) \Kau{n-1} - \Gau{n-1} \in \Cggo{nq}\) and  \((\sau{j})_{j=0}^{2n}\in \Fggqu{2n}\) follow.
 Applying \rlemp{FR428.}{FR428.b} to the sequence \((\sau{j})_{j=0}^{2n}\), we obtain from \rrem{R5.2.} that \((\scu{j})_{j=0}^{2n-1}\) belongs to \(\Lggqu{2n-1}\).
 In particular, \(\beta \Hcu{n-1} - \Kcu{n-1} \in\Cggo{nq}\).
 Since \rrem{R5.2.} yields \(\beta \scu{j} - \scu{j+1} = -\alpha\beta \sbu{j} + (\alpha +\beta) \sbu{j+1} - \sbu{j+2}\) for all \(j\in\mn{0}{2n-2}\), we conclude  then \(-\alpha\beta \Hbu{n-1} + (\alpha + \beta) \Kbu{n-1} - \Gbu{n-1} \in \Cggo{nq}\) and   \((\sbu{j})_{j=0}^{2n}\in \Fggqu{2n}\) and, because of \rrem{R5.2.}, furthermore   \((\scu{j})_{j=0}^{2n-1} \in \Fggqu{2n-1}\).

 Now we consider the case that \(\kappa = 2n\) with some \(n\in\N\).
 Then the matrices \(\Hu{n}\) and \(\Hcu{n-1}\) are both \tnnH{}.
 \rpropp{M8.1.21.}{M8.1.21.a} provides us \(\seqs{2n-1}\in\Fggqu{2n-1}\).
 Hence, \(\set{\Hau{n-1} \Hbu{n-1}} \in\Cggo{nq}\).
 Thus, \(\seqs{2n} \in \Kggqu{2n} \cap  \Lggqu{2n}\).
 Applying  \rlemp{FR428.}{FR428.a} to  \(\seqs{2n}\), we infer \((\sau{j})_{j=0}^{2n-1} \in \Kggqu{2n-1}\).
 Thus, \(-\alpha \Hau{n-1} + \Kau{n-1}\in\Cggo{nq}\).
 Since \rrem{R1539} shows that \(\beta \Hau{n-1} -  \Kau{n-1} =  \Hcu{n-1}\) is valid, from \rdefnp{D1159}{D1159.a}  we get  \((\sau{j})_{j=0}^{2n-1} \in \Fggqu{2n-1}\).
 \rlemp{FR428.}{FR428.b} provides us \(\sbu{j=0}^{2n-1}\in\Lggqu{2n-1}\).
 Consequently, the matrix \(\beta \Hbu{n-1} - \Kbu{n-1}\) is \tnnH{}.
 Since \rrem{R1539} shows that \( -\alpha \Hbu{n-1} + \Kbu{n-1} = \Hcu{n-1}\) holds true, then \rdefnp{D1159}{D1159.a} shows that  \((\sbu{j})_{j=0}^{2n-1} \in\Fggqu{2n-1}\).
 
 \eqref{L8.1.24.b} In view of \rrem{R5.2.}, this follows by applying \rpart{L8.1.24.a} to one of the sequences \((\sau{j})_{j=0}^{\kappa-1}\) or \((\sbu{j})_{j=0}^{\kappa-1}\) which, according to \rpart{L8.1.24.a}, belong to \(\Fggqu{\kappa-1}\).
\end{proof}

 Obviously, \(\Fggqu{0} = \Kggequ{0} \cap \Lggequ{0}\) and \(\Fggqu{0} =\Kggqu{0} \cap \Lggqu{0}\) for all \(\alpha,\beta\in\R\) with \(\alpha < \beta\).

\bpropl{L8.1.33.}
 Let \(\ug\in\R\) and let \(\og\in(\ug,\infp)\).
 Then \(\Fggqu{m} \subseteq \Kggequ{m} \cap \Lggequ{m}\) for all \(m\in\NO \).
\eprop
\begin{proof}
 The case \(m=0\) is trivial.
 Let \(m\in\N\) and let \(\seqs{m}\in \Fggqu{m}\).
 
 We consider now the case that \(m=2n+1\) with some \(n\in\NO \).
 From \rlem{FR5.3.} we know that  \(\seqs{2n+1} \in \Hggequ{2n+1}\).
 Furthermore, \rdefnp{D1159}{D1159.a} yields \(\set{(\sau{j})_{j=0}^{2n}, (\sbu{j})_{j=0}^{2n}}\subseteq \Hggqu{2n}\).
 Thus, \rrem{FRL425} shows that \(\seqs{2n+1}\) belongs to \(\Kggequ{2n+1} \cap\Lggequ{2n+1} \).

 It remains to consider the case that \(m=2n\) with some \(n\in\N\).
 Since \(\seqs{2n}\) belongs to \(\Fggqu{2n}\), from \rdefnp{D1159}{D1159.b}  we have \(\seqs{2n}\in \Hggqu{2n}\).
 \rpropp{L8.1.24.}{L8.1.24.a} yields \(\set{(\sau{j})_{j=0}^{2n-1}, (\sbu{j})_{j=0}^{2n-1}}\subseteq \Fggqu{2n-1}\).
 Consequently, \rlem{FR5.3.} provides us \(\set{(\sau{j})_{j=0}^{2n-1}, (\sbu{j})_{j=0}^{2n-1}}\subseteq \Hggequ{2n-1}\).
Hence, \rrem{FRL425} yields that \(\seqs{2n}\) belongs to \(\Kggequ{2n} \cap\Lggequ{2n}\).
\end{proof}

\begin{proof}[Alternative proof of \rprop{L8.1.33.}]
 Let \(\seqs{m}\in\Fggqu{m}\).
 In view of \rthm{LHP}, then \(\MggqFs{m}\neq\emptyset\).
 Defining \(\sigma_{[\ug]}\colon\BorK\to\Cqq\) by \(\sigma_{[\ug]}(B)\defeq\sigma(B\cap\ab)\) and $\sigma^{[\og]}\colon\BorL\to\Cqq$ by \(\sigma^{[\og]}(C)\defeq\sigma(C\cap\ab)\), we immediately see \(\sigma_{[\ug]}\in\MggqKsg{m}\) and $\sigma^{[\og]}\in\MggqLsg{m}\), resp.
 Thus, \rthmss{T1306}{T1328} provide $\seqs{m}\in\Kggequ{m}$ and $\seqs{m}\in\Lggequ{m}$.
\end{proof}

\bpropl{L8.1.34.}
 Let \(\ug\in\R\), let \(\og\in(\ug,\infp)\), and let \(n\in\NO\).
 Then \(\Fggqu{2n+1} = \Kggequ{2n+1} \cap \Lggequ{2n+1}\).
\eprop
\begin{proof}
 In view of \rdefnp{D1159}{D1159.a}, \rrem{FRL425} provides us \(\Kggequ{2n+1} \cap \Lggequ{2n+1} \subseteq \Fggqu{2n+1}\).
 Thus, the application of \rprop{L8.1.33.} completes the proof.
\end{proof}

 In view of  \rprop{L8.1.33.}, the following example shows that \(\Fgguuuu{1}{2}{\ug}{\og}\) is a proper subset of \(\Kggeuuu{1}{2}{\ug}\cap\Lggeuuu{1}{2}{\og}\) if \(\ug\) and \(\og\) are real numbers such that \(\ug <\og\).

\begin{ex}\label{E0844}
 Let \(\ug\in\R\), let \(\og\in(\ug,\infp)\), and let
 \begin{align*}
  \su{0}&\defeq  4,&
  \su{1}&\defeq  2(\alpha +\beta),&
 &\text{and}&
  \su{2}&\defeq  (\alpha +\beta)^2 + 3\ba^2.
 \end{align*}
 Then  \(\Hu{1} =\tmat{2 & 0\\ \alpha +\beta & \sqrt{3} \ba }\tmat{2 & 0\\ \alpha +\beta & \sqrt{3} \ba  }^\ad\) is \tpH{} and \(\Lu{1} = 3\ba\), \(\Hau{0} = \Lau{0} = 2 \ba\), and  \(\Hbu{0} = \Lbu{0} = 2 \ba\) are positive.
 Thus, \rpartss{L5.1.30.c}{L5.1.30.d} of \rprop{L5.1.30.} show that \(\seqs{2}\) belongs to \( \Kggeuuu{1}{2}{\ug} \cap \Lggeuuu{1}{2}{\og}\).
 However, taking into account that \(\Hcu{0} = -2 \ba ^2\) is negative, we see that \(\seqs{2} \notin \Fgguuuu{1}{2}{\ug}{\og}\).
\end{ex}

\section{On the structure of \habHnnd{} sequences}\label{S0747}
 In this section, we discuss essential aspects of the structure of the elements of the set \(\Fggqu{m}\).
 In particular, we will see that each element of such a sequence varies within a closed matricial subinterval of \(\CHq\) the endpoints of which are completely determined by the preceding elements of the sequence.
 First we introduce two types of matricial intervals:

 Let \(A,B\in\CHq\).
 Then we set\index{\([A,B]\defeq\setaa{X\in\CHq}{A\lleq X\lleq B}\)}\index{\((A,B)\defeq\setaa{X\in\CHq}{A<X<B}\)}
 \begin{align*}
  [A,B]&\defeq\setaa{X\in\CHq}{A\lleq X\lleq B}&
 &\text{and}&
  (A,B)&\defeq\setaa{X\in\CHq}{A<X<B}.
 \end{align*}
 The matrix \(A\) (resp.\ \(B\)) is called the \emph{matricial left} (resp.\ \emph{right}) \emph{endpoint} of \([A,B]\) or \((A,B)\).
 The matrix \(C\defeq\frac{1}{2}\rk{A+B}\) (resp.\ \(D\defeq B-A\)) is called the \emph{matricial midpoint} (resp.\ \emph{length}) of \([A,B]\) or \((A,B)\).

 The following result indicates that closed and open matricial intervals are intimately related to closed and open matricial matrix balls.
 
\bleml{L0719}
 Let \(A,B\in\CHq\) and let \(D\defeq B-A\).
 Then:
 \benui
  \il{L0719.a} \([A,B]\neq\emptyset\) if and only if \(D\in\Cggq\).
 In this case, 
  \[
   [A,B]
   =\setaa*{A+\sqrt{D}K\sqrt{D}}{K\in[\Oqq,\Iq]}.
  \]
  \il{L0719.b} \((A,B)\neq\emptyset\) if and only if \(D\in\Cgq\).
 In this case, 
 \[
   (A,B)
   =\setaa*{A+\sqrt{D}K\sqrt{D}}{K\in(\Oqq,\Iq)}.
  \]
 \eenui
\elem
\bproof
 \eqref{L0719.a} If \([A,B]\neq\emptyset\), then there exists a matrix \(X\in\CHq\) with \(A\lleq X\lleq B\).
 In particular, \(A\lleq B\), \ie{}, \(D\in\Cggq\). Conversely, assume \(D\in\Cggq\).
 Obviously, \([\Oqq,\Iq]\neq\emptyset\).
 Let \(K\in[\Oqq,\Iq]\).
 Then \(X\defeq A+\sqrt{D}K\sqrt{D}\) is \tH{}.
 By virtue of \rrem{R0705}, furthermore \(X-A=\sqrt{D}K\sqrt{D}\lgeq\Oqq\) and \(B-X=D-\sqrt{D}K\sqrt{D}=\sqrt{D}\rk{\Iq-K}\sqrt{D}\lgeq\Oqq\).
 Hence, \(X\in[A,B]\).
 In particular, \([A,B]\neq\emptyset\).
 Now let \(X\in[A,B]\).
 In view of \rrem{R0746}, then \(K\defeq\sqrt{D}^\mpi\rk{X-A}\sqrt{D}^\mpi\) is \tH{} and we have
 \beql{L0719.1}
  \Oqq
  \lleq X-A
  \lleq B-A
  =D.
 \eeq
 Hence, \(\ran{X-A}\subseteq\ran{D}\) and \(\nul{D}\subseteq\nul{X-A}\) by virtue of \rrem{R0736}.
 According to \rrem{R1546}, thus \(\ran{X-A}\subseteq\ran{\sqrt D}\) and \(\nul{\sqrt D}\subseteq\nul{X-A}\).
 Using \rremss{MA.7.8.}{A.7.8-1.}, we obtain then \(\sqrt{D}K\sqrt{D}=X-A\).
 Hence, \(X=A+\sqrt{D}K\sqrt{D}\).
 With \rrem{R0705}, we conclude from \eqref{L0719.1} that \(\Oqq\lleq K\lleq\sqrt{D}^\mpi D\sqrt{D}^\mpi\).
 Using \rremss{R0746}{R0757} and \rprop{PA*2}, we have
 \[
  \sqrt{D}^\mpi D\sqrt{D}^\mpi
  =\sqrt{D}^\mpi\sqrt{D}\sqrt{D}\sqrt{D}^\mpi
  =\sqrt{D}\sqrt{D}^\mpi\sqrt{D}\sqrt{D}^\mpi
  =\sqrt{D}\sqrt{D}^\mpi
  =\OPu{\ran{\sqrt{D}}}
  \lleq\Iq.
 \]
 Thus, \(\Oqq\lleq K\lleq\Iq\), \ie{}, \(K\in[\Oqq,\Iq]\).
 
 \eqref{L0719.b} Since every \tpH{} matrix is invertible, this can be shown more easy in a similar manner.
\eproof

 From \rlem{L0719} we conclude the following result:
\breml{R1503}
 Let \(A,B\in\CHq\) and let \(C\defeq\frac{1}{2}\rk{A+B}\) and \(D\defeq B-A\).
 Then:
 \benui
  \il{R1503.a} If \([A,B]\neq\emptyset\), then \(\set{A,C,B}\subseteq\setaa{A+\eta D}{\eta\in[0,1]}\subseteq[A,B]\).
  \il{R1503.b} If \((A,B)\neq\emptyset\), then \(C\in\setaa{A+\eta D}{\eta\in(0,1)}\subseteq(A,B)\).
 \eenui
\erem
 
 Now we are going to introduce those sequences of matrices which will be needed to describe that sequences of matricial intervals which turn out to be associated with an \tabHnnd{} sequence:
 
\bdefnl{D1719}
 Let \(\ug\in\R\), let \(\og\in(\ug,\infp)\), let \(\kappa\in\NOinf\), and let  \(\seqs{\kappa}\in\Fggqu{\kappa}\). 
 For all \(k\in\NO \) with \(2k\le \kappa\), let\index{a@\(\umg{2k} \defeq  \alpha\su{2k} + \Tripa{k}\)}\index{b@\(\omg{2k}\defeq  \beta\su{2k} - \Tripb{k}\)}
 \begin{align*}
  \umg{2k}&\defeq  \alpha\su{2k} + \Tripa{k}&
  &\text{and}&
  \omg{2k}&\defeq  \beta\su{2k} - \Tripb{k}.
 \end{align*}
 For all \(k\in\NO \) with \(2k + 1\le\kappa\), let\index{a@\(\umg{2k+1} \defeq\Trip{k+1}\)}\index{b@\(\omg{2k+1} \defeq -\alpha\beta\su{2k} + (\alpha +\beta)\su{2k+1} -\Tripc{k}\)}
 \begin{align*}
  \umg{2k+1}&\defeq\Trip{k+1}&
  &\text{and}&
  \omg{2k+1}&\defeq -\alpha\beta\su{2k} + (\alpha +\beta)\su{2k+1} -\Tripc{k}.
 \end{align*}
 Then the sequences \(\seq{\umg{j}}{j}{0}{\kappa}\) and \(\seq{\omg{j}}{j}{0}{\kappa}\) are called the \notii{\tsequmgo{\(\seqs{\kappa}\)}} and the \notii{\tseqomgo{\(\seqs{\kappa}\)}}, resp.
\edefn

\begin{rem}\label{R1531}
 Let \(\ug\in\R\), let \(\og\in(\ug,\infp)\), let \(\kappa\in\NOinf\), and let \(\seqs{\kappa} \in\Fggqu{\kappa}\).
 Then \(\umg{0} =\ug \su{0}\) and \(\omg{0} = \og \su{0}\).
 If \(\kappa\geq1\), further \(\umg{1} = \su{1} \su{0}^\mpi  \su{1}\) and \(\omg{1} = -\ug\og \su{0} + (\alpha +\beta) \su{1}\).
\end{rem}

\begin{rem}\label{B8.1.28.}
 Let \(\ug\in\R\), let \(\og\in(\ug,\infp)\), let \(\kappa\in\NOinf\), and let \(\seqs{\kappa} \in\Fggqu{\kappa}\).
 Because of \rdefn{D1719}, \rnota{N41}, and \rremss{L1319}{BN3.}, then \(\set{\umg{j},\omg{j}}\subseteq\CHq\) for all \(j\in\mn{0}{\kappa}\).
\end{rem}

\bdefnl{D1730}
 Let \(\ug\in\R\), let \(\og\in(\ug,\infp)\), let \(\kappa\in\NOinf\), and let  \(\seqs{\kappa}\in\Fggqu{\kappa}\).
 Then the sequences \((\usc{j})_{j=0}^\kappa\) and \((\osc{j})_{j=1}^\kappa\) given with \(\umg{-1} \defeq  \Oqq\) by
 \begin{align*}
  \usc{j}&\defeq  s_j -\umg{j-1}&
  &\text{and}&
  \osc{j}\defeq  \omg{j-1} - s_j
 \end{align*}
 are called the \notii{\tsequsco{\(\seqs{\kappa}\)}} and the \notii{\tseqosco{\(\seqs{\kappa}\)}}, resp.
\edefn

 Taking into account that \rpropp{F8.1.44.}{F8.1.44.a} will show \(\omg{j-1}-\umg{j-1}\in\Cggq\), the matrices \(\usc{j}\) and \(\osc{j}\) introduced in \rdefn{D1730} determine the position of the matrix \(\su{j}\) in the interval \([\umg{j-1},\omg{j-1}]\).

\begin{rem}\label{R1530}
 Let \(\ug\in\R\), let \(\og\in(\ug,\infp)\), let \(\kappa\in\NOinf\), and let \(\seqs{\kappa} \in\Fggqu{\kappa}\).
 Then \(\usc{0} = \su{0}\).
 If \(\kappa\geq1\), then furthermore \(\usc{1} = \sau{0}\) and \(\osc{1} = \sbu{0}\).
 If \(\kappa\geq2\), then moreover \(\osc{2} = \scu{0}\).
\end{rem}

\begin{rem}\label{M8.1.6.}
 Let \(\ug\in\R\), let \(\og\in(\ug,\infp)\), let \(\kappa\in\NOinf\), and let  \(\seqs{\kappa}\in\Fggqu{\kappa}\). 
 Then \(\usc{2n} = \Lu{n}\) for all \(n\in\NO \) with \(2n \le \kappa\) and \(\usc{2n+1} = \Lau{n}\) for all \(n\in\NO \) with \(2n+1\le \kappa\).
 Furthermore, \(\osc{2n+1} = \Lbu{n}\) for all \(n\in\NO \) with \(2n+1\le \kappa\) and \(\osc{2n+2} = \Lcu{n}\) for all \(n\in\NO \) with \(2n + 2\le \kappa\).
\end{rem}

 From \rrem{M8.1.6.} and \rnotap{N41}{N41.c} we see now that the matrices introduced in \rdefn{D1730} are indeed Schur complements in block Hankel matrices which are responsible for the property of belonging to the set \(\Fggqu{m}\).
For this reason, we had chosen the terminology introduced in \rdefn{D1730}.
 In the sequel, we investigate the interplay between the ranges of two consecutive elements of the sequence of matrices introduced in \rdefn{D1730}.

\bleml{L1433}
 Let \(\ug\in\R\), let \(\og\in(\ug,\infp)\), let \(\kappa\in\NOinf\), and let  \(\seqska\) be a sequence from \(\Cqq\).
 \benui
  \il{L1433.a} If \(\seqska\in\Fggqu{\kappa}\), then \(\usc{j}\in\Cggq\) for all \(j\in\mn{0}{\kappa}\) and \(\osc{j}\in\Cggq\) for all \(j\in\mn{1}{\kappa}\).
  \il{L1433.b} If \(\seqska\in\Fgqu{\kappa}\), then \(\usc{j}\in\Cgq\) for all \(j\in\mn{0}{\kappa}\) and \(\osc{j}\in\Cgq\) for all \(j\in\mn{1}{\kappa}\).
 \eenui
\elem
\bproof
 Use \rrem{M8.1.6.} and \rcor{B8.1.22.}.
\eproof
 
 In the sequel, we need detailed information on the interplay between the ranges of consecutive elements of the sequences of lower and upper Schur complements associated with an \tabHnnd{} sequence.
 The next result gives us a first impression on this theme.
\begin{lem}\label{L8.1.37.}
 Let \(\ug\in\R\), let \(\og\in(\ug,\infp)\), let \(\kappa\in\NOinf\), and let  \(\seqs{\kappa} \in \Fggqu{\kappa}\).
 Then \(\cR (\usc{j}) +\cR (\osc{j})\subseteq \cR (\usc{j-1})\cap \cR (\osc{j-1})\) for each \(j\in\mn{2}{\kappa}\).
\end{lem}
\begin{proof}
 Suppose \(\kappa \ge 2\).
 We consider an arbitrary \(n\in\N\) with \(2n \le \kappa\).
 In view of \rpropp{M8.1.21.}{M8.1.21.a}, we have \(\seqs{2n} \in \Fggqu{2n}\).
Thus, \rprop{L8.1.33.} yields  \(\seqs{2n} \in \Kggequ{2n}\cap \Lggequ{2n}\).
 Hence, using \rprop{KL5.1.30.},  we get 
\begin{equation}\label{CFMUS}
 \cR (\Lu{n})
 \subseteq \Ran{\Lau{n-1}}\cap \Ran{\Lbu{n-1}}.
\end{equation}
 \rpropp{L8.1.24.}{L8.1.24.a} provide us  \(\set{ (\sau{j})_{j=0}^{2n-1}, (\sbu{j})_{j=0}^{2n-1}} \subseteq  \Fggqu{2n-1}\).
 Hence, from \rprop{L8.1.33.} we conclude \((\sbu{j})_{j=0}^{2n-1} \in \Kggequ{2n-1}\) and \((\sau{j})_{j=0}^{2n-1} \in\Lggequ{2n-1}\).
 Using \rrem{R5.2.} and \rprop{KL5.1.30.}, we infer then
\begin{equation}\label{CFARL}
 \Ran{ \Lcu{n-1}}
 \subseteq \Ran{\Lau{n-1}}\cap \Ran{\Lbu{n-1}}.
\end{equation}
 From \eqref{CFMUS}, \eqref{CFARL}, and \rrem{M8.1.6.} we get \(\cR (\usc{2n}) + \cR (\osc{2n}) \subseteq \cR (\usc{2n-1}) \cap \cR (\osc{2n-1})\).
 
 Now we consider an arbitrary \(n\in\N\) with \(2n+1\le\kappa\).
 Then \rpropp{M8.1.21.}{M8.1.21.a} shows that \(\seqs{2n+1}\) belongs to \(\Fggqu{2n+1}\).
 According to \rprop{L8.1.33.}, this implies \(\seqs{2n+1} \in \Kggequ{2n+1} \cap \Lggequ{2n+1}\).
 Thus, from \rprop{KL5.1.30.} we obtain 
\begin{equation}\label{GM}
 \Ran{\Lau{n}}  \cup \Ran{\Lbu{n}}
 \subseteq \ran{\Lu{n}}.
\end{equation}
 \rpropp{L8.1.24.}{L8.1.24.a} yields \(\set{ (\sau{j})_{j=0}^{2n}, (\sbu{j})_{j=0}^{2n}} \subseteq  \Fggqu{2n}\).
 Consequently, \rprop{L8.1.33.} provides us \((\sau{j})_{j=0}^{2n} \in\Lggequ{2n-1}\) and \((\sbu{j})_{j=0}^{2n} \in \Kggequ{2n}\).
 Taking into account \rrem{R5.2.}, then \rprop{KL5.1.30.} shows that \(\ran{\Lau{n}}\cup\ran{\Lbu{n}}\subseteq\ran{\Lcu{n-1}}\).
 Therefore, in view of \eqref{GM} and \rrem{M8.1.6.}, we get then  the inclusion \(\cR (\usc{2n+1}) + \cR (\osc{2n+1})\subseteq \cR (\usc{2n})\cap\cR (\osc{2n})\).
\end{proof}

 \rcor{C0834} will show, that in the situation of \rlem{L8.1.37.} we even have the equality \(\cR (\usc{j}) +\cR (\osc{j})=\cR (\usc{j-1})\cap \cR (\osc{j-1})\) for all \(j\in\mn{2}{\kappa}\).

\bdefnl{D1516}
 Let \(\ug\in\R\), let \(\og\in(\ug,\infp)\), let \(\kappa\in\NOinf\), and let  \(\seqs{\kappa}\in\Fggqu{\kappa}\).
 Then the sequences \((\cen{j})_{j=0}^\kappa\) and \((\dia{j})_{j=0}^\kappa\) given by\index{c@\(\cen{j}\defeq\frac{1}{2}\rk{\umg{j}+\omg{j}}\)} \index{d@\(\dia{j}\defeq\omg{j}-\umg{j}\)}
 \begin{align*}
  \cen{j}&\defeq\frac{1}{2}\rk{\umg{j}+\omg{j}}&
  &\text{and}&
  \dia{j}&\defeq\omg{j}-\umg{j}
 \end{align*}
 are called the \notii{\tseqceno{\(\seqs{\kappa}\)}} and the \notii{\tseqdiao{\(\seqs{\kappa}\)}}, resp.
\edefn

\begin{rem}\label{M8.1.40.}
 Let \(\ug\in\R\), let \(\og\in(\ug,\infp)\), let \(\kappa\in\NOinf\), and let  \(\seqska\in\Fggqu{\kappa}\).
 If \(m \in \mn{0}{\kappa}\), then it is easily checked that \((\dia{j})_{j=0}^m\) is the \tseqdiao{\(\seqs{m}\)}.
 Obviously, \(\dia{0} = \ba  \su{0}\).
 If \(\kappa\ge 1\), then \(\dia{1} = -\alpha\beta \su{0} + (\alpha+\beta) \su{1} - \su{1} \su{0}^\mpi \su{1}\) and \rdefnss{D1516}{D1730} also show that \(\dia{j} = \usc{j+1} + \osc{j+1}\) for all \(j\in\mn{0}{\kappa -1}\).
\end{rem}

\begin{rem}\label{R1535}
 Let \(\ug\in\R\), let \(\og\in(\ug,\infp)\), let \(\kappa\in\NOinf\), and let \(\seqs{\kappa} \in\Fggqu{\kappa}\).
 Because of \rdefn{M8.1.40.} and \rrem{B8.1.28.}, then \(\set{\cen{j},\dia{j}}\subseteq\CHq\) for all \(j\in\mn{0}{\kappa}\).
\end{rem}

 The next result is of fundamental importance for our further considerations.
 Again we make essential use of the parallel sum of matrices:
\bthml{M8.1.43.}
 Let \(\ug\in\R\), let \(\og\in(\ug,\infp)\), let \(\kappa\in\NOinf\), and let  \(\seqs{\kappa} \in \Fggqu{\kappa}\).
 Then \(\dia{0} = \ba  \usc{0}\).
 If \(\kappa\ge 1\), then \(\dia{k} = \ba  (\usc{k} \ps  \osc{k})\) and \(\dia{k} = (\beta-\alpha) (\osc{k} \ps  \usc{k})\) for all \(k\in\mn{1}{\kappa}\).
\ethm
\begin{proof}
 In view of \rnotap{N41}{N41.c} and \rdefn{D1730}, we have \(\su{0} = \Lu{0} = \usc{0}\) and, consequently, \(\dia{0} = \omg{0} - \umg{0} = \ba  \su{0}\).
 Assume \(\kappa \ge 1\).
 According to \rlemp{L1433}{L1433.a}, we know that \(\set{\usc{k}, \osc{k}}\subseteq \Cggq \) for each \(k\in\mn{1}{\kappa}\).
 Thus, \rlemss{MA.8.7*.}{MA.8.6.} provide us \(\usc{k} \ps  \osc{k} =  \osc{k} \ps  \usc{k}\) for each \(k\in\mn{1}{\kappa}\).
 From \rrem{M8.1.40.} we see that
\begin{equation}\label{MTT}
 \ba  (\usc{k} \ps  \osc{k})
 = \ba  \usc{k} (\usc{k} +\osc{k})^\mpi  \osc{k}
 = \usc{k}\rk*{ \frac{1}{\beta -\alpha} \dia{k-1}}^\mpi  \osc{k}
\end{equation}
 is valid for each \(k\in\mn{1}{\kappa}\).
 Using \rpropp{M8.1.21.}{M8.1.21.a}  and \rrem{R1606}, we conclude that \(\seqs{1}\) belongs to \(\Dqqu{1}\).
 Hence, \rremss{MA.7.8.}{A.7.8-1.} yield \(\su{0} \su{0}^\mpi  \su{1} = \su{1}\) and \(\su{1} \su{0}^\mpi  \su{0} = \su{1}\).
 Taking into account \eqref{MTT} and \rremss{M8.1.6.}{M8.1.40.}, we obtain then
\begin{multline*}
 \ba  (\usc{1} \ps  \osc{1})
 =\usc{1} \rk*{\frac{1}{\beta -\alpha} \dia{0}}^\mpi  \osc{1}
 =\ek*{\Lau{0}}\su{0}^\mpi\ek*{\Lbu{0}}
 =\ek*{\sau{0}}\su{0}^\mpi\ek*{\sbu{0}}\\
 = (-\alpha \su{0} + \su{1}) \su{0}^\mpi   (\beta \su{0}-\su{1})
 = -\alpha\beta \su{0}\su{0}^\mpi \su{0} +\alpha \su{0}\su{0}^\mpi \su{1} + \beta \su{1}\su{0}^\mpi \su{0} - \su{1}\su{0}^\mpi\su{1}
 = \dia{1}.
\end{multline*}
 In the cases \(\kappa = 0\) and \(\kappa =1\), the proof is complete.
 Now suppose \(\kappa\ge 2\).
 We assume that \(m\in\mn{2}{\kappa}\) and that 
\begin{align}\label{IFUH}
\dia{j}&= \ba  (\usc{j} \ps  \osc{j})&\text{for each }j&\in\mn{0}{m-1}.
\end{align}
 Since the matrices \(\usc{m-1}\) and \(\osc{m-1}\) are both \tnnH{}, from \rlem{MA.8.7*.} and \rprop{MA.8.12.} we get 
\begin{equation}\label{IWR}
(\usc{m-1} \ps  \osc{m-1})^\mpi
= \OPu{\cR_{m-1}} (\usc{m-1}^\mpi  + \osc{m-1}^\mpi ) \OPu{\cR_{m-1}}
\end{equation}
 where \(\cR_{m-1} \defeq \cR (\usc{m-1})\cap \cR (\osc{m-1})\).
 By virtue of \rlem{L8.1.37.}, we infer \(\cR (\usc{m}) + \cR (\osc{m})\subseteq \cR_{m-1}\) and, consequently, \(\OPu{\cR_{m-1}} \usc{m} = \usc{m}\) and \(\OPu{\cR_{m-1}}  \osc{m} = \osc{m}\).
 Since the matrices \(\OPu{\cR_{m-1}}\) and \(\usc{m}\) are \tH{}, then \(\usc{m} \OPu{\cR_{m-1}} = \usc{m} \) also holds true.
 Using \eqref{MTT}, \eqref{IFUH}, and \eqref{IWR}, it follows
\beql{M8117N}\begin{split}
 \ba  (\usc{m} \ps  \osc{m})
 &= \usc{m} \rk*{\frac{1}{\beta -\alpha} \dia{m-1}}^\mpi  \osc{m}
 =  \usc{m} (\usc{m-1} \ps   \osc{m-1})^\mpi  \osc{m}\\
 &= \usc{m} \OPu{\cR_{m-1}}  (\usc{m-1}^\mpi  + \osc{m-1}^\mpi )  \OPu{\cR_{m-1}}  \osc{m}\\
 &= \usc{m} (\usc{m-1}^\mpi  + \osc{m-1}^\mpi ) \osc{m}
 = \usc{m} \usc{m-1}^\mpi  \osc{m} + \usc{m} \osc{m-1}^\mpi  \osc{m}.
\end{split}\eeq
 First we consider now the case that \(m=2n\) with some positive integer \(n\).
 Then from \eqref{M8117N} and \rrem{M8.1.6.} we obtain
\beql{M8118N}
 \ba  (\usc{2n} \ps  \osc{2n}) 
 = \Lu{n} \ek*{\Lau{n-1}}^\mpi  \Lcu{n-1} + \Lu{n} \ek*{\Lbu{n-1}}^\mpi  \Lcu{n-1}.
\eeq
 According to \rpropp{M8.1.21.}{M8.1.21.a}, the sequence \(\seqs{2n}\) belongs to \(\Fggqu{2n}\).
 Thus, \rprop{L8.1.33.} yields \(\seqs{2n} \in \Kggequ{2n} \cap \Lggequ{2n}\).
 Hence, \rlem{FR435.} shows that  \eqref{FT1} and \eqref{FT2} are true.
 This implies
\begin{multline}\label{NAU}
 \dia{2n}
 = \omg{2n} - \umg{2n}
 = \beta\su{2n} -\Tripb{n} -\ek*{\alpha\su{2n}+\Tripa{n}}\\
 = \beta \Lu{n} + \Lu{n} \ek*{\Lbu{n-1}}^\mpi\ek*{\sbu{2n-1} - \Mbu{n-1}}\\
 - \alpha \Lu{n} - \Lu{n} \ek*{\Lau{n-1}}^\mpi\ek*{\sau{2n-1} - \Mau{n-1}}.
\end{multline}
 Since  \(\seqs{2n}\) belongs to \(\Kggequ{2n}\cap \Lggequ{2n}\), we get from \rprop{L5.1.30.} that \(\nul{\Lau{n-1}} \cup \nul{\Lbu{n-1}}\subseteq \cN (\Lu{n})\).
 Consequently, \rrem{A.7.8-1.} shows that
\begin{align}\label{N8120N}
\Lu{n} \ek*{\Lau{n-1}}^\mpi   \Lau{n-1}&= \Lu{n}&
&\text{and}&
\Lu{n}  \ek*{\Lbu{n-1}}^\mpi  \Lbu{n-1}&= \Lu{n} .
\end{align}
 \rrem{R5.2.} and \rnota{N41} yield
\begin{align}
 \Lcu{0}
 &= \scu{0}
 = \beta \sau{0} - \sau{1}
 =\beta \Lau{0} - \sau{1} +  \Mau{0}\label{N8121N}
\intertext{and}
 \Lcu{0}
 &= \scu{0}
 = -\alpha \sbu{0} + \sbu{1} 
 = -\alpha \Lbu{0} + \sbu{1}-\Mbu{0}.\label{N8131N}
\end{align}
 Using \eqref{N8120N} for \(n=1\), from \eqref{N8121N} and \eqref{N8131N} we obtain then \(\Lu{1} \Lau{0}^\mpi  \Lcu{0} = \Lu{1} \Lau{0}^\mpi  \ek{\Mau{0} - \sau{1}} + \beta \Lu{1}\) and \(\Lu{1} \Lbu{0}^\mpi  \Lcu{0} = \Lu{1} \Lbu{0}^\mpi  [\sbu{1} -  \Mbu{0}] - \alpha \Lu{1}\).
 This implies
\begin{multline}\label{TL1}
\Lu{1} \Lau{0}^\mpi  \Lcu{0} + \Lu{1} \Lbu{0}^\mpi  \Lcu{0} \\
=  \Lu{1} \Lau{0}^\mpi  \ek*{\Mau{0} - \sau{1}} +\beta \Lu{1} + \Lu{1}  \Lbu{0}^\mpi  \ek*{\sbu{1} - \Mbu{0}} -\alpha \Lu{1}.
\end{multline}
 Comparing  \eqref{M8118N} and \eqref{NAU} for \(n=1\) with \eqref{TL1}, we conclude \(\dia{2} = (\beta-\alpha) (\usc{2} \ps  \osc{2})\).
 
 Now we suppose that \(n\ge 2\).
 From \rpropp{M8.1.21.}{M8.1.21.a} we know that \(\seqs{2n-1}\) belongs to \(\Fggqu{2n-1}\).
 Hence, \rpropp{L8.1.24.}{L8.1.24.a} implies \(\set{ (\sau{j})_{j=0}^{2n-2}, (\sbu{j})_{j=0}^{2n-2}} \subseteq \Fggqu{2n-2}\).
 Consequently,  because of \rprop{L8.1.33.}, we have \((\sau{j})_{j=0}^{2n-2}\in\Lggequ{2n-2}\) and \((\sbu{j})_{j=0}^{2n-2} \in \Kggequ{2n-2}\).
 In view \rrem{R5.2.}, the application of \rlemp{FR435.}{FR435.b} to the sequences \((\sau{j})_{j=0}^{2n-2}\) and of \rlemp{FR435.}{FR435.a} to the sequences  \((\sbu{j})_{j=0}^{2n-2}\) provides us 
\[\begin{split}
&((_\alpha \Theta)_\beta)_{n-1} \\
&=\beta \sau{2n-2}  - \ek*{   \Mau{n-1} + \Lau{n-1}\rk*{\beta \Iq  + ((_\alpha L)_\beta)_{n-2}^\mpi  \ek*{ ((_\alpha s)_\beta)_{2n-3} - ((_\alpha M)_\beta)_{n-2} }  }}\\
&=\beta \sau{2n-2} - \ek*{\Mau{n-1} + \Lau{n-1}\rk*{\beta \Iq  + \Lcu{n-2}^\mpi\ek*{\scu{2n-3} - \Mcu{n-2}}}}
\end{split}\]
 and, analogously,
\begin{multline*}
 \rk*{\fourIdx{}{\ug}{}{}{\rk{\Theta_\og}}}_{n-1}
 = - \alpha \sbu{2n-2} + \Mbu{n-1}\\
 +  \Lbu{n-1} \rk*{\alpha \Iq  + \Lcu{n-2}^\mpi\ek*{ \scu{2n-3} - \Mcu{n-2}}}.
\end{multline*}
 By virtue of \rrem{R5.2.}, we get 
\[\begin{split}
 \Lcu{n-1} 
 &= ((_\alpha L)_\beta)_{n-1} 
 =\beta \sau{2n-2} - \sau{2n-1} -  ((_\alpha \Theta)_\beta)_{n-1}\\
 &= \Mau{n-1} - \sau{2n-1} + \Lau{n-1}\rk*{\beta \Iq  + \Lcu{n-2}^\mpi\ek*{\scu{2n-3} - \Mcu{n-2}}}
\end{split}\]
and, analogously,
\[
 \Lcu{n-1} 
 = \sbu{2n-1} - \Mbu{n-1} -  \Lbu{n-1}\rk*{\alpha \Iq  + \Lcu{n-2}^\mpi\ek*{\scu{2n-3} - \Mcu{n-2}}}.
\]
 Using \eqref{N8120N}, we conclude then
\begin{multline*}
 \Lu{n} \ek*{\Lau{n-1}}^\mpi  \Lcu{n-1}
 = \Lu{n} \ek*{\Lau{n-1}}^\mpi\ek*{\Mau{n-1} -  \sau{2n-1}}\\
 + \Lu{n}\rk*{\beta \Iq  +  \Lcu{n-2}^\mpi\ek*{\scu{2n-3} - \Mcu{n-2}}}
\end{multline*}
and
\begin{multline*}
 \Lu{n} \ek*{\Lbu{n-1}}^\mpi  \Lcu{n-1}
 = \Lu{n} \ek*{\Lbu{n-1}}^\mpi\ek*{\sbu{2n-1} -  \Mbu{n-1}}\\
 - \Lu{n}\rk*{\alpha \Iq  +  \Lcu{n-2}^\mpi\ek*{\scu{2n-3} - \Mcu{n-2}}}.
\end{multline*}
 This implies
\begin{multline*}
\Lu{n} \ek*{\Lau{n-1}}^\mpi  \Lcu{n-1} + \Lu{n} \ek*{\Lbu{n-1}}^\mpi  \Lcu{n-1}\\
= \Lu{n} \ek*{\Lau{n-1}}^\mpi\ek*{\Mau{n-1} -  \sau{2n-1}}+\beta \Lu{n}\\
+ \Lu{n} \ek*{\Lbu{n-1}}^\mpi\ek*{\sbu{2n-1} - \Mbu{n-1}}-\alpha \Lu{n}.
\end{multline*}
 Comparing the foregoing equation with \eqref{M8118N} and \eqref{NAU}, we see that the equation \(\dia{2n} = \ba  (\usc{2n}\ps  \osc{2n})\) holds true.
 
 It remains to consider the case that \(m=2n+1\) with some \(n\in\N\).
 Because of \rrem{M8.1.6.} and \eqref{M8117N}, we have
\begin{equation}\label{IUA}
 \ba  (\usc{2n+1} \ps  \osc{2n+1})
 =  \Lau{n} \Lu{n}^\mpi  \Lbu{n} + \Lau{n} \Lcu{n-1}^\mpi  \Lbu{n}.
\end{equation}
 From \rpropp{M8.1.21.}{M8.1.21.a} we know that \(\seqs{2n+1}\) belongs to \(\Fggqu{2n+1}\).
 \rpropp{L8.1.24.}{L8.1.24.a} shows that   \((\sau{j})_{j=0}^{2n} \in \Fggqu{2n}\).
 Consequently, \rprop{L8.1.33.} yields \((\sau{j})_{j=0}^{2n} \in \Lggequ{2n}\).
 Using \rrem{R5.2.} and applying \rlemp{FR435.}{FR435.b} to the sequence  \((\sau{j})_{j=0}^{2n}\), we get
\begin{equation}\label{TP1}
 \begin{split}
 &\beta \sau{2n}  - \Tripc{n} 
 = \beta \sau{2n} - ((_\alpha \Theta)_\beta)_n\\
 &= \Mau{n} + \Lau{n}\rk*{\beta \Iq  + ((_\alpha L)_\beta)_{n-1}^\mpi    \ek*{((_\alpha s)_\beta)_{2n-1} - ((_\alpha M)_\beta)_{n-1}}}\\
 &= \Mau{n} + \Lau{n}\rk*{\beta \Iq  + \Lcu{n-1}^\mpi   \ek*{\scu{2n-1} -\Mcu{n-1}}}.
 \end{split}
\end{equation}
 Since \(\seqs{2n+1}\) belongs to \(\Fggqu{2n+1}\), we see from \rprop{L8.1.33.} that \(\seqs{2n+1} \in\Kggequ{2n+1}\) is valid.
 Thus, \rlemp{FR433.}{FR433.a} provides us
\begin{equation}\label{TP2}
 \alpha\su{2n+1} - \Trip{n+1} 
 = - \Mau{n}  - \Lau{n} \Lu{n}^\mpi  (\su{2n+1} - \Mu{n}).
\end{equation}
 In view of \rdefnss{D1516}{D1719}, \eqref{TP1}, and \eqref{TP2}, we conclude
\begin{equation}\label{TP3}\begin{split}
&\dia{2n+1} 
= \omg{2n+1} - \umg{2n+1} 
= -\alpha\beta\su{2n} +  (\alpha +\beta)\su{2n+1} - \Tripc{n} -\Trip{n+1}\\
&= \beta \sau{2n} -  \Tripc{n}  +  \alpha\su{2n+1} -\Trip{n+1}\\
&= \Lau{n}\rk*{\beta \Iq  + \Lcu{n-1}^\mpi  \ek*{\scu{2n-1} - \Mcu{n-1}}} - \Lau{n} \Lu{n}^\mpi  (\su{2n+1} - \Mu{n}).
\end{split}\end{equation}
 According to \rpropp{M8.1.21.}{M8.1.21.a}, the sequence \(\seqs{2n}\) belongs to \(\Fggqu{2n}\).
 Consequently, \rprop{L8.1.33.} shows that \(\seqs{2n} \in\Lggequ{2n}\).
 Therefore, because of \rlemp{FR435.}{FR435.b}, equation \eqref{FT2} holds true.
 By virtue of \rnotap{N41}{N41.c} and \eqref{FT2}, we have then 
\begin{multline}\label{LST}
 \Lbu{n}
 = \sbu{2n} -\Tripb{n}\\
 = -\su{2n+1} + \Mu{n} +\Lu{n} \rk*{ \beta \Iq  +  \ek*{\Lbu{n-1}}^\mpi\ek*{\sbu{2n-1} - \Mbu{n-1}}}.
\end{multline}
 Since \(\seqs{2n+1}\) belongs to \(\Kggequ{2n+1}\), we see from \rpropp{L5.1.30.}{L5.1.30.a} that \(\cN (\Lu{n}) \subseteq \nul{\Lau{n}}\).
 Hence, \rrem{A.7.8-1.} shows that \(\Lau{n} \Lu{n}^\mpi  \Lu{n} = \Lau{n}\).
 Thus, because of \eqref{LST}, we obtain
\begin{multline}\label{LST8.1.23}
 \Lau{n}  \Lu{n}^\mpi  \Lbu{n}\\
 = - \Lau{n} \Lu{n}^\mpi  (\su{2n+1} - \Mu{n}) + \Lau{n} \rk*{ \beta \Iq  +  \ek*{\Lbu{n-1}}^\mpi\ek*{\sbu{2n-1} - \Mbu{n-1}} }.
\end{multline}
 Taking \(\seqs{2n}\in\Fggqu{2n}\)  and \rpropp{L8.1.24.}{L8.1.24.a} into account, we get \((\sbu{j})_{j=0}^{2n-1}\in\Fggqu{2n-1}\).
 By virtue of \rprop{L8.1.33.}, this implies \((\sbu{j})_{j=0}^{2n-1}\in\Kggequ{2n-1}\).
 Applying \rlemp{FR433.}{FR433.a} to the sequence \((\sbu{j})_{j=0}^{2n-1}\) and using \rrem{R5.2.}, we infer
\[\begin{split}
 \Tripb{n} 
 &= \alpha \sbu{2n-1} + (_\alpha (M_\beta))_{n-1} + (_\alpha (L_\beta))_{n-1} \ek*{\Lbu{n-1}}^\mpi\ek*{\sbu{2n-1} - \Mbu{n-1}}\\
 &= \alpha \sbu{2n-1} + \Mcu{n-1} + \Lcu{n-1} \ek*{\Lbu{n-1}}^\mpi\ek*{\sbu{2n-1} - \Mbu{n-1}}.
\end{split}\]
 In view of \rnotap{N41}{N41.c} and \rrem{R5.2.}, then
\begin{equation}\label{CMG}\begin{split}
 \Lbu{n} 
 &= \sbu{2n} -  \Tripb{n}\\
 &= \scu{2n-1} - \Mcu{n-1} -  \Lcu{n-1} \ek*{\Lbu{n-1}}^\mpi  \ek*{\sbu{2n-1} - \Mbu{n-1}}
\end{split}\end{equation}
 follows.
 Since  \((\sau{j})_{j=0}^{2n}\in\Lggequ{2n}\) is true, \rpropp{L5.1.30.}{L5.1.30.d} yields \(\nul{((_\alpha L)_\beta)_{n-1}} \subseteq \nul{\Lau{n}}\).
 Since \rrem{R5.2.} shows that \(((_\alpha L)_\beta)_{n-1} = \Lcu{n-1}\), then \rrem{A.7.8-1.} yields \(\Lau{n} \Lcu{n-1}^\mpi  \Lcu{n-1} = \Lau{n}\).
 Consequently, from \eqref{CMG} we obtain 
\begin{multline}\label{FCAG1}
 \Lau{n}  \Lcu{n-1}^\mpi    \Lbu{n}\\
 = \Lau{n}  \Lcu{n-1}^\mpi   \ek*{\scu{2n-1} - \Mcu{n-1}} -  \Lau{n} \ek*{\Lbu{n-1}}^\mpi    \ek*{\sbu{2n-1} - \Mbu{n-1}} .
\end{multline}
 Combining \eqref{LST8.1.23} and \eqref{FCAG1}, we conclude
\begin{multline}\label{GL5}
 \Lau{n} \Lu{n}^\mpi  \Lbu{n} + \Lau{n} \Lcu{n-1}^\mpi   \Lbu{n} \\
 = - \Lau{n} \Lu{n}^\mpi  (\su{2n+1} -\Mu{n}) + \beta  \Lau{n} +  \Lau{n} \Lcu{n-1}^\mpi  \ek*{\scu{2n-1} - \Mcu{n-1}}.
\end{multline}
 Comparing \eqref{GL5} with \eqref{TP3} and \eqref{IUA}, we get \(\dia{2n+1} = \ba  (\usc{2n+1}\ps  \osc{2n+1})\). 
\end{proof}

 The following result contains an important property of the interval lengths introduced in \rdefn{D1516}:
\bpropl{F8.1.44.}
 Let \(\ug\in\R\), let \(\og\in(\ug,\infp)\), let \(\kappa\in\NOinf\), and let \(\seqska\) be a sequence from \(\Cqq\).
 Then:
 \benui
  \il{F8.1.44.a} If \(\seqska\in\Fggqu{\kappa}\), then \(\dia{j} \in\Cggq \) for each \(j\in\mn{0}{\kappa}\).
  \il{F8.1.44.b} If \(\seqska\in\Fgqu{\kappa}\), then \(\dia{j} \in\Cgq \) for each \(j\in\mn{0}{\kappa}\).
 \eenui
\eprop
\begin{proof}
 Use \rthm{M8.1.43.}, \rlemss{L1433}{MA.8.7*.}, and \rrem{L0751}.
\end{proof}

 Observe that \rprop{F8.1.44.} is the reason that \((\dia{j})_{j=0}^\kappa\) is said to be the \tseqdiao{\(\seqska \)}.

\bpropl{M8.1.37-1.}
 Let \(\ug\in\R\), let \(\og\in(\ug,\infp)\), let  \(m\in\NO \), and let \( \seqs{m+1}\) be a sequence from \(\Cqq\).
 Then:
 \benui
  \il{M8.1.37-1.a} If \(\seqs{m+1}\in\Fggqu{m+1}\), then \(\su{m+1}\in[\umg{m},\omg{m}]\).
  \il{M8.1.37-1.b} If \(\seqs{m+1}\in\Fgqu{m+1}\), then \(\su{m+1}\in(\umg{m},\omg{m})\).
 \eenui
\eprop
\begin{proof}
 Suppose \(\seqs{m}\in\Fggqu{m}\).
 According to  \rrem{B8.1.28.}, we have  \(\umg{m}^\ad  = \umg{m}\) and \(\omg{m}^\ad  = \omg{m}\).
 \rprop{S8.1.23.} provides us \(\seqs{m} \in\Hggequ{m}\).
 From \(\seqs{m+1} \in\Fggqu{m+1}\) and \rlem{L1319} we know that \(\su{m+1}^\ad  =\su{m+1}\).
 Since \(\seqs{m+1}\in\Fggqu{m+1}\) and \rlemp{L1433}{L1433.a} show that  \(\set{\usc{m+1}, \osc{m+1}} \subseteq  \Cggq \) holds true, in view of \rdefn{D1730}, then~\eqref{M8.1.37-1.a} follows.
 If \(\seqs{m}\) even belongs to \(\Fgqu{m}\), \rlemp{L1433}{L1433.b} shows that \(\usc{m+1}\) and \(\osc{m+1}\) are \tpH{}.
 In view of \rdefn{D1730}, then~\eqref{M8.1.37-1.b} follows.
\end{proof}

\bcorl{C1026}
 Let \(\ug\in\R\), let \(\og\in(\ug,\infp)\), let \(\kappa\in\NOinf\), and let  \(\seqska\) be a sequence from \(\Cqq\).
 Then:
 \benui
  \il{C1026.a} If \(\seqska\in\Fggqu{\kappa}\),then,
 \beql{C1026.1}
  \ug\su{2k}
  \lleq\umg{2k}
  \lleq\cen{2k}
  \lleq\omg{2k}
  \lleq\og\su{2k}
 \eeq
 for all \(k\in\N\) with \(2k\leq\kappa\), and
 \beql{C1026.2}
  \Oqq
  \lleq\umg{2k+1}
  \lleq\cen{2k+1}
  \lleq\omg{2k+1}
  \lleq-\ug\og\su{2k}+\rk{\ug+\og}\su{2k+1}
 \eeq
 for all \(k\in\N\) with \(2k+1\leq\kappa\).
  \il{C1026.b} If \(\seqska\in\Fgqu{\kappa}\),then,
 \beql{C1026.3}
  \ug\su{2k}
  <\umg{2k}
  <\cen{2k}
  <\omg{2k}
  <\og\su{2k}
 \eeq
 for all \(k\in\minf{2}\) with \(2k\leq\kappa\), and
 \beql{C1026.4}
  \Oqq
  <\umg{2k+1}
  <\cen{2k+1}
  <\omg{2k+1}
  <-\ug\og\su{2k}+\rk{\ug+\og}\su{2k+1}
 \eeq
 for all \(k\in\minf{2}\) with \(2k+1\leq\kappa\).
 \eenui
\ecor
\bproof
 Suppose \(\seqs{\kappa}\in\Fggqu{\kappa}\) (resp.\ \(\seqs{\kappa}\in\Fgqu{\kappa}\)).
 In view of \rlem{L1319} and \rremss{B8.1.28.}{R1535}, all the matrices involved are \tH{}.
 We consider a number \(k\in\N\) (resp.\ \(k\in\minf{2}\)).
 
 First assume that \(2k\leq\kappa\).
 Because of \rprop{F8.1.44.}, then \(\dia{2k}\) is \tnnH{} (resp.\ \tpH{}).
 Since \(\cen{2k}-\umg{2k}=\frac{1}{2}\dia{2k}\) and \(\omg{2k}-\cen{2k}=\frac{1}{2}\dia{2k}\), by virtue of \rdefn{D1516}, the two inner inequalities in \eqref{C1026.1} (resp.\ \eqref{C1026.3}) follow.
 According to \rprop{M8.1.21.}, we have \(\seqs{2k-1}\in\Fggqu{2k-1}\) (resp.\ \(\seqs{2k-1}\in\Fgqu{2k-1}\)).
 In view of \rdefnp{D1159}{D1159.a}, hence \(\set{\seqsa{2k-2},\seqsb{2k-2}}\subseteq\Hggqu{2k-2}\) (resp.\ \(\set{\seqsa{2k-2},\seqsb{2k-2}}\subseteq\Hgqu{2k-2}\)).
 Since the matrices \(\sau{2k-1}\) and \(\sbu{2k-1}\) are, according to \rlem{L1319} and \rnota{N5.1}, both \tH{}, we see from \rlem{L1548} that \(\Tripa{k}\) and \(\Tripb{k}\) are both \tnnH{} (resp.\ \tpH{}).
 Consequently, the two outer inequalities in \eqref{C1026.1} (resp.\ \eqref{C1026.3}) can be seen from \rdefn{D1719}.
 
 Now assume that \(2k+1\leq\kappa\).
 Because of \rprop{F8.1.44.}, then \(\dia{2k+1}\) is \tnnH{} (resp.\ \tpH{}).
 Since \(\cen{2k+1}-\umg{2k+1}=\frac{1}{2}\dia{2k+1}\) and \(\omg{2k+1}-\cen{2k+1}=\frac{1}{2}\dia{2k+1}\) by virtue of \rdefn{D1516}, the two inner inequalities in \eqref{C1026.2} (resp.\ \eqref{C1026.4}) follow.
 According to \rprop{M8.1.21.}, we have \(\seqs{2k}\in\Fggqu{2k}\) (resp.\ \(\seqs{2k}\in\Fgqu{2k}\)).
 In view of \rdefnp{D1159}{D1159.b}, hence \(\seqs{2k}\in\Hggqu{2k}\) and \(\seqsc{2k-2}\in\Hggqu{2k-2}\) (resp.\ \(\seqs{2k}\in\Hgqu{2k}\) and \(\seqsc{2k-2}\in\Hgqu{2k-2}\)).
 Since the matrices \(\su{2k+1}\) and \(\scu{2k-1}\) are, according to \rlem{L1319} and \rnota{N5.1}, both \tH{}, we see from \rlem{L1548} that \(\Trip{k+1}\) and \(\Tripc{k}\) are both \tnnH{} (resp.\ \tpH{}).
 Consequently, the two outer inequalities in \eqref{C1026.2} (resp.\ \eqref{C1026.4}) can be seen from \rdefn{D1719}.
\eproof

\bpropl{S8.1.48.}
 Let \(\ug\in\R\), let \(\og\in(\ug,\infp)\), let \(\kappa\in\NOinf\), and let  \(\seqs{\kappa}\in\Fggqu{\kappa}\).
 Then \(\cR (\dia{0}) = \cR (\usc{0})\) and \(\cN (\dia{0}) = \cN (\usc{0})\).
 If \(\kappa \ge 1\), then
\begin{align}\label{TSD}
 \cR (\dia{j})&= \cR (\usc{j}) \cap \cR (\osc{j})&
 &\text{and}&
 \cN (\dia{j})&= \cN (\usc{j}) + \cN (\osc{j})
\end{align}
 hold true for all \(j\in\mn{1}{\kappa}\) and
\begin{align}\label{TST}
 \cR (\dia{j})&= \cR (\usc{j+1}) + \cR (\osc{j+1})&
 &\text{and}&
 \cN (\dia{j})&= \cN (\usc{j+1}) \cap \cN (\osc{j+1})
\end{align}
 are fulfilled for all \(j\in\mn{0}{\kappa -1}\).
\eprop
\begin{proof}
 Because of \(\alpha < \beta\) and \rthm{M8.1.43.}, we have \(\cR (\dia{0}) = \cR (\usc{0})\) and \(\cN (\dia{0}) = \cN (\usc{0})\).
 Assume now \(\kappa \ge 1\).
 Then \rlemp{L1433}{L1433.a} provides us \(\set{\usc{j}, \osc{j}}\subseteq \Cggq \) for all \(j\in\mn{1}{\kappa}\).
 Using \rthm{M8.1.43.}, \rlemss{MA.8.7*.}{MA.8.11.}, and \(\alpha <\beta\), we get \eqref{TSD} for all \(j\in\mn{1}{\kappa}\).
 By virtue of \rremss{M8.1.40.}{A.20.}, we also obtain \eqref{TST} for each \(j\in\mn{0}{\kappa -1}\).
\end{proof}

 The following result improves \rlem{L8.1.37.}.
\bcorl{C0834}
 Let \(\ug\in\R\), let \(\og\in(\ug,\infp)\), let \(\kappa\in\Ninf\), and let  \(\seqs{\kappa}\in\Fggqu{\kappa}\).
 Then \(\cR (\usc{j}) +\cR (\osc{j})=\cR (\usc{j-1})\cap \cR (\osc{j-1})\) and \(\nul{\usc{j-1}}+\nul{\osc{j-1}}=\nul{\usc{j}}\cap\nul{\osc{j}}\) for all \(j\in\mn{2}{\kappa}\)
\ecor
\bproof
 This is a direct consequence of \rprop{S8.1.48.}.
\eproof

\bcorl{M8.1.49.}
 Let \(\ug\in\R\), let \(\og\in(\ug,\infp)\), let \(\kappa\in\Ninf\), and let \(\seqs{\kappa}\in\Fggqu{\kappa}\).
 Then \( \cR (\dia{j}) \subseteq \cR (\dia{j-1})\), \(\cN (\dia{j-1}) \subseteq \cN (\dia{j})\), \(\cR (\usc{j})\subseteq  \cR (\usc{j-1})\), and \(\cN (\usc{j-1})\subseteq \cN (\usc{j})\) for all \(j\in\mn{1}{\kappa}\).
 Furthermore, if \(\kappa \ge 2\), then \(\cR (\osc{j})\subseteq \cR (\osc{j-1})\) and \(\cN (\osc{j-1})\subseteq \cN (\osc{j})\) for each \(j\in\mn{2}{\kappa}\).
\ecor
\begin{proof}
 \rprop{S8.1.48.} provides us \(\cR (\dia{j}) = \cR (\usc{j}) \cap \cR (\osc{j}) \subseteq \cR (\usc{j}) + \cR (\osc{j}) = \cR (\dia{j-1}) \) for all \(j\in\mn{1}{\kappa}\) and \(\cR (\usc{j})\subseteq \cR (\usc{j}) + \cR (\osc{j}) = \cR  (\dia{j-1}) = \cR (\usc{j-1})\cap \cR (\osc{j-1})\subseteq \cR (\usc{j-1})\) as well as \( \cR (\osc{j}) \subseteq \cR (\dia{j-1}) \subseteq \cR (\osc{j-1})\) for all \(j\in\mn{2}{\kappa}\).
 \rprop{S8.1.48.}, \rrem{M8.1.40.}, and  \rdefn{D1730} also yield \(\cR (\usc{1})\subseteq \cR (\usc{1}) + \cR (\osc{1}) = \cR  (\dia{0}) = \cR (\su{0}) =\cR (\usc{0})\).
 The asserted inclusions for the null spaces follow analogously.
\end{proof}

\bcorl{M8.1.55}
 Let \(\ug\in\R\), let \(\og\in(\ug,\infp)\), let \(\kappa\in\Ninf\), and let \(\seqska \in\Fggqu{\kappa}\).
 For all \(j\in\mn{1}{\kappa}\), then
 \begin{align*}
  \dia{j}&=\ba\usc{j}\dia{j-1}^\mpi\osc{j}&
 &\text{and}&
  \dia{j}&=\ba\osc{j}\dia{j-1}^\mpi\usc{j}.
 \end{align*}
\ecor
\begin{proof}
 In view of \eqref{ps}, this is a consequence of \rthm{M8.1.43.} and \rrem{M8.1.40.}.
\end{proof}

 Now we are going to investigate the sequence of matricial interval lengths \(\seq{\dia{j}}{j}{0}{\kappa}\) associated with a sequence \(\seqska\in\Fggqu{\kappa}\).
 Using \rthm{M8.1.43.} and the arithmetics of the parallel sum of matrices, we derive a formula connecting two consecutive interval lengths \(\dia{j+1}\) and \(\dia{j}\):
\bpropl{P0959}
 Let \(\ug\in\R\), let \(\og\in(\ug,\infp)\), let \(\kappa\in\Ninf\), and let \(\seqska \in\Fggqu{\kappa}\).
 For all \(j\in\mn{0}{\kappa-1}\), then
 \[
  \dia{j+1}=\frac{\og-\ug}{4}\dia{j}-\ba\rk{\su{j+1}-\cen{j}}\dia{j}^\mpi\rk{\su{j+1}-\cen{j}}.
 \]
\eprop
\begin{proof}
 We consider an arbitrary \(j\in\mn{0}{\kappa-1}\).
 According to \rlemp{L1433}{L1433.a}, the matrices \(\usc{j+1}\) and \(\osc{j+1}\) are both \tnnH{}.
 In view of \rlemss{MA.8.7*.}{MA.8.11}, then \((\usc{j+1}+\osc{j+1})-4(\usc{j+1}\ps \osc{j+1})=(\usc{j+1}-\osc{j+1})(\usc{j+1}+\osc{j+1})^\mpi(\usc{j+1}-\osc{j+1})\).
 Using \rrem{M8.1.40.} and \rthm{M8.1.43.}, we get \(\ba\dia{j}-4\dia{j+1}=\ba(\usc{j+1}-\osc{j+1})\dia{j}^\mpi(\usc{j+1}-\osc{j+1})\).
 Since \(2(\su{j+1}-\cen{j})=\usc{j+1}-\osc{j+1}\) holds true, we obtain \(4\dia{j+1}=\ba\dia{j}-4\ba\rk{\su{j+1}-\cen{j}}\dia{j}^\mpi\rk{\su{j+1}-\cen{j}}\).
\end{proof}

 \rprop{P0959} leads us now to an important monotonicity property for the sequence of interval lengths.
 Moreover, the extremal role of the midpoint of the interval in this context will be clear.
\bpropl{C1521}
 Let \(\ug\in\R\), let \(\og\in(\ug,\infp)\), let \(\kappa\in\Ninf\), let \(\seqska \in\Fggqu{\kappa}\), and let \(m\in\mn{1}{\kappa}\).
 Then \(\dia{m}\lleq\frac{\og-\ug}{4}\dia{m-1}\). Furthermore, \(\dia{m}=\frac{\og-\ug}{4}\dia{m-1}\) if and only if \(\su{m}=\cen{m-1}\).
\eprop
\begin{proof}
 The matrices \(\dia{m-1}\) and \(\dia{m}\) are both \tnnH{} according to \rprop{F8.1.44.}.
 Let \(D\defeq\frac{\og-\ug}{4}\dia{m-1}-\dia{m}\).
 Because of \rprop{P0959}, we have \(D=\ba\rk{\su{m}-\cen{m-1}}\dia{m-1}^\mpi\rk{\su{m}-\cen{m-1}}\).
 Since the matrices \(\su{m}\) and \(\cen{m-1}\) are both \tH{} by virtue of \rlem{L1319} and \rrem{R1535}, hence 
 \beql{C1521.1}
  D
  =\ba\rk{\su{m}-\cen{m-1}}^\ad\dia{m-1}^\mpi\rk{\su{m}-\cen{m-1}}.
 \eeq
 According to \rrem{R1546}, the matrix \(\dia{m-1}^\mpi\) is \tnnH{}.
 In view of \(\og>\ug\), then \(D\in\Cggq\) follows.
 
 Now assume \(D=\Oqq\). Because of \(\og>\ug\) and \(\dia{m-1}^\mpi\in\Cggq\), we obtain from \eqref{C1521.1} then \(\sqrt{\dia{m-1}^\mpi}\rk{\su{m}-\cen{m-1}}=\Oqq\).
 In particular, \(\dia{m-1}\dia{m-1}^\mpi\rk{\su{m}-\cen{m-1}}=\Oqq\).
 In view of \rdefnss{D1516}{D1730}, we get
 \(
  \su{m}-\cen{m-1}
  =\frac{1}{2}\rk{\usc{m}-\osc{m}}
 \).
 Hence, \(\ran{\su{m}-\cen{m-1}}\subseteq\ran{\usc{m}}+\ran{\osc{m}}\). Since \(\ran{\usc{m}}+\ran{\osc{m}}=\ran{\dia{m-1}}\) by virtue of \rprop{S8.1.48.}, we have \(\ran{\su{m}-\cen{m-1}}\subseteq\ran{\dia{m-1}}\). Using \rrem{MA.7.8.}, we obtain \(\su{m}-\cen{m-1}=\dia{m-1}\dia{m-1}^\mpi\rk{\su{m}-\cen{m-1}}\) and hence \(\su{m}-\cen{m-1}=\Oqq\).

 Conversely, if \(\su{m}=\cen{m-1}\), then \(D=\Oqq\) follows from \eqref{C1521.1}.
\end{proof}

 We turn our the attention to an interesting subclass of \tabHnnd{} sequences:
\bdefnl{D0927}
 Let \(\ug\in\R\), let \(\og\in(\ug,\infp)\), let \(m\in\NO\) and let \(\seqs{m}\in\Fggqu{m}\).
 Then \(\seqs{m}\) is called \notii{\habHd{}} if \(\dia{m}=\Oqq\), where \(\dia{m}\) is given in \rdefn{D1516}.
 We denote by \symba{\Fggdqu{m}}{f} the set of all sequences \(\seqs{m}\in\Fggqu{m}\) which are \tabHd{}.
\edefn

\bdefnl{D0933}
 Let \(\ug\in\R\), let \(\og\in(\ug,\infp)\), and let \(\seqsinf\in\Fggqinf\).
\benui
 \il{D0933.a} Let \(m\in\NO\).
 Then \(\seqsinf\) is called \notii{\habHdo{m}} if \(\seqs{m}\in\Fggdqu{m}\).
 \il{D0933.b} The sequence \(\seqsinf\) is called \notii{\habHd{}} if there exists an \(m\in\NO\) such that \(\seqsinf\) is \tHdo{m}.
\eenui
 We denote by \symba{\Fggdqinf}{f} the set of all sequences \(\seqsinf\in\Fggqinf\) which are \tabHd{}.
\edefn

\bleml{L0901}
 Let \(\ug\in\R\), let \(\og\in(\ug,\infp)\), let \(m\in\N\), and let \(\seqs{m}\in\Fggqu{m}\).
 Then \(\seqs{m}\) is \tabHd{} if and only if \(\ran{\usc{m}}\cap\ran{\osc{m}}=\set{\Ouu{q}{1}}\).
\elem
\bproof
 Use \rprop{S8.1.48.}.
\eproof

\bleml{L1230}
 Let \(\ug\in\R\), let \(\og\in(\ug,\infp)\), let \(m\in\NO\), and let \(\seqsinf\in\Fggqinf\) be \tabHdo{m}. Then:
 \benui
  \il{L1230.a} If \(m=2n\) with some \(n\in\NO\), then \(\seqsinf\) is \tHdo{n+1} and the sequences $\seqsainf$, $\seqsbinf$, and $\seqscinf$ are \tHdo{n}.
  \il{L1230.b} If \(m=2n+1\) with some \(n\in\NO\), then \(\seqscinf\) is \tHdo{n} and the sequences $\seqsinf$, $\seqsainf$, and $\seqsbinf$ are \tHdo{n+1}.
 \eenui
\elem
\bproof
 First observe that the sequences \(\seqsinf\), \(\seqsainf\), \(\seqsbinf\), and \(\seqscinf\) belong to \(\Hggqinf\) by virtue of \rlem{L0923} and \rprop{L8.1.24.}.
 In view of \rdefnss{D0933}{D0927}, we have \(\dia{m}=\Oqq\).
 According to \rprop{S8.1.48.}, we get \(\ran{\usc{m+1}}+\ran{\osc{m+1}}=\set{\Ouu{q}{1}}\).
 In particular, \(\ran{\usc{m+1}}=\set{\Ouu{q}{1}}\) and \(\ran{\osc{m+1}}=\set{\Ouu{q}{1}}\).
 Using \rcor{M8.1.49.}, this implies \(\ran{\usc{m+2}}=\set{\Ouu{q}{1}}\) and \(\ran{\osc{m+2}}=\set{\Ouu{q}{1}}\).
 Thus, each of the matrices \(\usc{m+1}\), \(\osc{m+1}\), \(\usc{m+2}\), and \(\osc{m+2}\) is the zero matrix \(\Oqq\).
 
 \eqref{L1230.a} In this situation each of the matrices \(\usc{2n+1}\), \(\osc{2n+1}\), \(\usc{2n+2}\), and \(\osc{2n+2}\) is the zero matrix \(\Oqq\).
 In view of \rrem{M8.1.6.}, hence \(\Lau{n}=\Oqq\), \(\Lbu{n}=\Oqq\), \(\Lu{n+1}=\Oqq\), and \(\Lcu{n}=\Oqq\).
 
 \eqref{L1230.b} In this situation each of the matrices \(\usc{2n+2}\), \(\osc{2n+2}\), \(\usc{2n+3}\), and \(\osc{2n+3}\) is the zero matrix \(\Oqq\).
 In view of \rrem{M8.1.6.}, hence \(\Lu{n+1}=\Oqq\), \(\Lcu{n}=\Oqq\), \(\Lau{n+1}=\Oqq\), and \(\Lbu{n+1}=\Oqq\).
\eproof

\bleml{L0922}
 Let \(\ug\in\R\) and let \(\og\in(\ug,\infp)\).
 Then \(\Fggdqinf\subseteq\Hggdqinf\).
\elem
\bproof
 This is a direct consequence of \rlem{L1230}.
\eproof

\bleml{L1604}
 Let \(\ug\in\R\), let \(\og\in(\ug,\infp)\), let \(\seqsinf\in\Fggqinf\), and let \(m\in\NO\).
 If \(\seqsinf\) is \tabHdo{m}, then \(\seqsinf\) is \tabHdo{\ell} for all \(\ell\in\minf{m}\).
\elem
\bproof
 This can be seen from \rcor{M8.1.49.}
\eproof

\bpropl{L0942}
 Let \(\ug\in\R\), let \(\og\in(\ug,\infp)\), let \(m\in\NO\), and let \(\seqsinf\in\Fggqinf\) be \tabHdo{m}. Then \(\su{j}=\cen{j-1}=\umg{j-1}=\omg{j-1}\) for all \(j\in\minf{m+1}\).
\eprop
\bproof
 Let \(\ell\in\minf{m}\). In view of \rlem{L1604} and \rdefnss{D0933}{D0927}, then \(\dia{\ell}=\Oqq\).
 Because of \rdefn{D1516}, hence \(\umg{\ell}=\omg{\ell}\) and thus \(\cen{\ell}=\umg{\ell}\).
 In particular, \([\umg{\ell},\omg{\ell}]=\set{\cen{\ell}}\).
 According to \rdefn{D0846}, furthermore \(\seqs{\ell+1}\in\Fggqu{\ell+1}\). \rpropp{M8.1.37-1.}{M8.1.37-1.a} yields then \(\su{\ell+1}\in[\umg{\ell},\omg{\ell}]\).
 Since \([\umg{\ell},\omg{\ell}]=\set{\cen{\ell}}\), this implies \(\su{\ell+1}=\cen{\ell}\).
\eproof

\bpropl{L1610}
 Let \(\ug\in\R\), let \(\og\in(\ug,\infp)\), let \(m\in\NO\), and let \(\seqs{m}\in\Fggdqu{m}$.
 Then there exists a unique sequence \(\seq{\su{j}}{j}{m+1}{\infi}\) from \(\Cqq\) such that \(\seqsinf\in\Fggqinf\).
\eprop
\bproof
 Combine \rpropss{P0710}{L0942}.
\eproof

 \rprop{L1610} leads us to the following notion:
\bdefnl{D0839}
 Let \(\ug\in\R\), let \(\og\in(\ug,\infp)\), let \(m\in\NO\), and let \(\seqs{m}\in\Fggdqu{m}$.
 Let \(\seq{\su{j}}{j}{m+1}{\infi}\) be the unique sequence from \(\Cqq\) such that \(\seqsinf\in\Fggqinf\).
 Then \(\seqsinf\) is called the \noti{\tabHdso{\(\seqs{m}\)}}{h@\habHdso{\(\seqs{m}\)}}.
\edefn

 The considerations of this section lead us to a further interesting subclass of \(\Fggqinf\):
\bdefnl{D1509}
 Let \(\ug\in\R\), let \(\og\in(\ug,\infp)\), and let \(\seqsinf\in\Fggqinf\). 
 If \(m\in\N\) is such that \(\su{j}=\cen{j-1}\) for all \(j\in\minf{m}\), where \(\cen{j-1}\) is given by \rdefn{D1516}, then \(\seqsinf\) is called \notii{\habHco{m}}.
 If there exists an \(\ell\in\N\) such that \(\seqsinf\) is \tabHco{\ell}, then \(\seqsinf\) is simply called \notii{\habHc{}}.
\edefn

\breml{R1820}
 Let \(\ug\in\R\), let \(\og\in(\ug,\infp)\), let \(\seqsinf\in\Fggqinf\), and let \(m\in\N\).
 If \(\seqsinf\) is \tabHco{m}, then \(\seqsinf\) is \tabHco{\ell} for all \(\ell\in\minf{m}\).
\erem

\bpropl{L1526}
 Let \(\ug\in\R\), let \(\og\in(\ug,\infp)\), let \(m\in\N\), and let \(\seqsinf\in\Fggqinf\). Then \(\seqsinf\) is \tabHco{m} if and only if \(\dia{\ell}=\frac{\og-\ug}{4}\dia{\ell-1}\) for all \(\ell\in\minf{m}\).
\eprop
\bproof
 This is a direct consequence of \rprop{C1521}.
\eproof

 A closer look at \rpropss{C1521}{L1526} shows that the role of the central sequences in the context of \tabHnnd{} sequences is comparable with the role of the central sequences in the context of \tnn{} definite sequences from \(\Cqq\) or of the central sequences of \tpqa{Schur} sequences.
 On the one side we see that starting from some index the interval lengths are maximal with respect to the L\"owner semi-ordering in the set \(\CHq\).
 On the other side~\zitaa{MR885621}{\clem{6}} and~\zitaa{MR918682}{\clem{7}} indicate that the semi-radii of the corresponding matrix balls are maximal with respect to the L\"owner semi-ordering in the set \(\CHq\) as well for central \tnn{} definite sequences from \(\Cqq\)  as for central \tpqa{Schur} sequences, resp.

 In view of \rdefn{D1509}, we obtain from \rprop{L0942} the following result:
\breml{L1801}
 Let \(\ug\in\R\), let \(\og\in(\ug,\infp)\), let \(m\in\NO\), and let \(\seqsinf\in\Fggqinf\) be \tabHdo{m}. Then \(\seqsinf\) is \tabHco{m+1}.
\erem

\bleml{L1738}
 Let \(\ug\in\R\), let \(\og\in(\ug,\infp)\), let \(m\in\NO\), and let \(\seqs{m}\in\Fggdqu{m}\). Then the \tabHdso{\(\seqs{m}\)} is \tabHdo{m} and \tabHco{m+1}.
\elem
\bproof
 The \tabHd{} sequence \(\seqsinf\) associated to \(\seqs{m}\) belongs to \(\Fggqinf\) by definition and fulfills \(\seqs{m}\in\Fggdqu{m}\).
 In view of \rdefnp{D0933}{D0933.a}, hence \(\seqsinf\) is \tabHdo{m}.
 According to \rrem{L1801}, then \(\seqsinf\) is \tabHco{m+1}.
\eproof

 From \rlemss{L1738}{L1230} we obtain the following result:
\breml{R0735}
 Let \(\ug\in\R\), let \(\og\in(\ug,\infp)\), let \(m\in\NO\), and let \(\seqs{m}\in\Fggdqu{m}\). Then the \tabHdso{\(\seqs{m}\)} is \tHdo{n+1}, where \(n\) is the unique number from \(\NO\) with \(m=2n\) or \(m=2n+1\).
\erem

 \rsec{S0747} has some points of touch with the paper Dette/Studden~\zita{MR1883272}, where the moment space of a matrix measure \(\sigma\in\Mggqa{[0,1]}\) is studied and where it is assumed that the corresponding block Hankel matrices are non-singular.
 The strategy used in~\zita{MR1883272} is mainly based on using convexity techniques and constructing canonical representations for the points of the moment space as this was done in the classical scalar case by M.~G.~Krein~\zita{MR0044591} (see also Krein/Nudelman~\zitaa{MR0458081}{\cchap{3}}).
An important tool used in~\zita{MR1883272} is given by a matricial generalization of the theory of classical canonical moments (see the monograph Dette/Studden~\zita{MR1468473}).

\section{On the problem of \habHnnd{} extension}\label{S1501}
 
 Let \(m\in\NO\) and let \(\seqs{m}\in\Fggqu{m}\).
 Then against to the background of \rprop{M8.1.37-1.} we will be able to describe the set \(\setaa{\su{m+1}\in\Cqq}{\seqs{m+1}\in\Fggqu{m+1}}\).
\bleml{MTM}
 Let \(\ug\in\R\), let \(\og\in(\ug,\infp)\), let  \(m\in\NO \), and let \(\seqs{m}\) be a sequence from \(\Cqq\). Then:
 \benui
  \il{MTM.a} If \(\seqs{m}\in\Fggqu{m}\), then \([\umg{m},\omg{m}]\neq\emptyset\) and \(\seqs{m+1}\in\Fggqu{m+1}\) for all \(\su{m+1}\in[\umg{m},\omg{m}]\).
  \il{MTM.b} If \(\seqs{m}\in\Fgqu{m}\), then \((\umg{m},\omg{m})\neq\emptyset\) and \(\seqs{m+1}\in\Fgqu{m+1}\) for all \(\su{m+1}\in(\umg{m},\omg{m})\).
 \eenui
\elem
\begin{proof}
 Assume \(\seqs{m}\in\Fggqu{m}\) (resp.\ \(\seqs{m}\in\Fgqu{m}\)).
 From \rprop{F8.1.44.} we know that \(\dia{m}\) it \tnnH{} (resp.\ \tpH{}).
 In view of \rdefn{D1516}, \rlem{L0719} yields then \([\umg{m},\omg{m}]\neq\emptyset\) (resp.\ \((\umg{m},\omg{m})\neq\emptyset\)). \rprop{S8.1.23.} provides us \(\seqs{m} \in\Hggequ{m}\) (resp.\ \(\seqs{m} \in\Hgequ{m}\)).
 We consider an arbitrary \(\su{m+1}\in[\umg{m},\omg{m}]\) (resp.\ \(\su{m+1}\in(\umg{m},\omg{m})\)). Then \(\su{m+1}\in\CHq\) and \(\umg{m}\lleq\su{m+1}\lleq\omg{m}\) (resp.\ \(\umg{m}<\su{m+1}<\omg{m}\)).

 First we discuss the case \(m=0\).
 In view of \rrem{R1531} and \rnota{N5.1},  then \(\Hau{0}\) and \(\Hbu{0}\) are both \tnnH{} (resp.\ \tpH{}).
 Thus,  \(\seqs{1}\in\Fggqu{1}\) (resp.\ \(\seqs{1}\in\Fgqu{1}\)) follows from \rdefnp{D1159}{D1159.a}.
 
 Now let \(m=1\).
 Then \rrem{R1531} and \rnota{N5.1} yield \(\Trip{1}\le \su{2}\) and \(\Hcu{0}\in\Cggq \) (resp.\ \(\Trip{1}<\su{2}\) and \(\Hcu{0}\in\Cgq \)).
 Since \(\seqs{1}\)  belongs to \(\Hggequ{1}\) (resp.\ \(\Hgequ{1}\)), then we see from \rpropss{P0906}{P1033} that \(\seqs{2}\) belongs to \(\Hggqu{2}\), (resp.\ \(\Hgqu{2}\)) \ie{}, the matrix \(\Hu{1}\) is \tnnH{} (resp.\ \tpH{}).
 Thus, \rdefnp{D1159}{D1159.b} shows that \(\seqs{2} \in\Fggqu{2}\) (resp.\ \(\seqs{2} \in\Fgqu{2}\)).
 
 Now we consider the case that \(m=2n\) is valid where \(n\) is some positive integer.
 According to  \(\umg{2n}\le\su{2n+1}\le \omg{2n}\) (resp.\ \(\umg{2n}<\su{2n+1}<\omg{2n}\)), \rdefn{D1719} and \rnota{N5.1}, then 
\begin{align}\label{TAB}
 \Tripa{n}&\le \sau{2n},&
 \Tripb{n}&\le \sbu{2n}&
 \text{(resp.\ }\Tripa{n}&<\sau{2n},&
 \Tripb{n}&<\sbu{2n}\text{)}.
\end{align}
 Because of \rpropp{L8.1.24.}{L8.1.24.a} and \rprop{S8.1.23.}, we have \(\set{ (\sau{j})_{j=0}^{2n-1}, (\sbu{j})_{j=0}^{2n-1}}\subseteq \Fggqu{2n-1}\subseteq \Hggequ{2n-1}\).
 For \(\seqs{2n}\in\Fgqu{2n}\), we obtain \(\seqs{2n-1}\in\Fgqu{2n-1}\) by virtue of \rpropp{M8.1.21.}{M8.1.21.b} and \(\set{\sau{2n-1},\sbu{2n-1}}\subseteq\CHq\) according to \rlem{L1319}.
 Because of \rdefnp{D1159}{D1159.a} and \rrem{L1034}, thus \(\set{ (\sau{j})_{j=0}^{2n-1}, (\sbu{j})_{j=0}^{2n-1}}\subseteq\Hgequ{2n-1}\) if \(\seqs{2n}\in\Fgqu{2n}\). 
 Using \eqref{TAB} and \rpropss{P0906}{P1033}, we conclude \(\set{ (\sau{j})_{j=0}^{2n}, (\sbu{j})_{j=0}^{2n}}\subseteq \Hggqu{2n}\) (resp.\ \(\set{ (\sau{j})_{j=0}^{2n}, (\sbu{j})_{j=0}^{2n}}\subseteq \Hgqu{2n}\)).
 Hence, \rdefnp{D1159}{D1159.a} implies \(\seqs{2n+1} \in\Fggqu{2n+1}\) (resp.\ \(\seqs{2n+1} \in\Fgqu{2n+1}\)).
 
 Finally, we consider now the case that there is a positive integer \(n\) such that \(m=2n+1\).
 From \(\umg{2n+1}\le\su{2n+2}\le \omg{2n+1}\) (resp.\ \(\umg{2n+1}<\su{2n+2}<\omg{2n+1}\)), \rdefn{D1719} and \rnota{N5.1} we get then
\begin{align}\label{NTW}
 \Trip{n+1}&\le\su{2n+2},&
 \Tripc{n}&\le \scu{2n}&
 \text{(resp.\ }\Trip{n+1}&<\su{2n+2},&
 \Tripc{n}&<\scu{2n}\text{)}.
\end{align}
 From \(\seqs{2n+1} \in\Fggqu{2n+1}\) (resp.\ \(\seqs{2n+1} \in\Fgqu{2n+1}\)) and \rprop{S8.1.23.} we get  \(\seqs{2n+1} \in \Hggequ{2n+1}\) (resp.\ \(\seqs{2n+1} \in \Hgequ{2n+1}\)).
 In view of \rpropp{L8.1.24.}{L8.1.24.b} and \rprop{S8.1.23.}, we have \((\scu{j})_{j=0}^{2n-1}\in \Fggqu{2n-1}\subseteq\Hggequ{2n-1}\).
 For \(\seqs{2n+1}\in\Fgqu{2n+1}\), we obtain \(\seqs{2n}\in\Fgqu{2n}\) by virtue of \rpropp{M8.1.21.}{M8.1.21.b} and \(\scu{2n-1}\in\CHq\) according to \rlem{L1319} and \rnota{N5.1}.
 Because of \rdefnp{D1159}{D1159.b} and \rrem{L1034}, thus \(\seqsc{2n-1}\in\Hgequ{2n-1}\) if \(\seqs{2n+1}\in\Fgqu{2n+1}\).
 Taking into account \eqref{NTW} and \rpropss{P0906}{P1033}, we infer then \(\seqs{2n+2} \in \Hggqu{2n+2}\) and  \((\scu{j})_{j=0}^{2n}\in\Hggqu{2n}\) (resp.\ \(\seqs{2n+2} \in \Hgqu{2n+2}\) and  \((\scu{j})_{j=0}^{2n}\in\Hgqu{2n}\)).
 Because of \rdefnp{D1159}{D1159.b}, we conclude \( \seqs{2n+2} \in \Fggqu{2n+2}\) (resp.\ \( \seqs{2n+2} \in \Fgqu{2n+2}\)).
\end{proof}

\bthml{P1023}
 Let \(m\in\NO\) and let \(\seqs{m}\) be a sequence from \(\Cqq\). Then:
 \benui
  \il{P1023.a} If \(\seqs{m}\in\Fggqu{m}\), then \([\umg{m},\omg{m}]\neq\emptyset\) and
 \[
  \setaa*{\su{m+1}\in\Cqq}{\seqs{m+1}\in\Fggqu{m+1}}
  =[\umg{m},\omg{m}].
 \]
  \il{P1023.b} If \(\seqs{m}\in\Fgqu{m}\), then \((\umg{m},\omg{m})\neq\emptyset\) and
 \[
  \setaa*{\su{m+1}\in\Cqq}{\seqs{m+1}\in\Fgqu{m+1}}
  =(\umg{m},\omg{m}).
 \]
 \eenui
\ethm
\bproof
 Combine \rlem{MTM} and \rprop{M8.1.37-1.}.
\eproof

 A closer look at the proof of \rthm{P1023} shows that we now immediately obtain purely algebraic proofs of some statements formulated in \rsec{S1237}, which were proved there using results on moment problems. This concerns \rthm{LHP}, \rprop{P1750}, \rcor{C1754}, \rprop{P0710}, and \rthm{T0718}.

\bthml{T1537}
 Let \(\ug\in\R\),  \(\og\in(\ug,\infp)\), and \(m\in\NO\).
 Then \(\Fgequ{m}=\Fgqu{m}\).
\ethm
\begin{proof}
 Combine \rcor{C1342} and \rthmp{P1023}{P1023.b}.
\end{proof}

\bpropl{L1702}
 Let \(\ug\in\R\), let \(\og\in(\ug,\infp)\), let \(m\in\NO\), let \(\seqs{m}\in\Fggqu{m}\), and let \(\su{m+1}\in\set{\umg{m},\omg{m}}\).
 Then \(\seqs{m+1}\in\Fggdqu{m+1}\).
\eprop
\bproof
 From \rthmp{P1023}{P1023.a} and \rremp{R1503}{R1503.a} we obtain \(\seqs{m+1}\in\Fggqu{m+1}\).
 In view of \rdefn{D1730}, we have furthermore \(\usc{m+1}=\Oqq\) or \(\osc{m+1}=\Oqq\).
 In particular, \(\ran{\usc{m+1}}\cap\ran{\osc{m+1}}=\set{\Ouu{q}{1}}\). 
 Thus, \rlem{L0901} yields \(\seqs{m+1}\in\Fggdqu{m+1}\).
\eproof
 
 In view of \rpropsss{L1610}{L1702}{L0942}, the following notions seem to be natural:
\bdefnl{D0842}
 Let \(\ug\in\R\), let \(\og\in(\ug,\infp)\), let \(m\in\NO\), and let \(\seqs{m}\in\Fggqu{m}\).
 Let the sequence \((\su{j})_{j=m+1}^\infi\) be recursively defined by \(\su{j}\defeq\umg{j-1}\) (resp.\ \(\su{j}\defeq\omg{j-1}\)).
 Then \(\seqsinf\) is called the \emph{lower} (resp.\ \emph{upper}) \noti{\habHdso{\(\seqs{m}\)}}{h@\habHdlso{\(\seqs{m}\)}}\index{h@\habHduso{\(\seqs{m}\)}}.
\edefn

\bpropl{L1726}
 Let \(\ug\in\R\), let \(\og\in(\ug,\infp)\), let \(m\in\NO\), and let \(\seqs{m}\in\Fggqu{m}\). Then the lower (resp.\ upper) \tabHdso{\(\seqs{m}\)} belongs to \(\Fggqinf\) and is \tabHdo{m+1} as well as \tabHco{m+2}.
\eprop
\bproof
 Denote by \(\seqsinf\) the lower (resp.\ upper) \tabHdso{\(\seqs{m}\)}.
 In view of \rdefn{D0846} and \rremp{R1503}{R1503.a}, we conclude \(\seqsinf\in\Fggqinf\) by successive application of \rthmp{P1023}{P1023.a}.
 According to \rdefn{D0842} and \rprop{L1702}, furthermore \(\seqs{m+1}\in\Fggdqu{m+1}\). In view of \rdefnp{D0933}{D0933.a}, thus \(\seqsinf\) is \tabHdo{m+1}.
 By virtue of \rrem{L1801}, then \(\seqsinf\) is \tabHco{m+2}.
\eproof

 From \rprop{L1726} and \rlem{L1230} we obtain the following result:
\breml{R0738}
 Let \(\ug\in\R\), let \(\og\in(\ug,\infp)\), let \(m\in\NO\), and let \(\seqs{m}\in\Fggqu{m}\).
 Then the lower as well as the \tabHduso{\(\seqs{m}\)} is \tHdo{n+1}, where \(n\) is the unique number from \(\NO\) with \(m=2n-1\) or \(m=2n\).
\erem
 
\bleml{L0748}
 Let \(\ug\in\R\), let \(\og\in(\ug,\infp)\), let \(n\in\N\), and let \(\seqs{2n-1}\in\Fggqu{2n-1}\).
 Then the \tabHdlso{\(\seqs{2n-1}\)} is \tHdo{n}.
\elem
\bproof
 Denote by \(\seqsinf\) the \tabHdlso{\(\seqs{2n-1}\)}.
 According to \rprop{L1726}, then \(\seqsinf\in\Fggqinf\).
 In particular, \(\seqsinf\in\Hggqinf\) by virtue of \rlem{L0923}.
 In view of \rdefn{D0842}, we have furthermore $\su{2n}=\umg{2n-1}$.
 Thus, \(\su{2n}=\Trip{n}\) according to \rdefn{D1719}.
 From \rnotap{N41}{N41.c} we see then \(\Lu{n}=\Oqq\).
\eproof

\bdefnl{D1527}
 Let \(\ug\in\R\), let \(\og\in(\ug,\infp)\), let \(m\in\NO\), and let \(\seqs{m}\in\Fggqu{m}$.
 Let the sequence \((\su{j})_{j=m+1}^\infi\) be recursively defined by \(\su{j}\defeq\cen{j-1}\), where \(\cen{j-1}\) is given by \rdefn{D1516}.
 Then \(\seqsinf\) is called the \notii{\habHcso{\(\seqs{m}\)}}.
\edefn

\bpropl{R1532}
 Let \(\ug\in\R\), let \(\og\in(\ug,\infp)\), let \(m\in\NO\), and let \(\seqs{m}\in\Fggqu{m}$.
 Then the \tabHcso{\(\seqs{m}\)} is \tabHnnd{} and \tabHco{m+1}.
\eprop
\bproof
 Denote by \(\seqsinf\) the \tabHcso{\(\seqs{m}\)}.
 In view of \rdefn{D0846} and \rremp{R1503}{R1503.a}, we conclude \(\seqsinf\in\Fggqinf\) by successive application of \rthmp{P1023}{P1023.a}.
 By virtue of \rdefnss{D1527}{D1509}, then \(\seqsinf\) is \tabHco{m+1}.
\eproof

\bpropl{R1534}
 Let \(\ug\in\R\), let \(\og\in(\ug,\infp)\), let \(m\in\NO\), and let \(\seqs{m}\in\Fgqu{m}$.
 Then the \tabHcso{\(\seqs{m}\)} is \tabHpd{}.
\eprop
\bproof
 In view of \rdefn{D0846} and \rremp{R1503}{R1503.b}, we conclude by successive application of \rthmp{P1023}{P1023.b}, that the \tabHcso{\(\seqs{m}\)} belongs to \(\Fgqinf\).
\eproof

\bpropl{R1839}
 Let \(\ug\in\R\), let \(\og\in(\ug,\infp)\), let \(m\in\NO \), and let \(\seqs{m}\in\Fgqu{m}\).
 Then there exists a sequence \((s_j)_{j=m+1}^\infi\) of complex \tqqa{matrices} such that \(\seqsinf\in\Fgqinf\).
\eprop
\begin{proof}
 Use \rprop{R1534}.
\end{proof}

\section{Applications to the moment problem~\mproblem{\ab}{m}{=}}\label{S0733}
 In this section, we discuss first applications of the preceding investigations on the structure of matricial \tabHnnd{} sequences.
 The following statement concretizes \rprop{P0756}.

\begin{thm}\label{MMH}
 Let \(\ug\in\R\), let \(\og\in(\ug,\infp)\), let \(m\in\NO \), and let \(\seqs{m}\in\Fggqu{m}\).
 Then \(\nextmom{m}\) given by \eqref{SGG} admits the representation 
\(
 \nextmom{m}%
 =[\umg{m},\omg{m}]%
\)
 where \(\umg{m}\) and \(\omg{m}\) are given in \rdefn{D1719}.
\end{thm}
\begin{proof}
 Combine \rprop{P0756} and \rthmp{P1023}{P1023.a}.
\end{proof}

\begin{thm}\label{LHPUN}
 Let \(\ug\in\R\), let \(\og\in(\ug,\infp)\), and let \(\seqsinf \) be a sequence from \(\Cqq\).
 Then \(\MggqFsinf \ne \emptyset\) if and only if \(\seqsinf  \in\Fggqinf\). In this case, the set \(\MggqFsinf\) contains exactly one element.
\end{thm}
\begin{proof}
 In view of a matricial version of the Helly-Prohorov Theorem (see, \eg{}~\cite[Satz~9, Bemerkung~2]{MR975253}), it is readily checked that \(\MggqFs{\infi}\neq\emptyset\) if and only if  \(\MggqFs{m}\neq\emptyset\) for all \(m\in\NO \).
 (The essential idea of this argumentation is originated in~\zitaa{MR0184042}{proof of \ctheo{2.1.1}}).
 In view of \rdefn{D0846}, then the asserted equivalence follows using \rthm{LHP}.
 Since the interval \(\ab\) is bounded, it is a well-known fact that the set \(\MggqFs{\infi}\) contains at most one element.
 This completes the proof.
\end{proof}

 Now we turn our attention to the special molecular solutions of truncated matricial \tabH{} moment problems, which were constructed in \rsec{S1626} (see \rthm{T1019} and \rprop{P1050}).

\bpropl{P1809}
 Let \(\ug\in\R\), let \(\og\in(\ug,\infp)\), let \(m\in\NO\), and let \(\seqs{m}\in\Fggqu{m}\). Then:
 \benui
  \il{P1809.a} The lower (resp.\ upper) \tabHd{} sequence \(\seqsl\) (resp.\ \(\seqsu\)) associated to \(\seqs{m}\) belongs to \(\Fggdqinf\).
	\il{P1809.b} The set \(\Mggoaag{q}{\ab}{\seqsl}\) contains exactly one element \(\cdablm{m}\) and the set \(\Mggoaag{q}{\ab}{\seqsu}\) contains exactly one element \(\cdabum{m}\).
 \eenui
\eprop
\bproof
 \eqref{P1809.a} Use \rprop{L1726}.

 \eqref{P1809.b} In view of~\eqref{P1809.a}, this follows from \rthm{LHPUN}.
\eproof

\bdefnl{D1824}
 Let \(\ug\in\R\), let \(\og\in(\ug,\infp)\), let \(m\in\NO\), and let \(\seqs{m}\in\Fggqu{m}\). Then the \tnnH{} \tqqa{measure} \(\cdablm{m}\) (resp.\ \(\cdabum{m}\)) is called the \emph{lower} (resp.\ \emph{upper}) \notii{\tCDabmo{\(\seqs{m}\)}}. 
\edefn

\bpropl{L1221}
 Let \(\ug\in\R\), let \(\og\in(\ug,\infp)\), let \(n\in\N\), and let \(\seqs{2n-1}\in\Fggqu{2n-1}\). Then:
\benui
 \il{L1221.a} Denote by \(\cdablm{2n-1}\) the \tCDablmo{\(\seqs{2n-1}\)}. Then \(\seqs{2n-1}\in\Hggequ{2n-1}\) and \(\cdablm{2n-1}\) is the restriction onto \(\BorF\) of the \tCDm{} \(\cdm{n}\) associated with \(\seqs{2n-1}\).
 \il{L1221.b} Denote by \(\cdabum{2n-1}\) the \tCDabumo{\(\seqs{2n-1}\)}. Let \(\su{2n}\defeq\omg{2n-1}\) and let \(\su{2n+1}\defeq\omg{2n}\). Then \(\seqs{2n+1}\in\Hggequ{2n+1}\) and \(\cdabum{2n-1}\) is the restriction onto \(\BorF\) of the \tCDm{} \(\cdm{n+1}\) associated with \(\seqs{2n+1}\).
 \eenui
\eprop
\bproof
 \eqref{L1221.a} In view of \rlem{FR5.3.}, we have \(\seqs{2n-1}\in\Hggequ{2n-1}\).
 Denote by \(\seqsinf\) the \tabHdlso{\(\seqs{2n-1}\)}.
 According to \rlem{L0748}, then \(\seqsinf\) is \tHdo{n}.
 By virtue of \rthm{T1002}, thus \(\seq{\su{j}}{j}{2n}{\infi}\) coincides with the unique sequence from \rthmp{T1002}{T1002.a} and therefore \(\MggqRsg{\infi}=\set{\cdm{n}}\).
 By definition, \(\cdablm{2n-1}\) belongs to \(\MggqFs{\infi}\).
 Then, \(\mu\colon\BorR\to\Cggq\) defined by \(\mu(B)\defeq\cdablm{2n-1}\rk{B\cap\ab}\) belongs to \(\MggqRsg{\infi}\).
 Consequently, \(\mu=\cdm{n}\).
 Since the restriction of \(\mu\) onto \(\BorF\) is \(\cdablm{2n-1}\), the proof of \rpart{L1221.a} is complete.

 \eqref{L1221.b} Applying \rprop{L1702} twice, we see that \(\seqs{2n+1}\) belongs to \(\Fggdqu{2n+1}\).
 Denote by \(\seqsinf\) the \tabHduso{\(\seqs{2n+1}\)}.
 In view of \rdefn{D0842}, the sequence \(\seqsinf\) coincides with the \tabHduso{\(\seqs{2n-1}\)}.
 Hence, \(\set{\cdabum{2n-1}}=\MggqFs{\infi}\) by definition.
 In view of \(\seqs{2n+1}\in\Fggdqu{2n+1}\), \rprop{L0942} yields \(\su{j}=\umg{j-1}\) for all \(j\in\minf{2n+2}\).
 According to \rdefn{D0842}, then \(\seqs{\infi}\) is the \tabHdlso{\(\seqs{2n+1}\)}.
 By definition, then \(\set{\cdablm{2n+1}}=\MggqFs{\infi}\).
 Hence, \(\cdabum{2n-1}=\cdablm{2n+1}\).
 The application of \rpart{L1221.a} to the sequence \(\seqs{2n+1}\) completes the proof.
\eproof

\bleml{L0725}
 Let \(\ug\in\R\), let \(\og\in(\ug,\infp)\), let \(n\in\NO\), and let \(\seqs{2n}\in\Fggqu{2n}\). Then:
 \benui
  \il{L0725.a} Denote by \(\cdablm{2n}\) the \tCDablmo{\(\seqs{2n}\)} and let \(\su{2n+1}\defeq\umg{2n}\). Then \(\seqs{2n+1}\in\Hggequ{2n+1}\) and \(\cdablm{2n}\) is the restriction onto \(\BorF\) of the \tCDm{} \(\cdm{n+1}\) associated with \(\seqs{2n+1}\).
  \il{L0725.b} Denote by \(\cdabum{2n}\) the \tCDabumo{\(\seqs{2n}\)} and let \(\su{2n+1}\defeq\omg{2n}\). Then \(\seqs{2n+1}\in\Hggequ{2n+1}\) and \(\cdabum{2n}\) is the restriction onto \(\BorF\) of the \tCDm{} \(\cdm{n+1}\) associated with \(\seqs{2n+1}\).
 \eenui
\elem
\bproof
 \eqref{L0725.a} From \rprop{L1702} we see that \(\seqs{2n+1}\) belongs to \(\Fggqu{2n+1}\).
 In view of \rdefn{D0842}, the lower \tabHd{} sequence \(\seqsinf\) associated with \(\seqs{2n}\) coincides with the \tabHdlso{\(\seqs{2n+1}\)}.
 Hence, \(\cdablm{2n}\) is the lower \tCDm{} \(\cdablm{2n+1}\) associated with \(\seqs{2n+1}\) and \(\ab\).
 The application of \rpropp{L1221}{L1221.a} to the sequence \(\seqs{2n+1}\) completes the proof of~\eqref{L0725.a}.
 
 \eqref{L0725.b} From \rprop{L1702} we see that \(\seqs{2n+1}\) belongs to \(\Fggdqu{2n+1}\).
 Denote by \(\seqsinf\) the \tabHduso{\(\seqs{2n+1}\)}.
 In view of \rdefn{D0842}, then \(\seqsinf\) is the \tabHduso{\(\seqs{2n}\)}.
 Hence, \(\set{\cdabum{2n}}=\MggqFs{\infi}\) by definition.
 In view of \(\seqs{2n+1}\in\Fggdqu{2n+1}\), \rprop{L0942} yields \(\su{j}=\umg{j-1}\) for all \(j\in\minf{2n+2}\).
 According to \rdefn{D0842}, then \(\seqs{\infi}\) is the \tabHdlso{\(\seqs{2n+1}\)}.
 By definition, then \(\set{\cdablm{2n+1}}=\MggqFs{\infi}\).
 Hence, \(\cdabum{2n}=\cdablm{2n+1}\).
 The application of \rpropp{L1221}{L1221.a} to the sequence \(\seqs{2n+1}\) completes the proof of \rpart{L0725.b}.
\eproof

\bcorl{C1800}
 Let \(\ug\in\R\), let \(\og\in(\ug,\infp)\), let \(m\in\NO\), and let \(\seqs{m}\in\Fggqu{m}\). Then \(\set{\cdablm{m},\cdabum{m}}\subseteq\MggqFs{m}\cap\MggqmolF\)
\ecor
\bproof
 Combine \rpropss{L1221}{L0725} with \rthmp{T1002}{T1002.c}.
\eproof

 Now we want to add some comments concerning the case \(q=1\).
 Let \(\seqs{m}\in\Fgguuuu{1}{m}{\ug}{\og}\).
 Then it can be expected that the lower (resp.\ upper) \tCDm{} \(\cdablm{m}\) (resp.\ \(\cdabum{m}\)) associated with \(\seqs{m}\) and \(\ab\) coincides with the lower (resp.\ upper) principal solution constructed by M.~G.~Krein in~\zita{MR0044591} in~\zitaa{MR0458081}{\cchap{III}}.
 Furthermore, the other molecular solutions obtained in \rprop{P1050} should be canonical solutions in the terminology of M.~G.~Krein.
 The case of a sequence \(\seqs{m}\in\Fgqu{m}\) deserves particular attention.
 In this case, the theory of orthogonal matrix polynomials can be effectively applied
 and promises to produce a whole collection of useful explicit formulas.
 This will be confirmed by the strategy used by M.~G.~Krein~\zita{MR0044591} (see also~\zitaa{MR0458081}{\cchap{III}}).
 It should be mentioned that first steps in this direction are already contained in the papers~\zitas{MR3014200,MR3079837} by A.~E.~Choque~Rivero, where the resolvent matrices for the non-degenerate truncated \tabH{} moment problem were expressed in terms of orthogonal matrix polynomials and moreover multiplicative decompositions of these resolvent matrices were derived.
 We will handle this theme in separate work.

\appendix
\section{Some facts from matrix theory}\label{A1046}

\begin{rem}\label{A.20.}
 Let  \(n\in\N\) and let \((A_j)_{j=0}^n\)  be a sequence of \tnnH{} complex \tqqa{matrices}.
 Then \(\ran{\sum_{j=1}^n A_{j}) } = \sum_{j=1}^n\cR (A_{j})\) and \(\nul{\sum_{j=1}^n A_{j}} =\bigcap_{j=1}^n\cN (A_{j})\).
\end{rem}

\begin{rem}\label{R0705}
 If \(A\in\Cggq\) and if \(X\in\Cqp\), then \(X^\ad AX\in\Cggp\). If \(A\in\Cgq\) and if \(X\in\Cqp\) with \(\rank X=p\), then \(X^\ad AX\in\Cgp\).
\end{rem}

 Denote by \symba{\OPu{\mathcal{U}}}{p} the matrix associated with the orthogonal projection in the Euclidean space \(\Cq\) onto a linear subspace \(\mathcal{U}\).

\breml{R0757}
 Let \(\mathcal{U}\) be a linear subspace of \(\Cq\), then \(\OPu{\mathcal{U}}\in\CHq\) and \(\Oqq\lleq\OPu{\mathcal{U}}\lleq\Iq\).
\erem

 For the convenience of the reader, we state some well-known and some special results on Moore-Penrose inverses of matrices (see \eg{}, Rao/Mitra~\zita{MR0338013} or~\zita{MR1152328}{\csec{1}}).
 If \(A\in\Cpq\), then (by definition) the Moore-Penrose inverse \(A^\mpi\) of \(A\) is the unique matrix \(A^\mpi\in\Cqp\)\index{\(A^\mpi\)} which satisfies the four equations
\begin{align*}
 AA^\mpi A&=A,&A^\mpi AA^\mpi&=A^\mpi,&(AA^\mpi)^\ad&=AA^\mpi,&
 &\text{and}&
 (A^\mpi A)^\ad&=A^\mpi A.
\end{align*}

\begin{prop}[see, \eg{}~\zita{MR1152328}{\cthm{1.1.1}}]\label{PA*2}
 If \(A\in\Cpq\) then a matrix \(X\in\Cqp\) is the Moore-Penrose inverse of \(A\) if and only if \(AX=\OPu{\ran{A}}\) and \(XA=\OPu{\ran{X}}\).
\end{prop}

\begin{rem}\label{MA.7.8.}
 Let  \(A\in\Cpq\) and \(B\in\Coo{p}{r}\).
 Then \(\cR (B) \subseteq \cR (A)\) if and only if \(AA^\mpi B = B\).
\end{rem}

\begin{rem}\label{A.7.8-1.}
 Let  \(A\in\Cpq\) and \(B\in\Coo{n}{q}\).
 Then \(\cN (A) \subseteq \cN (B)\) if and only if \(BA^\mpi A = B\).
\end{rem}

\begin{rem}\label{BN3.}
 \((A^\ad )^\mpi  = (A^\mpi )^\ad \) for each \(A\in \Cpq\).
\end{rem}

\begin{rem}\label{R0746}
 Let \(A\in\CHq\). Then \(A^\mpi\in\CHq\) and \(A^\mpi A=AA^\mpi\).
\end{rem}

\begin{rem}\label{MA.7.3.}
 Let  \(A\in\Cpq\).
 If \(U\in\Coo{n}{p}\) fulfills \(U^\ad  U = \Ip \) and if \(V\in\Coo{q}{r} \) is such that \(VV^\ad =\Iq \), then \((UAV)^\mpi  = V^\ad  A^\mpi  U^\ad \).
\end{rem}

\begin{rem}\label{R1546}
 Let \(A\in\Cggq\). Then \(A^\mpi\in\Cggq\) and \(\sqrt{A^\mpi}=\sqrt{A}^\mpi\). Furthermore, \(\ran{\sqrt{A}}=\ran{A}\) and \(\nul{\sqrt{A}}=\nul{A}\).
\end{rem}

\begin{rem}\label{R0736}
 Let \(A,B\in\CHq\) with \(\Oqq\lleq A\lleq B\). Then \(\ran{A}\subseteq\ran{B}\) and \(\nul{B}\subseteq\nul{A}\).
\end{rem}

\begin{lem}[\cite{MR0245582,MR0394287}]\label{AEP}
 Let \(M=\tmat{A&B\\C&D}\) be the block representation of a complex \taaa{(p+q)}{(p+q)}{matrix} \(M\) with \tppa{block} \(A\).
 Then:
 \benui
  \il{AEP.a} \(M\) is \tnnH{} if and only if \(A\) and \(D-CA^\mpi B\) are both \tnnH{}, \(\ran{B}\subseteq\ran{A}\), and \(C=B^\ad\).
  \il{AEP.b} \(M\) is \tpH{} if and only if \(A\) and \(D-CA^\mpi B\) are both \tpH{} and \(C=B^\ad\).
 \eenui
\end{lem}

\section{Parallel sum of matrices}\label{A1542}
 For every choice of complex \tpqa{matrices} \(A\) and \(B\), the \notii{parallel sum} \(A\ps  B\) of \(A\) and \(B\) is defined by \index{\(A\ps  B \defeq  A(A+B)^\mpi  B\)}
\beql{ps}
 A\ps  B
 \defeq  A(A+B)^\mpi  B.
\eeq
 Furthermore, let \symba{\PSpq }{p}  be the set of all pairs \((A,B)\in\Cpq \times \Cpq \) such that \(\cR (A)\subseteq \cR (A+B)\) and \(\cN (A+B) \subseteq \cN(A)\) hold true.
 
 Let \(A\) and \(B\) be non-singular matrices from \(\Cqq\) such that \(\det\rk{A+B}\neq0\) and \(\det\rk{A^\inv+B^\inv}\neq0\).
 Then \((A,B)\in\PSqq\) and \(A\ps B=\rk{A^\inv+B^\inv}^\inv\).
 
\breml{L0751} 
 If \(A,B\in\Cgq\), then \((A,B)\in\PSqq\) and \(A\ps  B\in\Cgq \).
\erem

\begin{lem}[{\cite[\clem{4}]{MR0242573}}]\label{MA.8.7*.}
 If \(A,B\in\Cggq\), then \((A,B)\in\PSqq\) and \(A\ps  B\in\Cggq \).
\end{lem}

\begin{lem}[{\cite[Theorem~2.2(a)]{MR0325642}}]\label{MA.8.6.} 
 If   \((A,B)\in \PSpq \), then \((B,A)\in\PSpq \) and \(A\ps  B = B \ps  A\).
\end{lem}

\begin{lem}[{\cite[Theorem~2.2(f)]{MR0325642}}]\label{MA.8.11.} 
 If   \((A,B)\in \PSpq \), then \(\cR (A\ps  B) = \cR (A)\cap \cR (B)\) and \(\cN (A \ps  B) = \cN (A) + \cN (B)\).
\end{lem}

\begin{prop}[{\cite[Theorem~2.2(g)]{MR0325642}}]\label{MA.8.12.} 
 If   \((A,B)\in \PSqq \), then \((A\ps   B)^\mpi  = \OPu{\cR (A^\ad )\cap\cR (B^\ad )} (A^\mpi  + B^\mpi ) \OPu{\cR (A)\cap \cR (B)}\).
\end{prop}

\bleml{MA.8.11}
 If \((A,B)\in\PSpq\), then
 \(
  (A+B)-4(A\ps B)
  =(A-B)(A+B)^\mpi(A-B)
 \).
\elem
\bproof
 Let \((A,B)\in\PSpq\) and let \(C\defeq A+B\).
 In view of \rlem{MA.8.6.}, we get
 \[
  (A\pm B)C^\mpi(A\pm B)
  =AC^\mpi A\pm AC^\mpi B\pm BC^\mpi A+BC^\mpi B
  =AC^\mpi A+BC^\mpi B\pm2(A\ps B).
 \]
 Hence,
 \[
    (A+B)-(A-B)(A+B)^\mpi(A-B)
    =(A+B)C^\mpi(A+B)-(A-B)C^\mpi(A-B)
    =4(A\ps B).\qedhere
 \]
\eproof

\bibliography{165arxiv}

\def\cprime{$'$}
\begin{thebibliography}{10}

\bibitem{MR0184042}
N.~I. Akhiezer.
\newblock {\em The classical moment problem and some related questions in
  analysis}.
\newblock Translated by N. Kemmer. Hafner Publishing Co., New York, 1965.

\bibitem{MR0245582}
A.~Albert.
\newblock Conditions for positive and nonnegative definiteness in terms of
  pseudoinverses.
\newblock {\em SIAM J. Appl. Math.}, 17:434--440, 1969.

\bibitem{MR0242573}
W.~N. Anderson, Jr. and R.~J. Duffin.
\newblock Series and parallel addition of matrices.
\newblock {\em J. Math. Anal. Appl.}, 26:576--594, 1969.

\bibitem{MR1511425}
C.~Carath\'eodory.
\newblock \"uber den {V}ariabilit\"atsbereich der {K}oeffizienten von
  {P}otenzreihen, die gegebene {W}erte nicht annehmen.
\newblock {\em Math. Ann.}, 64(1):95--115, 1907.

\bibitem{zbMATH02629876}
C.~{Carath\'eodory}.
\newblock {\"Uber den Variabilit\"atsbereich der {\it Fourier}schen Konstanten
  von positiven harmonischen Funktionen.}
\newblock {\em {Rend. Circ. Mat. Palermo}}, 32:193--217, 1911.

\bibitem{MR3014200}
A.~E. Choque~Rivero.
\newblock Multiplicative structure of the resolvent matrix for the truncated
  {H}ausdorff matrix moment problem.
\newblock In {\em Interpolation, {S}chur functions and moment problems. {II}},
  volume 226 of {\em Oper. Theory Adv. Appl.}, pages 193--210.
  Birkh\"auser/Springer Basel AG, Basel, 2012.

\bibitem{MR3079837}
A.~E. Choque~Rivero.
\newblock The resolvent matrix for the {H}ausdorff matrix moment problem
  expressed in terms of orthogonal matrix polynomials.
\newblock {\em Complex Anal. Oper. Theory}, 7(4):927--944, 2013.

\bibitem{MR2222521}
A.~E. Choque~Rivero, {\relax Yu}.~M. Dyukarev, B.~Fritzsche, and B.~Kirstein.
\newblock A truncated matricial moment problem on a finite interval.
\newblock In {\em Interpolation, {S}chur functions and moment problems}, volume
  165 of {\em Oper. Theory Adv. Appl.}, pages 121--173. Birkh\"auser, Basel,
  2006.

\bibitem{MR2342899}
A.~E. Choque~Rivero, {\relax Yu}.~M. Dyukarev, B.~Fritzsche, and B.~Kirstein.
\newblock A truncated matricial moment problem on a finite interval. {T}he case
  of an odd number of prescribed moments.
\newblock In {\em System theory, the {S}chur algorithm and multidimensional
  analysis}, volume 176 of {\em Oper. Theory Adv. Appl.}, pages 99--164.
  Birkh\"auser, Basel, 2007.

\bibitem{MR1468473}
H.~Dette and W.~J. Studden.
\newblock {\em The theory of canonical moments with applications in statistics,
  probability, and analysis}.
\newblock Wiley Series in Probability and Statistics: Applied Probability and
  Statistics. John Wiley \& Sons, Inc., New York, 1997.
\newblock A Wiley-Interscience Publication.

\bibitem{MR1883272}
H.~Dette and W.~J. Studden.
\newblock Matrix measures, moment spaces and {F}avard's theorem for the
  interval {$[0,1]$} and {$[0,\infty)$}.
\newblock {\em Linear Algebra Appl.}, 345:169--193, 2002.

\bibitem{MR1152328}
V.~K. Dubovoj, B.~Fritzsche, and B.~Kirstein.
\newblock {\em Matricial version of the classical {S}chur problem}, volume 129
  of {\em Teubner-Texte zur Mathematik [Teubner Texts in Mathematics]}.
\newblock B. G. Teubner Verlagsgesellschaft mbH, Stuttgart, 1992.
\newblock With German, French and Russian summaries.

\bibitem{MR2735313}
{\relax Yu}.~M. Dyukarev, B.~Fritzsche, B.~Kirstein, and C.~M{\"a}dler.
\newblock On truncated matricial {S}tieltjes type moment problems.
\newblock {\em Complex Anal. Oper. Theory}, 4(4):905--951, 2010.

\bibitem{MR2570113}
{\relax Yu}.~M. Dyukarev, B.~Fritzsche, B.~Kirstein, C.~M{\"a}dler, and H.~C.
  Thiele.
\newblock On distinguished solutions of truncated matricial {H}amburger moment
  problems.
\newblock {\em Complex Anal. Oper. Theory}, 3(4):759--834, 2009.

\bibitem{MR0394287}
A.~V. Efimov and V.~P. Potapov.
\newblock {$J$}-expanding matrix-valued functions, and their role in the
  analytic theory of electrical circuits.
\newblock {\em Uspehi Mat. Nauk}, 28(1(169)):65--130, 1973.

\bibitem{MR885621}
B.~Fritzsche and B.~Kirstein.
\newblock An extension problem for nonnegative {H}ermitian block {T}oeplitz
  matrices.
\newblock {\em Math. Nachr.}, 130:121--135, 1987.

\bibitem{MR918682}
B.~Fritzsche and B.~Kirstein.
\newblock A {S}chur type matrix extension problem.
\newblock {\em Math. Nachr.}, 134:257--271, 1987.

\bibitem{MR975253}
B.~Fritzsche and B.~Kirstein.
\newblock Schwache {K}onvergenz nichtnegativ hermitescher {B}orelma\ss e.
\newblock {\em Wiss. Z. Karl-Marx-Univ. Leipzig Math.-Natur. Reihe},
  37(4):375--398, 1988.

\bibitem{MR2805417}
B.~Fritzsche, B.~Kirstein, and C.~M{\"a}dler.
\newblock On {H}ankel nonnegative definite sequences, the canonical {H}ankel
  parametrization, and orthogonal matrix polynomials.
\newblock {\em Complex Anal. Oper. Theory}, 5(2):447--511, 2011.

\bibitem{MR3014201}
B.~Fritzsche, B.~Kirstein, and C.~M{\"a}dler.
\newblock On a special parametrization of matricial {$\alpha$}-{S}tieltjes
  one-sided non-negative definite sequences.
\newblock In {\em Interpolation, {S}chur functions and moment problems. {II}},
  volume 226 of {\em Oper. Theory Adv. Appl.}, pages 211--250.
  Birkh\"auser/Springer Basel AG, Basel, 2012.

\bibitem{MR3133464}
B.~Fritzsche, B.~Kirstein, and C.~M{\"a}dler.
\newblock Transformations of matricial {$\alpha$}-{S}tieltjes non-negative
  definite sequences.
\newblock {\em Linear Algebra Appl.}, 439(12):3893--3933, 2013.

\bibitem{MR3014199}
B.~Fritzsche, B.~Kirstein, C.~M{\"a}dler, and T.~Schwarz.
\newblock On a {S}chur-type algorithm for sequences of complex {$p\times
  q$}-matrices and its interrelations with the canonical {H}ankel
  parametrization.
\newblock In {\em Interpolation, {S}chur functions and moment problems. {II}},
  volume 226 of {\em Oper. Theory Adv. Appl.}, pages 117--192.
  Birkh\"auser/Springer Basel AG, Basel, 2012.

\bibitem{MR0059329}
S.~Karlin and L.~S. Shapley.
\newblock Geometry of moment spaces.
\newblock {\em Mem. Amer. Math. Soc.}, No. 12:93, 1953.

\bibitem{MR0204922}
S.~Karlin and W.~J. Studden.
\newblock {\em Tchebycheff systems: {W}ith applications in analysis and
  statistics}.
\newblock Pure and Applied Mathematics, Vol. XV. Interscience Publishers John
  Wiley \& Sons, New York-London-Sydney, 1966.

\bibitem{MR0080280}
I.~S. Kats.
\newblock On {H}ilbert spaces generated by monotone {H}ermitian
  matrix-functions.
\newblock {\em Har\cprime kov Gos. Univ. U\v c. Zap. 34 = Zap. Mat. Otd.
  Fiz.-Mat. Fak. i Har\cprime kov. Mat. Ob\v s\v c. (4)}, 22:95--113 (1951),
  1950.

\bibitem{MR0044591}
M.~G. Kre{\u\i}n.
\newblock The ideas of {P}. {L}. \v {C}eby\v sev and {A}. {A}. {M}arkov in the
  theory of limiting values of integrals and their further development.
\newblock {\em Uspehi Matem. Nauk (N.S.)}, 6(4 (44)):3--120, 1951.

\bibitem{MR0458081}
M.~G. Kre{\u\i}n and A.~A. Nudel{\cprime}man.
\newblock {\em The {M}arkov moment problem and extremal problems}.
\newblock American Mathematical Society, Providence, R.I., 1977.
\newblock Ideas and problems of P. L. {\v{C}}eby{\v{s}}ev and A. A. Markov and
  their further development, Translated from the Russian by D. Louvish,
  Translations of Mathematical Monographs, Vol. 50.

\bibitem{MR0325642}
S.~K. Mitra and M.~L. Puri.
\newblock On parallel sum and difference of matrices.
\newblock {\em J. Math. Anal. Appl.}, 44:92--97, 1973.

\bibitem{MR0338013}
C.~R. Rao and S.~K. Mitra.
\newblock {\em Generalized inverse of matrices and its applications}.
\newblock John Wiley \& Sons, Inc., New York-London-Sydney, 1971.

\bibitem{MR0163346}
M.~Rosenberg.
\newblock The square-integrability of matrix-valued functions with respect to a
  non-negative {H}ermitian measure.
\newblock {\em Duke Math. J.}, 31:291--298, 1964.

\bibitem{MR0228040}
M.~Skibinsky.
\newblock The range of the {$(n+1)$}th moment for distributions on {$[0,\,1]$}.
\newblock {\em J. Appl. Probability}, 4:543--552, 1967.

\bibitem{MR0246351}
M.~Skibinsky.
\newblock Extreme {$n$}th moments for distributions on {$[0,\,1]$} and the
  inverse of a moment space map.
\newblock {\em J. Appl. Probability}, 5:693--701, 1968.

\bibitem{Thi06}
H.~C. Thiele.
\newblock {\em Beitr\"age zu matriziellen {P}otenzmomentenproblemen}.
\newblock Dissertation, Universit{\"a}t Leipzig, Leipzig, May 2006.

\bibitem{zbMATH02629875}
O.~{Toeplitz}.
\newblock {\"Uber die {\it Fourier}sche Entwicklung positiver Funktionen.}
\newblock {\em {Rend. Circ. Mat. Palermo}}, 32:191--192, 1911.

\end{thebibliography}
\bibliographystyle{abbrv}

\vfill\noindent
\begin{minipage}{0.5\textwidth}
 Universit\"at Leipzig\\
Fakult\"at f\"ur Mathematik und Informatik\\
PF~10~09~20\\
D-04009~Leipzig
\end{minipage}
\begin{minipage}{0.49\textwidth}
 \begin{flushright}
  \texttt{
   fritzsche@math.uni-leipzig.de\\
   kirstein@math.uni-leipzig.de\\
   maedler@math.uni-leipzig.de
  } 
 \end{flushright}
\end{minipage}

\end{document}